\documentclass[12pt]{article}
\usepackage{srcltx}
\usepackage{eurosym}
\usepackage{mathtools}
\usepackage{amsmath}
\usepackage{amsfonts}
\usepackage{amssymb}
\usepackage{amsthm}
\usepackage{graphicx}
\usepackage{mathrsfs}
\usepackage{xcolor}
\usepackage{exscale}
\usepackage{latexsym}
\usepackage{authblk}

\usepackage[colorlinks,plainpages=true,pdfpagelabels,hypertexnames=true,colorlinks=true,pdfstartview=FitV,linkcolor=blue,citecolor=red,urlcolor=black]{hyperref}
\PassOptionsToPackage{unicode}{hyperref}
\PassOptionsToPackage{naturalnames}{hyperref}
\usepackage{enumerate}
\usepackage[shortlabels]{enumitem}
\usepackage{bookmark}
\usepackage{wasysym}
\usepackage{esint}
\usepackage[ddmmyyyy]{datetime}
\usepackage[margin=1in]{geometry}
\makeatletter
\g@addto@macro\normalsize{%
	\setlength\abovedisplayskip{4pt}
	\setlength\belowdisplayskip{4pt}
	\setlength\abovedisplayshortskip{4pt}
	\setlength\belowdisplayshortskip{4pt}
}
\numberwithin{equation}{section}
\everymath{\displaystyle}
\usepackage[capitalize,nameinlink]{cleveref}
\crefname{section}{Section}{Sections}
\crefname{subsection}{Subsection}{Subsections}
\crefname{condition}{Condition}{Conditions}
\crefname{hypothesis}{Hypothesis}{Conditions}
\crefname{assumption}{Assumption}{Assumptions}
\crefname{lemma}{Lemma}{Lemmas}
\crefname{definition}{Definition}{Definitions}

\crefformat{equation}{\textup{#2(#1)#3}}
\crefrangeformat{equation}{\textup{#3(#1)#4--#5(#2)#6}}
\crefmultiformat{equation}{\textup{#2(#1)#3}}{ and \textup{#2(#1)#3}}
{, \textup{#2(#1)#3}}{, and \textup{#2(#1)#3}}
\crefrangemultiformat{equation}{\textup{#3(#1)#4--#5(#2)#6}}%
{ and \textup{#3(#1)#4--#5(#2)#6}}{, \textup{#3(#1)#4--#5(#2)#6}}%
{, and \textup{#3(#1)#4--#5(#2)#6}}

\Crefformat{equation}{#2Equation~\textup{(#1)}#3}
\Crefrangeformat{equation}{Equations~\textup{#3(#1)#4--#5(#2)#6}}
\Crefmultiformat{equation}{Equations~\textup{#2(#1)#3}}{ and \textup{#2(#1)#3}}
{, \textup{#2(#1)#3}}{, and \textup{#2(#1)#3}}
\Crefrangemultiformat{equation}{Equations~\textup{#3(#1)#4--#5(#2)#6}}%
{ and \textup{#3(#1)#4--#5(#2)#6}}{, \textup{#3(#1)#4--#5(#2)#6}}%
{, and \textup{#3(#1)#4--#5(#2)#6}}

\crefdefaultlabelformat{#2\textup{#1}#3}
\newtheorem{theorem} {Theorem}[section]
\newtheorem{proposition} [theorem]{Proposition}

\newtheorem{lemma}[theorem]{Lemma}
\newtheorem{corollary}[theorem]{Corollary}

\newtheorem{counter example}[theorem]{Counter Example}
\newtheorem{remark}[theorem] {Remark}

\newtheorem{claim}[theorem] {Claim}
\def\CC{{\rm \kern.24em \vrule width.02em height1.4ex depth-.05ex \kern-.26emC}}

\def\TagOnRight

\def\AA{{it I} \hskip-3pt{\tt A}}

\def\QQ{\rlap {\raise 0.4ex \hbox{$\scriptscriptstyle |$}} {\hskip -0.1em Q}}

\makeatletter
\newcommand{\vo}{\vec{o}\@ifnextchar{^}{\,}{}}
\makeatother
\def\YYint#1#2#3{{\setbox0=\hbox{$#1{#2#3}{\iint}$}
		\vcenter{\hbox{$#2#3$}}\kern-.50\wd0}}
\def\XXint#1#2#3{{\setbox0=\hbox{$#1{#2#3}{\int}$}
		\vcenter{\hbox{$#2#3$}}\kern-.50\wd0}}

\makeatletter
\def\namedlabel#1#2{\begingroup
	\def\@currentlabel{#2}%
	\label{#1}\endgroup
}
\makeatother
\makeatletter
\newcommand{\rmh}[1]{\mathpalette{\raisem@th{#1}}}
\newcommand{\raisem@th}[3]{\hspace*{-1pt}\raisebox{#1}{$#2#3$}}
\makeatother

\newcounter{desccount}

\newcommand{\descref}[2]{\hyperref[#1]{\textnormal{\textcolor{black}{}\textcolor{blue}{ #2}\textcolor{black}{}}}}
\newcommand{\dref}[2]{\hyperref[#1]{\textcolor{black}{(}\textcolor{blue}{\bf #2}\textcolor{black}{)}}}
\newcommand{\be} {\begin{eqnarray}}
	\newcommand{\ee} {\end{eqnarray}}
\newcommand{\Bea} {\begin{eqnarray*}}
	\newcommand{\Eea} {\end{eqnarray*}}
\newcommand{\pa} {\partial}

\newcommand{\al} {\alpha}
\newcommand{\rr}{\rightarrow}

\newcommand{\B} {\beta}
\newcommand{\de} {\delta}

\newcommand{\p}  {\prime}
\newcommand{\e}  {\varepsilon}

\newcommand{\la} {\lambda}
\newcommand{\si} {\sigma}

\newcommand{\f}{\infty}
\newcommand{\R}{\mathbb{R}}

\newcommand{\noi} {\noindent}

\newcommand{\va} {\varphi}

\newcommand{\norm}[1]{\left|\hspace{-0.2mm}\left| #1 \right|\hspace{-0.2mm}\right|}
\newcommand{\abs}[1]{\left| #1\right|}

\newcounter{whitney}
\refstepcounter{whitney}

\newcounter{ineqcounter}
\refstepcounter{ineqcounter}
\makeatletter
\def\ps@pprintTitle{%
	\let\@oddhead\@empty
	\let\@evenhead\@empty
	\def\@oddfoot{}%
	\let\@evenfoot\@oddfoot}
\makeatother
\usepackage[doublespacing]{setspace}
\usepackage[titletoc,toc,page]{appendix}
\makeatletter
\newcommand{\refcheckize}[1]{%
	\expandafter\let\csname @@\string#1\endcsname#1%
	\expandafter\DeclareRobustCommand\csname relax\string#1\endcsname[1]{%
		\csname @@\string#1\endcsname{##1}\wrtusdrf{##1}}%
	\expandafter\let\expandafter#1\csname relax\string#1\endcsname
}
\makeatother
\refcheckize{\cref}
\refcheckize{\Cref}


\makeatletter
\newcommand{\mainsectionstyle}{%
	\renewcommand{\@secnumfont}{\bfseries}
	\renewcommand\section{\@startsection{section}{2}%
		\z@{.5\linespacing\@plus.7\linespacing}{-.5em}%
		{\normalfont\bfseries}}%
}
\makeatother
\usepackage{pgf,tikz}
\usetikzlibrary{arrows}
\usetikzlibrary{decorations.pathreplacing}
\usepackage{xpatch}
\xpatchcmd{\MaketitleBox}{\hrule}{}{}{}% remove first horizontal rule (above abstract)
\xpatchcmd{\MaketitleBox}{\hrule}{}{}{}% remoce second horizonral rule (below keywords)

\date{}
\usepackage{scalerel}
\makeatletter

\title{Viscous approximation of triangular system in 1-d with nonlinear viscosity}

\author[1,a]{Boris Haspot}
\author[2,b]{Animesh Jana}

\affil[a]{\footnotesize	 Universit\'e Paris Dauphine, PSL Research University, CEREMADE (UMR CNRS 7534), Place du Mar\' echal De Lattre De Tassigny 75775 Paris cedex 16, France.}
\affil[b]{\footnotesize Istituto di Matematica Applicata e Tecnologie Informatiche ``Enrico Magenes'' - Consiglio Nazionale delle Ricerche, Via Ferrata, 5/a - 27100 Pavia, Italy.}
\affil[1]{\em \footnotesize	 haspot@ceremade.dauphine.fr}
\affil[2]{\em \footnotesize	animesh.jana@imati.cnr.it}

\allowdisplaybreaks

\begin{document}
	\maketitle
	\begin{abstract}
		We study the vanishing viscosity limit for $2\times2$ triangular system of hyperbolic conservation laws when the viscosity coefficients are non linear. In this article, we assume that the viscosity matrix $B(u)$ is commutating with the convective part $A(u)$. We show the existence of global smooth solution to the parabolic equation satisfying uniform total variation bound in $\e$ provided that the initial data is small in $BV$. This extends the previous result of Bianchini and Bressan [Commun. Pure Appl. Anal. (2002)] which was considering the case $B(u)=I$.
	\end{abstract}
	
	\tableofcontents

\section{Introduction}
It is well known that hyperbolic systems generally generate infinitely many weak solutions, in order to overcome this lack of uniqueness it is relevant to consider processes of vanishing viscous limits in order to select the relevant physical solutions for systems of nonlinear hyperbolic systems.
%The vanishing viscosity limit is an important tool to study the weak solutions to the systems of hyperbolic conservation laws. 
In this article, we are interested in studying the vanishing viscosity limit for the following hyperbolic triangular system in one dimension which reads as follows
\begin{equation}\label{hyperbo}
	\left\{\begin{array}{rl}
		u_{1,t}+(f(u_1))_x&=0,\\
		u_{2,t}+(g(u_1,u_2))_x&=0,
	\end{array}\mbox{ for }x\in\R\mbox{ and }t>0,
\right.
\end{equation}
where $f$ and $g$ are smooth functions. This system appears in various physics and engineering models, including multi-component chromatography, two-phase flow in porous media and sedimentation processes (see \cite{BGJ,KMR,KMR-1} and references therein for more detailed discussion on applications of the system). Here, we consider the flux functions $f$ and $g$ such that the system of equations \eqref{hyperbo} becomes strictly hyperbolic (see \eqref{condhyperbo} below for more precise condition). We study now the following viscosity approximation,
\begin{equation}	\label{viscous-system-2}
	\begin{cases}
	\begin{aligned}
			u_{1,t}+(f(u_1))_x&=\e(\al_1(u_1)u_{1,x})_x,\\
			u_{2,t}+(g(u_1,u_2))_x&=\e\left[(\B(u_1,u_2) u_{1,x})_x+(\al_2(u_1,u_2)u_{2,x})_x\right],
	\end{aligned}
	\end{cases}
\end{equation}
with $\e>0$. The system \eqref{viscous-system-2} is supplemented with the following initial condition,
\begin{equation}
			u(0,\cdot)=\bar{u},
\end{equation}
where $\bar{u}$ will be assumed small in $TV$, the Total Variation.
In this article, we prove the global existence of smooth solution to \eqref{viscous-system-2} and study the vanishing viscosity limit, that is, to show that the limit of $(u_1^\e,u_2^\e)$ as $\e\rr0$ is solution of the system \eqref{hyperbo}. 
We can write \eqref{hyperbo} in the following form
\begin{equation}%\label{eqn-main}
	u_t+A(u)u_x=0,
\end{equation}
where $A(u)$ is determined as follows
\begin{equation}\label{defA}
	A(u)=\begin{pmatrix}
		f^\p(u_1)&0\\
		\frac{\pa g}{\pa u_1}&	\frac{\pa g}{\pa u_2}
	\end{pmatrix}
	\;\;\mbox{and}\;\;u=\begin{pmatrix}
		u_1\\
		u_2
	\end{pmatrix}.
\end{equation}
We can note that the eigenvalues of $A(u)$ are $\lambda_1(u)=f'(u_1)$ and $\lambda_2(u)=\frac{\pa g}{\pa u_2}(u)$. We assume now the following strict hyperbolicity condition:
\begin{equation}\label{condhyperbo}
	\lambda_2(u)-\lambda_1(u)=\frac{\pa g}{\pa u_2}(u)-f^\p(u_1)\geq c>0,
\end{equation}
for $u$ in a neighborhood of a compact $K$ with $\{\bar{u}(x),x\in\R\}$ included in $K$. %In particular, it implies that the formula \eqref{eigenv} remains always well defined in a neighbourhood of the origin.
It is well known that under the condition \eqref{condhyperbo} the system \eqref{hyperbo} admits a unique global solution provided that  the initial data $u(0)=\bar{u}$ is sufficiently small in $BV(\R)$ (see \cite{BB-vv-lim-ann-math,Bressan-book,Glimm}), the uniqueness is defined in a suitable class of solutions which satisfied in particular the {\em Liu condition} (we refer also to \cite{BDL}). To study the vanishing viscosity limit for \eqref{hyperbo} we recall that we assume that $\bar{u}$ is small in $BV(\R)$ so that we enter in the framework of the theory of existence and uniqueness previously mentionned. Next, we would like to write the parabolic system of equations \eqref{viscous-system-2} in the following form
\begin{equation}\label{eqn-main1}
	u_t+A(u)u_x=\e (B(u)u_{x})_x,
\end{equation}
where $B$ is defined as follows
\begin{equation}\label{defB}
	B(u)=\begin{pmatrix}
		\al_1(u_1)&0\\
		\B(u)&\al_2(u)
	\end{pmatrix}.
\end{equation} 
In addition, we assume that the coefficients $\alpha_i$ are smooth functions such that there exists $c_1,M>0$ such that for any $u\in\R^2$ we have for any $i\in\{1,2\}$
\begin{equation}
	0<c_1\leq \alpha_i(u)\leq M<+\infty.
	\label{condialpha}
\end{equation}
In the sequel we deal with a viscosity matrix $B(u)$ which commutate with the convective matrix $A(u)$. Note that to satisfy the condition $AB=BA$, the function $\B(u)$ must satisfy the following condition
\begin{equation}\label{condicommut}
	\B(u)=(\al_1(u_1)-\al_2(u))\gamma(u)\mbox{ where }\gamma(u)=\frac{\frac{\pa g}{\pa u_1}(u)}{f^\p(u_1)-\frac{\pa g}{\pa u_2}(u)}.
\end{equation}
This yields that the matrices $A(u),B(u)$ have same eigenvectors for any $u\in K$.

We can observe that by using the following rescaling of coordinates $s=\frac{t}{\e}$ and $y=\frac{x}{\e}$ we are reduced to study the following system :
\begin{equation}\label{eqn-main}
	\begin{cases}
	\begin{aligned}
			&u_s+A(u)u_y= (B(u)u_{y})_y,\\
			&u(0,\cdot)=\bar{u}(\e\cdot).
			\end{aligned}
	\end{cases}
\end{equation}
We point out here that the total variation of $\bar{u}^\e$  is independent of $\e$, in other term the $BV$ space is invariant by the scaling $x\rightarrow \e x$. We are finally reduced to study the system \eqref{viscous-system-2} with $\e=1$.

The study of vanishing viscosity limit for system of conservation laws can be found in \cite{BB-temple,BB-triangular,BB-vv-lim-ann-math,Chen-Kang-Vas,Chen-Per,GX,Serre-1} including in particular the so called case of compressible Euler equations when we consider the viscous approximation issue of the compressible Navier-Stokes system (see in particular \cite{Chen-Kang-Vas,Chen-Per}, it is important to point out that the study is particularly delicate in this case since the
the viscosity is only partial, in other words $B(u)$ is not invertible). We refer to \cite{Ancona-Bianchini,Spinolo} for the study of vanishing viscosity limit for system of conservation laws with boundaries. Here, we briefly discuss a few of them which are relevant to the setting and techniques of the current article. When the convective matrix $A$ belongs to the Temple class, that is when each characteristic is a straight line, the existence, the uniform $TV$ estimates in $\e$ and the $L^1$ Lipschitz estimate for the system \eqref{eqn-main1} are established in \cite{BB-temple} for $B=I$. We have studied the vanishing viscosity limit for the Temple class hyperbolic system in \cite{HJ} when the matrices $A$ and $B$ share same eigenvectors. Furthermore, Bianchini and Bressan \cite{BB-triangular,BB-vv-lim-ann-math} developed a powerful method for studying the vanishing viscosity limit by decomposing the gradient of solution as a sum of the gradient of viscous travelling wave and by introducing suitable functionals for carefully analyzing  all
interaction terms. In other words Bianchini and Bressan decompose the solution $u_x=\sum_i v_i\tilde{r}_i(u,v_i,\sigma_i)$ into scalar components in a basis $(\tilde{r}_i(u,v_i,\sigma_i))_{\{1\leq i\leq n\}}$ selected by a center manifold technique which is suitable inasmuch as it enables to estimates in particular the $TV$ bound of the viscous travelling wave. It is important to point out that generally  in the purely hyperbolic case without viscosity, we decompose $u_x$ along a basis $\{r_1(u),\cdots,r_n(u)\}$ of eigenvectors of the
matrix $A(u)$, however this choice does not work here due to the fact that gradient of the $i$-viscous travelling waves $U'$ is not necessary collinear to $r_i(U)$.
In particular this aspect is crucial for obtaining  BV estimates when $A$ is not in Temple class (in fact when $A$ is Temple we can {\em a priori} decompose the gradient $u_x$ in the basis $\{r_1(u),\cdots,r_n(u)\}$ provided that $A$ and $B$ are commutating because in this case gradient of the $i$-viscous travelling wave are collinear to $r_i(U)$). In these articles \cite{BB-triangular,BB-vv-lim-ann-math}, the viscosity matrix $B$ is taken to be $I$. Here we study the vanishing viscosity limit for triangular system \eqref{hyperbo} when $A$ and $B$ have same eigenvectors. Now, we would like to highlight the following key points of the current article.
\begin{itemize}
	\item It deals with a class of nonlinear viscosity matrix $B(u)$ when the hyperbolic part $A(u)$ is not Temple class. This extends the result of \cite{BB-triangular} from $B(u)=I$ to the class of matrices which are diagonalizable with strictly positive eigenvalues and satisfying the relation $AB=BA$. Since $A(u)$ is not necessarily Temple class the gradient of the viscous travelling waves are not in the direction of eigenvectors. Therefore, as in \cite{BB-triangular,BB-vv-lim-ann-math} we need to decompose $u_x$ in the basis of gradient of travelling waves as follows $u_x=\sum_{i=1}^2 v_i\tilde{r}_i(u,v_i,\sigma_i)$. In particular each $v_i$ satisfies the following equation:
	$$\pa_t v_i+\pa_x(\tilde{\lambda}_i v_i)-\pa_x(\alpha_i \pa_x v_i)=\phi_i,$$
	with $\tilde{\lambda}_i$ well chosen and $\phi_i$ a remainder term.
	Due to the fact that we deal with a variable viscosity matrix $B(u)$, we get new terms in $\phi_i$ the forcing part of the parabolic equations which govern $v_i$. As in \cite{BB-triangular,BB-vv-lim-ann-math}, the proof is decomposed in two part, the first one consists in collecting parabolic estimates on the solution $u$ which allows in particular to prove that $\|u_x(t,\cdot)\|_{L^1}$ remains small until a time $\hat{t}$. This time $\hat{t}$ can be seen as a threshold time which delimits the parabolic behavior of the system and the hyperbolic behavior of the equations. In our case since we deal with variable viscosity coefficients, compared with  \cite{BB-triangular}, we need to develop new parabolic estimates. To do this by a suitable change of variable on $u_x$ we manage to diagonalize the system in a such way that we can extend the analysis of parabolic estimates done in \cite{BB-triangular,BB-vv-lim-ann-math}.
	
	\item There exist new terms of the form $(\al_1-\al_2)(w_{1,xx}v_1-w_1v_{1,xx})$ in $\varphi_2$ which does not appear in the analysis of  \cite{BB-triangular} (here $v_1=u_{1,x}$ and $w_1=u_{1,t}$). It is needless to say that for $\al_1=\al_2$ these terms do not appear. In particular for $B=\al(u) I$ we do not need to deal with such terms. To resolve this issue, we introduce a new variable $z$. We discuss below where we mention how this new variable enable us to estimate this new term in $L^1_tL^1_x$ norm. 
	
\end{itemize}

To prove the global existence of smooth solution and uniform total variation estimate we adapt the method developed by Bianchini and Bressan \cite{BB-temple,BB-triangular,BB-vv-lim-ann-math}. As we have mentioned earlier, we need to decompose $u_x$ in the basis of travelling waves which enables us to diagonalize the system, and to deal only with scalar equations on $v_1$ and $v_2$. In our case $\phi_1=0$ 
%since all the characteristic fields does not necessarily belong to Temple class. In this process, we get diagonalize the system and obtain a parabolic equations of $v_1$ and $v_2$ with forcing terms. The forcing term 
and $\phi_2$ is composed of terms of the following form which exists in the analysis of  \cite{BB-temple,BB-triangular,BB-vv-lim-ann-math}
\begin{equation*}
	v_1v_2,\,v_1v_{2,x},\,v_{1,x}v_2,\, (w_{1,x}v_1-v_{1,x}w_1),\, v_1^2\left(\frac{w_1}{v_1}\right)_x^2\chi_{\left\{\abs{\frac{w_1}{v_1}}\leq \de_1\right\}}\,\mbox{ and }w_1^2\chi_{\left\{\abs{\frac{w_1}{v_1}}\geq \de_1\right\}}.
\end{equation*}
Additionally, the force $\phi_2$ contains a term of the form
\begin{equation*}
	(\al_1-\al_2)(w_{1,xx}v_1-w_1v_{1,xx}),
\end{equation*}
which requires to be estimated in $L^1_{t,x}$ norm.
%For $\al_1\neq \al_2$, this term creates additional difficulty to estimate in the $L^1_tL^1_x$ norm. This is a new term which was not there when $B=I$ (as considered in \cite{BB-triangular,BB-vv-lim-ann-math}). 
In order to handle this new difficulty we introduce a new variable $z_1=\al_1w_{1,x}-\tilde{\la}_1w_1$ which can be seen as {\em an effective flux} for the parabolic equation of $w_1$. %Then we can estimate $v_{1,xx}w_1-w_{1,xx}v_1$ by $z_{1,x}v_1-z_1v_{1,x}$ and $w_{1,x}v_1-w_1v_{1,x}$ by applying the functionals related to shortening curves. 
Thanks to the structure of the system the new variable $z_1$ satisfies a parabolic equation 
\begin{equation*}
	z_{1,t}+(\tilde{\la}_1z_1)_x-(\al_1 z_{1,x})_x=\phi,
\end{equation*}
where remainder term $\phi$ can be bound by $z_{1,x}v_1-z_1v_{1,x}$ and $w_{1,x}v_1-w_1v_{1,x}$. See \eqref{phi} for detailed definition of $\phi$. By adapting the area functional formula from Bianchini and Bressan \cite{BB-triangular,BB-vv-lim-ann-math} we can prove the required $L^1_tL^1_x$ estimate for the following terms
\begin{equation*}
	z_{1,x}v_1-z_1v_{1,x}.%\mbox{ and }z_{1,x}w_1-z_1w_{1,x}. 
\end{equation*}
Then we can estimate $v_{1,xx}w_1-w_{1,xx}v_1$ by $z_{1,x}v_1-z_1v_{1,x}$ and $w_{1,x}v_1-w_1v_{1,x}$ by applying a bootstrap arguments. 
In section \ref{section:stability}, we study the stability of solution. We consider an initial data $\bar{u}^\theta:=\theta \bar{u}+(1-\theta)\bar{v}$ and $u^\theta$ is the corresponding to initial data $\bar{u}^\theta$. As in \cite{BB-temple,BB-triangular,BB-vv-lim-ann-math}, $h:=\frac{\pa u}{\pa\theta}$ satisfies a linear equation (see \eqref{equationh}). We decompose $h$ in the basis of travelling waves. Then the unknowns $h_1,h_2$ satisfies parabolic equations with force. In the equation for $h_2$, there appears new term like $(\al_1-\al_2)(h_{1,xx}v_1-h_1v_{1,xx})$ and $(\al_1-\al_2)(h_{1,xx}w_1-h_1w_{1,xx})$. We resolve the issue by defining a new quantity $\hat{h}_1=\al_1h_{1,x}-\tilde{\la}_1h_1$. Again, due to the structure of the system, we obtain that $\hat{h}_1$ satisfies the following parabolic equation,
\begin{equation*}
	\hat{h}_{1,t}+(\tilde{\la}_1\hat{h}_1)_x-(\al_1 \hat{h}_{1,x})_x=\hat{\phi}_1,
\end{equation*}
where remainder term $\hat{\phi}_1$ can be bound by $\hat{h}_{1,x}v_1-\hat{h}_1v_{1,x}$, $h_{1,x}v_1-h_1v_{1,x}$ and $w_{1,x}h_1-w_1h_{1,x}$. Again by using the area functional formula from Bianchini and Bressan \cite{BB-triangular,BB-vv-lim-ann-math} we can prove the required $L^1_tL^1_x$ estimate.

Rest of the article is organized as follows. In the next subsection \ref{sec:main-result}, we state the main result on existence and stability of global smooth solution for \eqref{eqn-main}. In section \ref{sec:parabolic} and \ref{sec:grad-decomp}, we study the parabolic estimates of the solution and the decomposition of $u_x$ respectively. The remainder terms are calculated in section \ref{section:rem}. We study the global existence of smooth solution by proving uniform total variation estimate which is followed from an interaction estimates which can be found in section \ref{sec:shortening} and \ref{sec:transversal}. We prove the stability in section \ref{section:stability}. Finally, we establish the vanishing viscosity limit in section \ref{sec:vv-lim}.

\subsection{Presentation of the results}\label{sec:main-result}
Now, we state the main result of the article. 
 \begin{theorem}\label{theo1}
   	Consider the following hyperbolic system with viscosity,
   	\begin{equation}\label{eqn-thm-1}
   		u_t+A(u)u_x=(B(u)u_x)_x,\quad u(0,x)=\bar{u}(x),
   	\end{equation}
	where $A$ and $B$ are defined as in \eqref{defA} and \eqref{defB}.
   We assume that $f$ and $g$ are smooth functions verifying the strict hyperbolicity condition \eqref{condhyperbo} in a neighborhood of a compact set $K\subset\R^2$ 
   and that $A$ and $B$ are commutating such that the condition \eqref{condicommut} is satisfied.
There exists $L_1,L_2,L_3>0$ and $\de_0>0$ such that the following holds. If $\bar{u}$ satisfies
   \begin{equation}\label{condition-data-thm-1}
   	TV(\bar{u})\leq\de_0\mbox{ and }\lim\limits_{x\rr-\f}\bar{u}(x)=u^*\in K,
   \end{equation}
 then there exists unique solution $u$ to the Cauchy problem \eqref{eqn-thm-1} and it satisfies the following properties for any $t, s\geq 0$ 
\begin{align}
	TV(u(t))&\leq L_1TV(\bar{u}),\label{thm-1:BV-bound}\\
		\norm{u(t)-v(t)}_{L^1}&\leq L_2\norm{\bar{u}-\bar{v}}_{L^1},\label{thm1-Lipschitz}\\
	\norm{u(t)-u(s)}_{L^1}&\leq L_3\left(\abs{t-s}+\abs{\sqrt{t}-\sqrt{s}}\right),\label{L1-cont}
\end{align} 
where $v$ is the unique solution of \eqref{eqn-thm-1} with initial data $\bar{v}$ satisfying \eqref{condition-data-thm-1}. 
%Furthermore, when $A=Df$ for some $f\in C^1$, as $\e\rr0$ (up to a subsequence), $u^\e(t)\rr u(t,x)$  {\color{blue} in $L^1_{loc}(\R)$ topology} a solution to hyperbolic system \eqref{eqn:hyperbolic}. % and it satisfies
%\begin{equation}
%	\norm{u(t_1,\cdot)-v(t_2,\cdot)}_{L^1}\leq L_1\norm{\bar{u}-\bar{v}}_{L^1}+L_3\abs{t_1-t_2}.
%\end{equation}
   \end{theorem}
   \begin{remark}
   We would like to point out that in the previous Theorem our analysis does not require to consider a conservative system. Indeed we could deal with a triangular matrix $A(u)$ which does not satisfy necessary $A(u)=D f(u)$ with $f=(f_1,g)$. 
   \end{remark}
   We deduce the following results on the vanishing viscosity limit process.
   
    \begin{corollary}\label{theo2}
   	Consider the Cauchy problem hyperbolic system with vanishing viscosity,
   	\begin{equation}\label{eqn-thm-1a}
   		u^\e_t+A(u^\e)u^\e_x=\e(B(u^\e)u^\e_x)_x,\quad u^\e(0,x)=\bar{u}(x).
   	\end{equation}
   We assume that the drift $A$  and viscosity matrix $B$ satisfy the same conditions as in the Theorem \ref{theo1}. There exists $L_1,L_2,L_3>0$ and $\de_0>0$ such that the following holds. If $\bar{u}$ satisfies
   \begin{equation}\label{condition-data-thm-1a}
   	TV(\bar{u})\leq\de_0\mbox{ and }\lim\limits_{x\rr-\f}\bar{u}(x)\in K,
   \end{equation}
for some compact set $K\subset\R^2$ then there exists unique solution $u^\e$ to the Cauchy problem \eqref{eqn-thm-1a} and it satisfies the following properties
\begin{align}
	TV(u^\e(t))&\leq L_1TV(\bar{u}),\label{cor-1:BV-bound}\\
		\norm{u^\e(t)-v^\e(t)}_{L^1}&\leq L_2\norm{\bar{u}-\bar{v}}_{L^1},\label{cor-Lipschitz}\\
	\norm{u^\e(t)-u^\e(s)}_{L^1}&\leq L_3\left(\abs{t-s}+\sqrt{\e}\abs{\sqrt{t}-\sqrt{s}}\right),\label{cor-L1-cont}
\end{align} 
where $v^\e$ is the unique solution corresponding to $\bar{v}$ satisfying \eqref{condition-data-thm-1a}. 
\\
Furthermore, for any $t\geq 0$ the sequence $(u^\e(t))_{\e>0}$ converges to $u(t)$ in $L^1_{loc}(\R)$ topology which is the unique solution to hyperbolic system \eqref{hyperbo}. % and it satisfies
%\begin{equation}
%	\norm{u(t_1,\cdot)-v(t_2,\cdot)}_{L^1}\leq L_1\norm{\bar{u}-\bar{v}}_{L^1}+L_3\abs{t_1-t_2}.
%\end{equation}
   \end{corollary}

\section{Parabolic estimates}\label{sec:parabolic}
We wish now to prove different regularity estimates for solutions of \eqref{eqn-main}. In the sequel we consider a global $BV$ solution $u=u(t,x)$ of  \eqref{eqn-main} satisfying $\lim u(t,x)_{x\rightarrow-\infty}=u^*$, we observe easily that $u^*$ does not depend on the time $t$. We denote by
\begin{equation*}
\begin{aligned}
&P(u)=\begin{pmatrix}
		1&0\\
		-\gamma(u)&1
	\end{pmatrix}
\;\;\mbox{and}\;\;P^{-1}(u)=\begin{pmatrix}
		1&0\\
		\gamma(u)&1
	\end{pmatrix},
	\end{aligned}
\end{equation*}
which satisfy
\begin{equation}\label{diag}
P(u)A(u)P^{-1}(u)=\begin{pmatrix}
		f'_1(u_1)&0\\
		0&\frac{\pa g}{\pa u_2}(u)
	\end{pmatrix}\;\;\mbox{and}\;\;P(u)B(u)P^{-1}(u)=\begin{pmatrix}
		\alpha_1(u_1)&0\\
		0&\alpha_2(u)
	\end{pmatrix}.
\end{equation}
Furthermore, we indicate the eigenvectors of $A(u)$ as follows
	\begin{equation}	\label{eigenv}
		r_1(u)=\begin{pmatrix}
			1\\
			\gamma(u)
		\end{pmatrix}\mbox{and}\; r_2(u)=r_2=\begin{pmatrix}
			0\\
			1
		\end{pmatrix}\mbox{ where }\gamma(u)=\frac{\frac{\pa g}{\pa u_1}(u)}{f^\p(u_1)-\frac{\pa g}{\pa u_2}(u)}.
	\end{equation}
	We can note that the previous formula is well defined due to the strict hyperbolicity assumption \eqref{condhyperbo}.
	 We consider in the sequel the vectors $l_i$ for $i\in\{1,2\}$ with $(l_i)_j=1$ if $j=i$ and $0$ otherwise. 
	 
	 Note that $A(u)=DF(u)$ with $F=(f,g)^T$, it implies in particular that the system is conservative. In the sequel we consider $A_1,B_1$ defined as follows
	 \begin{equation*}
	 	A_1(u)=\begin{pmatrix}
	 		f'_1(u_1)&0\\
	 		0&\frac{\pa g}{\pa u_2}(u)
	 	\end{pmatrix}\mbox{ and }
 	B_1(u)=\begin{pmatrix}
 		\alpha_1(u_1)&0\\
 		0&\alpha_2(u)
 	\end{pmatrix}.
	 \end{equation*}
 Taking now the derivative of \eqref{eqn-main}, we get:
\begin{equation}
\begin{aligned}
&u_{x,t}+A(u)u_{xx}=(B(u)u_{xx})_x+{\cal R},
\end{aligned}
\label{ux}
\end{equation}
with:
\begin{equation}
{\cal R}=\begin{pmatrix}
		(\alpha_1'(u_1)(u_{1,x})^2)_x-f''(u_1)(u_{1,x})^2\\
		(\frac{\pa }{\pa x}(\beta(u))u_{1,x})_x+(\frac{\pa }{\pa x}(\alpha_2(u))u_{2,x})_x
	\end{pmatrix}
\end{equation}
We wish now as in \cite{HJ} to diagonalize the previous system, then we consider in this section the unknown $w=\begin{pmatrix}
		w_1\\
		w_2
	\end{pmatrix}=P(u)u_x=\begin{pmatrix}
		u_{1,x}\\
		-\gamma(u)u_{1,x}+u_{2,x}
	\end{pmatrix}$. We want to use the equation \eqref{ux} and the matrix $P(u)$. The direct computations give
\begin{equation}\label{ux2}
\begin{aligned}
&P(u)u_{x,t}=w_t-\begin{pmatrix}
		0\\
		\frac{\pa }{\pa t}(\gamma(u))w_1
	\end{pmatrix},\;\;P(u)A(u)u_{xx}=A_1(u)w_x+A_1(u)\begin{pmatrix}
		0\\
		\frac{\pa }{\pa x}(\gamma(u))w_1
	\end{pmatrix}\\
	&P(u)(B(u)u_{xx})_x=P(u)B(u)u_{xxx}+P(u)\begin{pmatrix}
		\alpha'_1(u_1)u_{1,x} u_{1,xx}\\
		\frac{\pa }{\pa x}(\beta(u))u_{1,xx}+\frac{\pa }{\pa x}(\alpha_2(u))u_{2,xx}
	\end{pmatrix}.
\end{aligned}
\end{equation}
Using the fact that $P(u)u_{xxx}=w_{xx}+{\cal R}_1$ with ${\cal R}_1=\begin{pmatrix}
		0\\
		2\frac{\pa }{\pa x}(\gamma(u))u_{1,xx}+\frac{\pa ^2}{\pa^2 x}(\gamma(u))u_{1,x}
	\end{pmatrix}$ it follows that
\begin{equation}
\begin{aligned}
&P(u)B(u)u_{xxx}=B_1(u)w_{xx}+B_1(u){\cal R}_1
\end{aligned}
\label{ux3}
\end{equation}
Setting 
\begin{align*}
&{\cal R}_2=\begin{pmatrix}
	\alpha'_1(u_1)u_{1,x} u_{1,xx}\\
	\frac{\pa}{\pa x}(\beta(u))u_{1,xx}+\frac{\pa }{\pa x}(\alpha_2(u))u_{2,xx}
\end{pmatrix},\\
& {\cal R}_3=\begin{pmatrix}
	0\\
	\frac{\pa }{\pa t}(\gamma(u))u_{1,x}
\end{pmatrix},
 {\cal R}_4=A_1(u)\begin{pmatrix}
	0\\
	\frac{\pa }{\pa x}(\gamma(u))u_{1,x}
\end{pmatrix},
\end{align*}
and combining \eqref{ux}, \eqref{ux2} and \eqref{ux3} it yields
\begin{equation}
w_t+A_1(u)w_x-B_1(u)w_{xx}=B_1(u){\cal R}_1+P(u){\cal R}_2+{\cal R}_3-{\cal R}_4+P(u){\cal R}.
\end{equation}
We can simplify ${\cal R}_3$ and observe that
 \begin{align}
 ({\cal R}_3)_2&=u_{1,x}\biggl(\frac{\pa}{\pa u_1}\gamma(u)\big(-f'(u_1)u_{1,x}+\alpha_1'(u_1)(u_{1,x})^2+\alpha_1(u_1)u_{1,xx}\big) \nonumber\\
 &+\frac{\pa}{\pa u_2}\gamma(u)(-\frac{\pa}{\pa u_1}g(u)u_{1,x}-\frac{\pa}{\pa u_2}g(u)u_{2,x}
 +(\beta(u)u_{1,x}+\alpha_2(u)u_{2,x})_x\big)\biggl).\label{hypertech}
\end{align}
 %Using the fact that $u_{1,x}=w_1$ and $u_{2,x}=\gamma(u)w_1+w_2$ direct computations show that:
%\begin{equation}
%\begin{aligned}
%&(B_1(u){\cal R}_1)_2=\alpha_2(u)\biggl[2\big((\frac{\pa \gamma}{\pa u_1}(u)+\frac{\pa \gamma}{\pa u_2}(u)\gamma(u))w_1+\frac{\pa \gamma}{\pa u_2}(u) w_2\big)w_{1,x}+\big( \frac{\pa^2 \gamma}{\pa^2 u_1}(u)w_1^2\\
%&+2 \frac{\pa^2 \gamma}{\pa u_1\pa u_2}(u)w_1(\gamma(u)w_1+w_2)+ \frac{\pa^2 \gamma}{\pa^2 u_2}(u)(\gamma(u)w_1+w_2)^2+\frac{\pa \gamma}{\pa u_1}(u)w_{1,x}\\
%&+\frac{\pa \gamma}{\pa u_2}(u)(\gamma(u)w_1+w_2)_x\big)w_{1}\biggl],\\
%&(B_1(u){\cal R}_1)_1=0,\\
%&({\cal R}_3)_2=\frac{d}{dt}(\gamma(u))w_1=w_1\biggl(\frac{\pa \gamma}{\pa u_1}(u)\big(\alpha'_1(u_1)w_1^2+\alpha_1(u_1)w_{1,x}-f'(u_1)w_1)+\frac{\pa \gamma}{\pa u_2}(u)\big( -(g(u))_x\\
%&+(\beta(u))_x w_1+\beta(u)w_{1,x}+(\alpha_2(u))_x(\gamma(u)w_1+w_2)+
%\alpha_2(u)(\gamma(u)_xw_1+\gamma(u)w_{1,x}+w_{2,x})\big)\biggl)\\
%&({\cal R})_1=(\alpha'(u_1))_x(w_1)^2+\alpha'(u_1)w_1w_{1,x}-f''(u_1)w_1^2\\
%&({\cal R})_2=\frac{d^2}{dx^2}(\beta(u))w_1+\frac{d}{dx}(\beta(u)) w_{1,x}+\frac{d^2}{dx^2}(\alpha_2(u))(\gamma(u)w_1+w_2)+\frac{d}{dx}(\alpha_2(u))(\gamma(u)w_1+w_2)_x
%\end{aligned}
%\label{hypertech}
%\end{equation}
We have in particular obtained the following system:
\begin{equation}
\begin{cases}
\begin{aligned}
&w_{1,t}+f_1'(u_1)w_{1,x}-\alpha_1(u_1)w_{1,xx}=\big(B_1(u){\cal R}_1+P(u){\cal R}_2+{\cal R}_3-{\cal R}_4+P(u){\cal R}\big)_1,\\
&w_{2,t}+\frac{\pa }{\pa u_2}g(u) w_{2,x}-\alpha_2(u)w_{2,xx}=\big(B_1(u){\cal R}_1+P(u){\cal R}_2+{\cal R}_3-{\cal R}_4+P(u){\cal R}\big)_2.
\end{aligned}
\end{cases}
\end{equation}
As in \cite{HJ}, we set $w_1(t,x)=\widetilde{w}_1(t,X_1(t,x))$, $w_2(t,x)=\widetilde{w}_1(t,X_2(t,x))$ with $(X_1)_x(t,x)=\frac{1}{\sqrt{\alpha_1(u_1(t,x))}}$ and $(X_2)_x(t,x)=\frac{1}{\sqrt{\alpha_2(u(t,x))}}$. It implies that for $i\in\{1,2\}$,
\begin{equation}\label{estimw}
\begin{cases}
\begin{aligned}
w_{i,t}(t,x)&=\widetilde{w}_{i,t}(t,X_i(t,x))+\widetilde{w}_{i,x}(t,X_i(t,x))X_{i,t}(t,x),\\
w_{i,x}(t,x)&=\widetilde{w}_{i,x}(t,X_i(t,x))\frac{1}{\sqrt{\alpha_i(u(t,x))}},\\
w_{i,xx}(t,x)&=\widetilde{w}_{i,xx}(t,X_i(t,x))\frac{1}{\alpha_i(u(t,x))}\\
&-\frac{1}{2}(\alpha_{i}(u(t,x))^{-\frac{3}{2}}(\frac{\pa \alpha_i}{\pa u_1}(u(t,x))u_{1,x}(t,x)\\
&+\frac{\pa \alpha_i}{\pa u_2}(u(t,x))u_{2,x}(t,x))\widetilde{w}_{i,x}(t,X_i(t,x)).
\end{aligned}
\end{cases}
\end{equation}
Furthermore,
\begin{align}
	&w_{i,xxx}(t,x)\nonumber\\
	&=\widetilde{w}_{i,xxx}(t,X_i(t,x))\frac{1}{\alpha_i^{3/2}(u(t,x))}\nonumber\\
	&-\frac{3}{2\alpha^{2}_{i}(u(t,x))}\left(\frac{\pa \alpha_i}{\pa u_1}(u(t,x))u_{1,x}(t,x)+\frac{\pa \alpha_i}{\pa u_2}(u(t,x))u_{2,x}(t,x)\right)\widetilde{w}_{i,xx}(t,X_i(t,x))\nonumber\\
	&+\frac{3}{4\alpha^{5/2}_{i}(u(t,x))}\left(\frac{\pa \alpha_i}{\pa u_1}(u(t,x))u_{1,x}(t,x)+\frac{\pa \alpha_i}{\pa u_2}(u(t,x))u_{2,x}(t,x)\right)^2\widetilde{w}_{i,x}(t,X_i(t,x))\nonumber\\
	&-\frac{1}{2\alpha^{3/2}_{i}(u(t,x))}\left(\frac{\pa^2 \alpha_i}{\pa u_1^2}(u(t,x))u_{1,x}(t,x)+\frac{\pa^2 \alpha_i}{\pa u_2\pa u_1}(u(t,x))u_{2,x}(t,x)\right)u_{1,x}(t,x)\widetilde{w}_{i,x}(t,X_i(t,x))\nonumber\\
	&-\frac{1}{2\alpha^{3/2}_{i}(u(t,x))}\left(\frac{\pa^2 \alpha_i}{\pa u_1\pa u_2}(u(t,x))u_{1,x}(t,x)+\frac{\pa^2 \alpha_i}{\pa u_2^2}(u(t,x))u_{2,x}(t,x)\right)u_{2,x}(t,x)\widetilde{w}_{i,x}(t,X_i(t,x))\nonumber\\
	&-\frac{1}{2\alpha^{3/2}_{i}(u(t,x))}\left(\frac{\pa \alpha_i}{\pa u_1}(u(t,x))u_{1,xx}(t,x)+\frac{\pa \alpha_i}{\pa u_2}(u(t,x))u_{2,xx}(t,x)\right)\widetilde{w}_{i,x}(t,X_i(t,x)).\label{eqn:w-i-xxx}
\end{align}
Since $X_i$ is invertible from $\R$ to $\R$ we obtain the following system on $\widetilde{w}_i$,
\begin{equation}
\begin{cases}
\begin{aligned}
&\widetilde{w}_{1,t}+\frac{f'_1(u_1)}{\sqrt{\alpha_1(u_1)}}(\cdot,X_1^{-1})\widetilde{w}_{1,x}-\widetilde{w}_{1,xx}\\
&\quad\quad\quad=-\frac{1}{2}(\alpha_{1}(u_1(\cdot,X_1^{-1}))^{-\frac{1}{2}}(\frac{\pa \alpha_1}{\pa u_1}(u(\cdot,X_1^{-1}))u_{1,x}(\cdot,X_1^{-1})\\
&\quad\quad\quad\quad+\frac{\pa \alpha_1}{\pa u_2}(u(\cdot,X_1^{-1})))u_{2,x}(\cdot,X_1^{-1})))\widetilde{w}_{1,x}+{\cal G}_1(\cdot,X_1^{-1})-\widetilde{w}_{1,x}X_{1,t}(\cdot,X_1^{-1}),\\
&\widetilde{w}_{2,t}+\frac{\frac{\pa }{\pa u_2}g(u) }{\sqrt{\alpha_2(u)}}(\cdot,X_2^{-1})\widetilde{w}_{2,x}-\widetilde{w}_{2,xx}\\
&\quad\quad\quad=-\frac{1}{2}(\alpha_{2}(u(\cdot,X_2^{-1}))^{-\frac{1}{2}}(\frac{\pa \alpha_2}{\pa u_1}(u(\cdot,X_2^{-1}))u_{1,x}(\cdot,X_2^{-1})\\
&\quad\quad\quad\quad+\frac{\pa \alpha_2}{\pa u_2}(u(\cdot,X_2^{-1})))u_{2,x}(\cdot,X_2^{-1})))\widetilde{w}_{2,x}+{\cal G}_2(\cdot,X_2^{-1})-\widetilde{w}_{2,x}X_{2,t}(\cdot,X_2^{-1}),
\end{aligned}
\end{cases}
\label{sol29}
\end{equation}
where the term $\mathcal{G}$ is defined as follows
\begin{equation}\label{defG}
{\cal G}=B_1(u){\cal R}_1+P(u){\cal R}_2+{\cal R}_3-{\cal R}_4+P(u){\cal R}.
\end{equation}
We denote now respectively by $G_1$ and $G_2$ the fundamental solution of the two following equations
\begin{equation}
\begin{cases}
\begin{aligned}
&w_t+\frac{f'_1(u^*_1)}{\sqrt{\alpha_1(u^*_1)}}w_{x}-w_{xx}=0,\\
&w_t+\frac{\frac{\pa }{\pa u_2}g(u^*) }{\sqrt{\alpha_2(u^*)}}w_{x}-w_{xx}=0.
\end{aligned}
\end{cases}
\end{equation}
The functions $G_i$ satisfy the following estimates (see \cite{Friedman}) for $\kappa>0$ large enough
\begin{equation}\label{def:kappa}
	\norm{G_i(t,\cdot)}_{L^1}\leq \kappa,\quad	\norm{G_{i,x}(t,\cdot)}_{L^1}\leq \frac{\kappa}{\sqrt{t}},\quad	\norm{G_{i,xx}(t,\cdot)}_{L^1}\leq \frac{\kappa}{t}.
\end{equation}
To have the parabolic estimates for \eqref{eqn-main}, we define the following constants
\begin{align*}
	\kappa_A&=\sup\limits_{\tiny\abs{u-u^*}\leq \bar{\de}}\max\limits_{\tiny\begin{array}{cc}&1\leq k\leq5,\\
			& |\theta|\leq 5\end{array}}\left\{\abs{f^{(k)}_1(u_1)},\abs{\pa^\theta g(u)},1\right\},\\
	\kappa_B&=\sup\limits_{\tiny \abs{u-u^*}\leq \bar{\de}}\max\limits_{\tiny\begin{array}{cc}&0\leq|\theta|\leq 5,\\
			&0\leq k\leq5,\\
			 &1\leq i\leq 2\end{array}}\left\{\frac{1}{\sqrt{\alpha_i(u)}},\sqrt{\alpha_i(u)},|\alpha_1^{(k)}(u_1)|,|\pa^\theta\alpha_2(u)|,\right.\\
		 &\hspace*{8cm}\left.\frac{1}{\alpha_i(u)},|\pa^\theta(\gamma(u))|,|\pa^\theta(\beta(u))|\right\},
\end{align*}
with $\bar{\delta}>0$ such that the solution $u(t,x)$ lives all along the time in the ball $B(u^*, {\bar{\delta}}/{2})$ which will be satisfied provided that $\mbox{Tot.Var.}(u(0,\cdot))$ is chosen sufficiently small. 

Now, we have the following result on the parabolic regularizing of $u_1$ and $u_2$. This helps us to establish the smoothness of solution up to a small time $\hat{t}$.
\begin{proposition}\label{prop:parabolic}
	There exist $C_1,C_2,C_3,C_4>0$ depending on $\kappa_A,\kappa_B,\kappa$ such that if $u$ is a solution to the equation \eqref{eqn-main} satisfying 
	\begin{equation}\label{assumption:u-L1}
		\norm{u_{1,x}(t,\cdot)}_{L^1}+\norm{u_{2,x}(t,\cdot)}_{L^1}\leq \de_0\mbox{ for all }t\in[0,\hat{t}]\mbox{ where }\hat{t}\leq \frac{1}{(C_4\delta_0)^2},
	\end{equation}
	for $\delta_0$ sufficiently small in terms of $\kappa_A$, $\kappa_B$ and $\kappa$, then we have for $i\in\{1,2\}$
	\begin{equation}\label{estimate:parabolic-1}
		\norm{u_{i,xx}(t,\cdot)}_{L^1}\leq \frac{C_1\delta_0}{\sqrt{t}},\,\, \norm{u_{i,xxx}(t,\cdot)}_{L^1}\leq \frac{C_2\de_0}{t}\mbox{ and }\norm{u_{i,xxxx}(t,\cdot)}_{L^1}\leq \frac{C_3\de_0}{t^{3/2}}.
	\end{equation}
%	\begin{equation}
%		%
%	\begin{aligned}
%		&\sum_{i=1}^2\norm{u_{i,xx}(t,\cdot)}_{L^1}\leq \frac{C_1\delta_0}{\sqrt{t}}\\
%		%\frac{16 \kappa \kappa_B^6\de_0}{\sqrt{t}},\\
%		&\sum_{i=1}^2 \norm{u_{i,xxx}(t,\cdot)}_{L^1}\leq %\frac{2500\kappa_1^6\kappa^3\kappa_A\kappa_B^7\kappa_P^{13}\delta_0}{\sqrt{t}.
%		\frac{C_2\de_0}{t},\\
%		&\sum_{i=1}^2\norm{u_{i,xxxx}(t,\cdot)}_{L^1}\leq %\frac{2500\kappa_1^6\kappa^3\kappa_A\kappa_B^7\kappa_P^{13}\delta_0}{\sqrt{t}.
%		\frac{C_3}{\sqrt{t}t}.
%		\end{aligned}
%	\end{equation}
\end{proposition}
\begin{proof}We first prove \eqref{estimate:parabolic-1} for smooth initial data. We argue by contradiction. To this end, first we assume that the conclusion \eqref{estimate:parabolic-1} does not hold. Due to the assumption of smoothness of initial data, solution is smooth up to a small time and due to the continuity we can assume that there exists a time $t^*$ such that \eqref{estimate:parabolic-1} holds for $t\in[0,t^*]$ and equality attains at $t=t^*$.
	
	 From \eqref{sol29}, the solution $\widetilde{w}_i$ with $i\in\{1,2\}$ can be represented as follows for $t>0$,
\begin{align*}
	\widetilde{w}_{i,x}(t,\cdot)&=G_{i,x}({t}/{2},\cdot)*\widetilde{w}_{i}({t}/{2},\cdot)\\
	&-\frac{1}{2}\int_{\frac{t}{2}}^tG_{i,x}(t-s)*\Big[\alpha_{i}(u(\cdot,X_i^{-1}))^{-\frac{1}{2}}\Big(\frac{\pa \alpha_i}{\pa u_1}(u(\cdot,X_i^{-1}))u_{1,x}(\cdot,X_i^{-1})\\
	&+\frac{\pa \alpha_i}{\pa u_2}(u(\cdot,X_i^{-1}))u_{2,x}(\cdot,X_i^{-1})\Big)\widetilde{w}_{i,x}+{\cal G}_i(\cdot,X_i^{-1})+2{\cal A}_i-\widetilde{w}_{i,x}X_{i,t}(\cdot,X_i^{-1})\Big]ds,
\end{align*}
where ${\cal A}_1$ and ${\cal A}_2$ are defined as follows
\begin{equation}\label{defn:A-i}
\begin{array}{rl}
&{\cal A}_1=\bigg(\frac{f'_1(u^*_1)}{\sqrt{\alpha_1(u^*_1)}}-\frac{f'_1(u_1)}{\sqrt{\alpha_1(u_1)}}(\cdot,X_1^{-1})\bigg)\widetilde{w}_{1,x},\\
&{\cal A}_2=\bigg(\frac{\frac{\pa }{\pa u_2}g(u^*) }{\sqrt{\alpha_2(u^*)}}-\frac{\frac{\pa }{\pa u_2}g(u) }{\sqrt{\alpha_2(u)}}(\cdot,X_2^{-1})\bigg)\widetilde{w}_{2,x}.
\end{array}	
\end{equation}
From \eqref{def:kappa}, we deduce that for $t>0$, $C>0$
\begin{align}
\| \widetilde{w}_{i,x}(t,\cdot)\|_{L^1}&\leq \frac{\kappa\sqrt{2}}{\sqrt{t}}\|\widetilde{w}_{i}({t}/{2},\cdot)\|_{L^1}\nonumber\\
&+\mathcal{O}(1) \int_{\frac{t}{2}}^t\frac{\kappa}{\sqrt{t-s}}\bigg(\|\Big[\alpha_{i}(u(\cdot,X_i^{-1}))^{-\frac{1}{2}}\Big(\frac{\pa \alpha_i}{\pa u_1}(u(\cdot,X_i^{-1}))u_{1,x}(\cdot,X_i^{-1})\nonumber\\
&\hspace{6cm}+\frac{\pa \alpha_2}{\pa u_2}(u(\cdot,X_i^{-1}))u_{2,x}(\cdot,X_i^{-1})\Big)\Big]\widetilde{w}_{i,x}\|_{L^1}\nonumber\\
&\hspace{1cm}+\|{\cal G}_i(\cdot,X_i^{-1})\|_{L^1}+\|{\cal A}_i\|_{L^1}+\|\widetilde{w}_{i,x}(s,\cdot)X_{i,t}(s,X_i^{-1}(s,\cdot))\|_{L^1}\bigg)ds. \label{estim0}
\end{align}
We observe using \eqref{defG}, \eqref{hypertech},
\begin{align}
&\|{\cal G}_i(s,X_i^{-1}(s,\cdot))\|_{L^1}=\mathcal{O}(1) \|{\cal G}_i(s,\cdot)\|_{L^1} \nonumber\\
&=\mathcal{O}(1) \big(\|u_{2,x} u_{1,xx}(s,\cdot)\|_{L^1}+\|u_{1,x} u_{1,xx}(s,\cdot)\|_{L^1}+\|u_{1,x}^3(s,\cdot)\|_{L^1}+\|u_{1,x} u_{2,x}^2(s,\cdot)\|_{L^1}\nonumber\\
&\hspace{0.5cm}+\| u_{2,x}u_{1,x}^2(s,\cdot)\|_{L^1}+\|u_{2,xx}u_{1,x}(s,\cdot)\|_{L^1}+\|u_{2,xx}u_{2,x}(s,\cdot)\|_{L^1}
+\|u_{1,x}^2(s,\cdot)\|_{L^1}\nonumber\\
&\hspace{0.5cm}+\|u_{1,x}u_{2,x}(s,\cdot)\|_{L^1}\big)\nonumber\\
&= \mathcal{O}(1)(\|u_{2,xx}\|_{L^1}\| u_{1,xx}\|_{L^1}+\| u_{1,xx}\|^2_{L^1}+\|u_{1,xx}\|_{L^1}^2\delta_0+\delta_0\| u_{2,xx}\|^2_{L^1}\nonumber\\
&\hspace{2cm}+\delta_0\| u_{2,xx}\|_{L^1}\|u_{1,xx}\|_{L^1}+\|u_{2,xx}\|^2_{L^1}+\delta_0\|u_{1,xx}\|_{L^1}+\delta_0\|u_{2,xx}\|_{L^1}\big).\label{estim1}
\end{align}
Similarly, we have since $w_1=u_{1,x}$ and $w_2=-\gamma(u)u_{1,x}+u_{2,x}$,
\begin{align}
\|{\cal A}_i(s,\cdot)\|_{L^1}&=\mathcal{O}(1)\| u_x(s,\cdot)\|_{L^1}\|\widetilde{w}_{i,x}(s,\cdot)\|_{L^1}\nonumber\\
&=\mathcal{O}(1)\| u_x(s,\cdot)\|_{L^1}\norm{w_{i,x}(s,X_i^{-1}(s,\cdot)))\sqrt{\alpha_i(u(s,X_i^{-1}(s,\cdot))}}_{L^1}\nonumber\\
&=\mathcal{O}(1) \| u_x(s,\cdot)\|_{L^1}\|w_{i,x}(s,\cdot)\|_{L^1}\nonumber\\
&=\mathcal{O}(1)\| u_x(s,\cdot)\|_{L^1}(\|u_{1,xx}(s,\cdot)\|_{L^1}+\|u_{2,xx}(s,\cdot)\|_{L^1}
\nonumber\\
&\hspace{5cm}+\|u_{1,x}^2(s,\cdot)\|_{L^1}+\|u_{1,x}u_{2,x}(s,\cdot)\|_{L^1})\nonumber\\
&=\mathcal{O}(1)\delta_0(\|u_{1,xx}(s,\cdot)\|_{L^1}+\|u_{2,xx}(s,\cdot)\|_{L^1}
).\label{estim2}
\end{align}
Similarly, we have
\begin{align}
&\norm{\alpha_{i}(u(\cdot,X_i^{-1}))^{-\frac{1}{2}}\Big(\frac{\pa \alpha_i}{\pa u_1}(u(\cdot,X_i^{-1}))u_{1,x}(\cdot,X_i^{-1})+\frac{\pa \alpha_i}{\pa u_2}(u(\cdot,X_i^{-1}))u_{2,x}(\cdot,X_i^{-1})\Big)\widetilde{w}_{i,x}}_{L^1}\nonumber\\
&=\mathcal{O}(1)(\|u_{1,x}(\cdot,X_i^{-1})\|_{L^\infty}+\|u_{2,x}(\cdot,X_i^{-1})\|_{L^\infty})\|\widetilde{w}_{i,x}(s,\cdot)\|_{L^1}\nonumber\\
&=\mathcal{O}(1) \big(\|u_{1,xx}(s,\cdot)\|_{L^1}+\|u_{2,xx}(s,\cdot)\|_{L^1}) (\|u_{1,xx}(s,\cdot)\|_{L^1}+\|u_{2,xx}(s,\cdot)\|_{L^1}\nonumber\\
&\hspace{1cm}+\|u_{1,x}^2(s,\cdot)\|_{L^1}+\|u_{1,x}u_{2,x}(s,\cdot)\|_{L^1}\big)\nonumber\\
&=\mathcal{O}(1) (\|u_{1,xx}(s,\cdot)\|_{L^1}+\|u_{2,xx}(s,\cdot)\|_{L^1}) ^2.
\label{estim3}
\end{align}
We want to estimate now $\widetilde{w}_{i,x}X_{i,t}(\cdot,X_i^{-1})$, let us compute $X_{i,t}$, we have then:
\begin{align}
&X_{i,t}(x)=-\frac{1}{2}\int^x_0\alpha_i(u(t,z))^{-\frac{3}{2}}(\frac{\pa}{\pa u_1}\alpha_i(u(t,z))u_{1,t}(t,z)+
\frac{\pa}{\pa u_2}\alpha_i(u(t,z))u_{2,t}(t,z))dz\nonumber\\
&=-\frac{1}{2}\int^x_0\alpha_i(u(t,z))^{-\frac{3}{2}}\Big[\frac{\pa}{\pa u_1}\alpha_i(u(t,z))\big(-f'(u_1(t,z))u_{1,x}(t,z)+(\alpha_1(u_1)u_{1,x})_x(t,z)\big)
\nonumber\\
&\hspace*{2cm}+
\frac{\pa}{\pa u_2}\alpha_i(u(t,z))\big(-(g(u))_x(t,z)+(\beta(u)u_{1,x}+\alpha_2(u)u_{2,x})_x(t,z)\big)\Big]\,dz.\label{eqn-X-i-t}
\end{align}
It implies that for $C_2>0$,
\begin{equation*}
\begin{aligned}
\|X_{i,t}(s,\cdot)\|_{L^\infty}&=\mathcal{O}(1)\Big(\|u_{1,x}\|_{L^1}+\|u_{2,x}\|_{L^1}+\|u^2_{1,x}\|_{L^1}+\|u_{2,x}u_{1,x}\|_{L^1}+\|u^2_{2,x}\|_{L^1}\\
&\hspace{6cm}+\|u_{1,xx}\|_{L^1}+\|u_{2,xx}\|_{L^1}\Big)(s,\cdot)\\
&=\mathcal{O}(1)(2\delta_0+2\|u_{1,xx}\|_{L^1}+2\|u_{2,xx}\|_{L^1}
)(s,\cdot).
\end{aligned}
\label{2.24}
\end{equation*}
It implies that
\begin{align}
&\|\widetilde{w}_{i,x}X_{i,t}(s,X_i^{-1}(s,\cdot))\|_{L^1}\nonumber\\
&=\mathcal{O}(1)\Big(2\delta_0+2\delta_0 \|u_{1,xx}\|_{L^1}+\delta_0\|u_{2,xx}\|_{L^1}+\|u_{1,xx}\|_{L^1}+\|u_{2,xx}\|_{L^1}\Big)(s,\cdot)\|\widetilde{w}_{i,x}(s,\cdot)\|_{L^1}\nonumber\\
&=\mathcal{O}(1)\big(\delta_0+ \|u_{1,xx}(s,\cdot)\|_{L^1}+\|u_{2,xx}(s,\cdot)\|_{L^1}\big)\big(\|u_{1,xx}(s,\cdot)\|_{L^1}+\|u_{2,xx}(s,\cdot)\|_{L^1}\big).\label{estim4}
\end{align}
We wish now to estimate $\|u_{i,xx}(t,\cdot)\|_{L^1}$, we recall that
\begin{align}
&\|u_{i,xx}(t,\cdot)\|_{L^1}\leq \kappa_B^2 (\|w_{1,x}(t,\cdot)\|_{L^1}+\|w_{2,x}(t,\cdot)\|_{L^1}+\|w^2_1(t,\cdot)\|_{L^1}+\|w_1w_2(t,\cdot)\|_{L^1})\nonumber\\
&\leq \kappa_B^4 (\|\widetilde{w}_{1,x}(t,\cdot)\|_{L^1}+\|\widetilde{w}_{2,x}(t,\cdot)\|_{L^1})+ \kappa_B^2\|u_{1,x}(t,\cdot)\|_{L^1}\|u_{1,xx}(t,\cdot)\|_{L^1}\nonumber\\
&+2\kappa^3_B\|u_{1,x}\|_{L^1}\big(\|u_{1,xx}(t,\cdot)\|_{L^1}+\|u_{2,xx}(t,\cdot)\|_{L^1}\big)\nonumber\\
&\leq \kappa_B^4 (\|\widetilde{w}_{1,x}(t,\cdot)\|_{L^1}+\|\widetilde{w}_{2,x}(t,\cdot)\|_{L^1})+ 3\kappa_B^3\delta_0\big(\|u_{1,xx}(t,\cdot)\|_{L^1}
\|u_{2,xx}(t,\cdot)\|_{L^1}\big).\label{estim5}
\end{align}
It yields
\begin{equation}\label{estim5bis}
(1- 6\kappa_B^3\delta_0)\sum_{i=1}^{2}\|u_{i,xx}(t,\cdot)\|_{L^1}\leq 2\kappa_B^4 (\|\widetilde{w}_{1,x}(t,\cdot)\|_{L^1}+\|\widetilde{w}_{2,x}(t,\cdot)\|_{L^1}).
\end{equation}
Since $1- 6\kappa_B^3\delta_0\geq\frac{1}{2}$ we deduce that
\begin{equation}
\begin{aligned}
&\sum_{i=1}^{2}\|u_{i,xx}(t,\cdot)\|_{L^1}\leq 4\kappa_B^4 (\|\widetilde{w}_{1,x}(t,\cdot)\|_{L^1}+\|\widetilde{w}_{2,x}(t,\cdot)\|_{L^1}).
\end{aligned}
\label{estim5bis1}
\end{equation}
Combining \eqref{estim0}, \eqref{estim1}, \eqref{estim2}, \eqref{estim3}, \eqref{estim4}, \eqref{estim5bis1} and \eqref{estimate:parabolic-1} it gives for $C_3>0$ large enough
\begin{align*}
	\sum_{i=1}^2\|u_{i,xx}(t,\cdot)\|_{L^1}&\leq \frac{8 \kappa \kappa_B^6\sqrt{2}}{\sqrt{t}}\Big(\|u_{1,x}({t}/{2},\cdot)\|_{L^1}+\|u_{2,x}({t}/{2},\cdot)\|_{L^1}\Big)\\
	&+C_3 \sum_{i=1,2}\int_{\frac{t}{2}}^t\frac{1}{\sqrt{t-s}} \Big[\delta_0(\|u_{1,xx}(s,\cdot)\|_{L^1}+\|u_{2,xx}(s,\cdot)\|_{L^1})\\
	&\hspace*{4.5cm}+\|u_{1,xx}(s,\cdot)\|^2_{L^1}+\|u_{2,xx}(s,\cdot)\|^2_{L^1}\\
	&\hspace*{5cm}+\|u_{1,xx}(s,\cdot)\|_{L^1}\|u_{2,xx}(s,\cdot)\|_{L^1}\Big] ds.
\end{align*}
Hence, we obtain
\begin{align}
	\sum_{i=1}^2\|u_{i,xx}(t,\cdot)\|_{L^1}&\leq  \frac{\kappa \kappa_B^4\sqrt{2}}{\sqrt{t}}\delta_0
	+4 C_3 \kappa_B^4\int_{\frac{t}{2}}^t\frac{\kappa}{\sqrt{t-s}} \kappa_B^7\kappa_A^2(2 \delta_0^2 \frac{2\kappa \kappa_B^4}{\sqrt{s}}+ \frac{12\kappa^2 \kappa_B^8\de_0^2}{s}
	)ds\nonumber\\
	&\leq \frac{\kappa \kappa_B^6 8 \sqrt{2}}{\sqrt{t}}\delta_0+32 C_3\kappa_B^{15}\kappa^2\kappa_A^2\delta_0^2+\frac{96\sqrt{2} C_3\kappa^3\kappa_B^{19}\kappa_A^2\delta_0^2}{\sqrt{t}},\label{estim6}
\end{align}
 which implies  since $\de_0<\min\left\{\frac{1}{2},\frac{2\sqrt{2}(2-\sqrt{2})}{96 C_3\kappa_B^13\kappa^2\kappa_A^2}\right\}$ and $t^*\leq\hat{t}$ that
	\begin{equation*}
		\sum_{i=1}^2\norm{u_{xx}(t^*,\cdot)}_{L^1}<\frac{16 \kappa \kappa_B^6\de_0}{\sqrt{t^*}}.
	\end{equation*}
  This contradicts with the assumption that equality holds in \eqref{estimate:parabolic-1} at $t=t^*$. Hence, by density argument we conclude that \eqref{estimate:parabolic-1} holds for BV initial data as well.
  
  We apply similar arguments to prove the second and third estimate of \eqref{estimate:parabolic-1}, we have by using \eqref{estimw},
\begin{align*}
\|u_{i,xxx}(t,\cdot)\|_{L^1}&=\mathcal{O}(1)   \Big(\|w_{1,xx}(t,\cdot)\|_{L^1}+\|w_{2,xx}(t,\cdot)\|_{L^1}+
\|w^3_{1}(t,\cdot)\|_{L^1}+\|w^2_{1}w_2(t,\cdot)\|_{L^1}\\
&\hspace*{2cm}+\|w_{1}w^2_2(t,\cdot)\|_{L^1}+\|w_{1,x}w_1(t,\cdot)\|_{L^1}+\|w_{2,x}w_1(t,\cdot)\|_{L^1}\Big)\\
&=\mathcal{O}(1) \Big(\|w_{1,xx}(t,\cdot)\|_{L^1}+\|w_{2,xx}(t,\cdot)\|_{L^1}+\|u_{1,x}(t,\cdot)\|_{L^1}\|u_{1,xx}(t,\cdot)\|^2_{L^1}\\
&\hspace*{1cm}+\|u_{1,xx}(t,\cdot)\|^2_{L^1} \|w_2(t,\cdot)\|_{L^1}+\|u_{1,xx}(t,\cdot)\|_{L^1}\|w_2(t,\cdot)\|_{L^\infty}\|w_2(t,\cdot)\|_{L^1}\\
&\hspace*{1cm}+\|u_{1,xx}(t,\cdot)\|^2_{L^1}+\|u_{1,xx}(t,\cdot)\|_{L^1}\|w_{2,x}(t,\cdot)\|_{L^1}\Big)\\
&=\mathcal{O}(1) \Big(\|w_{1,xx}(t,\cdot)\|_{L^1}+\|w_{2,xx}(t,\cdot)\|_{L^1}+
\|u_{1,xx}(t,\cdot)\|^2_{L^1}\\
&\hspace{3cm}+\|u_{1,xx}(t,\cdot)\|_{L^1}\|u_{2,xx}(t,\cdot)\|_{L^1}\Big)\\
&=\mathcal{O}(1) \Big(\sum_{i=1}^2\|\widetilde{w}_{i,xx}(t,\cdot)\|_{L^1}+ \sum_{i=1}^2\|u_{i,xx}(t,\cdot)\|_{L^1}\sum_{i=1}^2\|\widetilde{w}_{i,x}(t,\cdot)\|_{L^1}\\
&\hspace*{3.5cm}+\|u_{1,xx}(t,\cdot)\|^2_{L^1}+\|u_{1,xx}(t,\cdot)\|_{L^1}\|u_{2,xx}(t,\cdot)\|_{L^1}\Big)\\
&=\mathcal{O}(1)  \Big( \sum_{i=1}^2\|\widetilde{w}_{i,xx}(t,\cdot)\|_{L^1}+ \left(\sum_{i=1}^2\|u_{i,xx}(t,\cdot)\|_{L^1}\right)^2+
\|u_{1,xx}(t,\cdot)\|^2_{L^1}\\
&\hspace*{1cm}+\|u_{1,xx}(t,\cdot)\|_{L^1}\|u_{2,xx}(t,\cdot)\|_{L^1}\Big)\\
&=\mathcal{O}(1)  \left(\sum_{i=1}^2\|\widetilde{w}_{i,xx}(t,\cdot)\|_{L^1}+\left(\sum_{i=1}^2\|u_{i,xx}(t,\cdot)\|_{L^1}\right)^2\right).
\end{align*}
Again by contradiction we assume that the second inequality in  \eqref{estimate:parabolic-1} is not satisfied, there exists again $t^*<\hat{t}$ such that this estimate is a equality.
From the first inequality in \eqref{estimate:parabolic-1}, we deduce in particular that for $t\in]0,t^*]$ we have for $C_4$ large enough
\begin{equation}\label{estimate-u-xxx-1}
\|u_{i,xxx}(t,\cdot)\|_{L^1}\leq {C}_4   \bigg(\sum_{i=1}^2\|\widetilde{w}_{i,xx}(t,\cdot)\|_{L^1}+\frac{\kappa^2 \kappa_B^{12}\de_0^2}{t}\bigg).
\end{equation}
It remains now to estimate $\sum_{i=1}^2\|\widetilde{w}_{i,xx}(t,\cdot)\|_{L^1}$, again we have
\begin{align}
\widetilde{w}_{i,xx}(t,\cdot)&=G_{i,x}({t}/{2},\cdot)*\widetilde{w}_{i,x}({t}/{2},\cdot)\nonumber\\
&-\frac{1}{2}\int_{\frac{t}{2}}^tG_{i,x}(t-s)*\Big[{\cal H}_i \widetilde{w}_{i,xx}+{\cal H}_{i,x} \widetilde{w}_{i,x}+{\cal G}_{i,x}(\cdot,X_i^{-1})(X_i^{-1})_x+{\cal A}_{i,x} \nonumber\\
&\hspace{3cm}-\widetilde{w}_{i,xx}X_{i,t}(\cdot,X_i^{-1})-\widetilde{w}_{i,x}X_{i,t x}(\cdot,X_i^{-1})(X_i^{-1})_x\Big]ds,\label{estima21}
\end{align}
with
$${\cal H}_i=\alpha_{i}(u(\cdot,X_i^{-1}))^{-\frac{1}{2}}\left[\frac{\pa \alpha_i}{\pa u_1}(u(\cdot,X_i^{-1}))u_{1,x}(\cdot,X_i^{-1})+\frac{\pa \alpha_i}{\pa u_2}(u(\cdot,X_i^{-1}))u_{2,x}(\cdot,X_i^{-1}))\right].$$ It implies in particular using \eqref{estimate:parabolic-1} that for ${C}_5>0$ large enough
\begin{align}
\| \widetilde{w}_{i,xx}(t,\cdot)\|_{L^1}&\leq \frac{C_5  \delta_0}{t} +\frac{1}{2}\int_{\frac{t}{2}}^t\left\|G_{i,x}(t-s)*\big[{\cal H}_i \widetilde{w}_{i,xx}+{\cal H}_{i,x} \widetilde{w}_{i,x}+{\cal A}_{i,x}\right.\nonumber\\
&\hspace*{3.5cm}+{\cal G}_{i,x}(\cdot,X_i^{-1})(X_i^{-1})_x-\widetilde{w}_{i,xx}X_{i,t}(\cdot,X_i^{-1})\nonumber\\
&\hspace*{3.5cm}\left.-\widetilde{w}_{i,x}X_{i,t x}(\cdot,X_i^{-1})(X_i^{-1})_x\big]\right\|_{L^1}ds.\label{estima211}
\end{align}
We observe that
\begin{align}
	\norm{\mathcal{R}}_{L^1}&=\mathcal{O}(1)\left(\norm{u_x}_{L^\f}\norm{u_{xx}}_{L^1}+\norm{u_x}_{L^\f}\norm{u_{x}}_{L^1}+\norm{u_x}^2_{L^\f}\norm{u_{x}}_{L^1}\right),\\
	\norm{\mathcal{R}_x}_{L^1}&=\mathcal{O}(1)\Big(\norm{u_x}_{L^\f}\norm{u_{xxx}}_{L^1}+\norm{u_{xx}}_{L^\f}\norm{u_{xx}}_{L^1}+\norm{u_x}^2_{L^\f}\norm{u_{xx}}_{L^1}\\
	&\hspace{1cm}+\norm{u_x}^3_{L^\f}\norm{u_{x}}_{L^1}+\norm{u_x}_{L^\f}\norm{u_{xx}}_{L^1}+\norm{u_x}^2_{L^\f}\norm{u_{x}}_{L^1}\Big),\\
   \norm{\mathcal{R}_1}_{L^1}&=\mathcal{O}(1)\left(\norm{u_x}_{L^\f}\norm{u_{xx}}_{L^1} +\norm{u_x}^2_{L^\f}\norm{u_{x}}_{L^1}\right),\\
	\norm{\mathcal{R}_{1,x}}_{L^1}&=\mathcal{O}(1)\Big(\norm{u_x}_{L^\f}\norm{u_{xxx}}_{L^1}+\norm{u_{xx}}_{L^\f}\norm{u_{xx}}_{L^1}+\norm{u_x}^2_{L^\f}\norm{u_{xx}}_{L^1}\\
	&\hspace{1.5cm}+\norm{u_x}^3_{L^\f}\norm{u_{x}}_{L^1}\Big),\\
	 \norm{\mathcal{R}_2}_{L^1}&=\mathcal{O}(1) \norm{u_x}_{L^\f}\norm{u_{xx}}_{L^1} ,\\
	\norm{\mathcal{R}_{2,x}}_{L^1}&=\mathcal{O}(1)\Big(\norm{u_x}_{L^\f}\norm{u_{xxx}}_{L^1}+\norm{u_{xx}}_{L^\f}\norm{u_{xx}}_{L^1}+\norm{u_x}^2_{L^\f}\norm{u_{xx}}_{L^1}\Big),\\
	\norm{\mathcal{R}_3}_{L^1}&=\mathcal{O}(1)\left(\norm{u_x}_{L^\f}\norm{u_{xx}}_{L^1}+\norm{u_x}_{L^\f}\norm{u_{x}}_{L^1}+\norm{u_x}^2_{L^\f}\norm{u_{x}}_{L^1}\right),\\
	\norm{\mathcal{R}_{3,x}}_{L^1}&=\mathcal{O}(1)\Big(\norm{u_x}_{L^\f}\norm{u_{xxx}}_{L^1}+\norm{u_{xx}}_{L^\f}\norm{u_{xx}}_{L^1}+\norm{u_x}^2_{L^\f}\norm{u_{xx}}_{L^1}\\
	&\hspace{1cm}+\norm{u_x}^3_{L^\f}\norm{u_{x}}_{L^1}+\norm{u_x}_{L^\f}\norm{u_{xx}}_{L^1}+\norm{u_x}^2_{L^\f}\norm{u_{x}}_{L^1}\Big),\\
	\norm{\mathcal{R}_4}_{L^1}&=\mathcal{O}(1) \norm{u_x}_{L^\f}\norm{u_{x}}_{L^1} ,\\
	\norm{\mathcal{R}_{4,x}}_{L^1}&=\mathcal{O}(1)\Big(\norm{u_x}_{L^\f}\norm{u_{xx}}_{L^1}+\norm{u_x}^2_{L^\f}\norm{u_{x}}_{L^1}\Big).
\end{align}
We note now that using \eqref{estimate:parabolic-1}
\begin{align}
&\|{\cal G}_{i,x}(t,X_i^{-1}(t,\cdot))(X_i^{-1})_x(t,\cdot)\|_{L^1}=\mathcal{O}(1)\|{\cal G}_{i,x}(t,\cdot)\|_{L^1}\nonumber\\
&=\mathcal{O}(1)\big[\norm{u_x}_{L^\f}\norm{\mathcal{R}_1}_{L^1}+\norm{\mathcal{R}_{1,x}}_{L^1}+\norm{u_x}_{L^\f}\norm{\mathcal{R}_2}_{L^1}+\norm{\mathcal{R}_{2,x}}_{L^1}\nonumber\\
&\hspace{1cm}+\norm{\mathcal{R}_{3,x}}_{L^1}+\norm{\mathcal{R}_{4,x}}_{L^1}+\norm{u_x}_{L^\f}\norm{\mathcal{R}}_{L^1}+\norm{\mathcal{R}_{x}}_{L^1}\big]\nonumber\\
&=\mathcal{O}(1)\Big(\norm{u_x}_{L^\f}\norm{u_{xxx}}_{L^1}+\norm{u_{xx}}_{L^\f}\norm{u_{xx}}_{L^1}+\norm{u_x}^2_{L^\f}\norm{u_{xx}}_{L^1}+\norm{u_x}^3_{L^\f}\norm{u_{x}}_{L^1}\nonumber\\
&\hspace{1.5cm}+\norm{u_x}_{L^\f}\norm{u_{xx}}_{L^1}+\norm{u_x}^2_{L^\f}\norm{u_{x}}_{L^1}\Big)\nonumber\\
&=\mathcal{O}(1)\left(\frac{\de_0^2}{t}+\frac{\de_0^2}{t^{3/2}}\right).\label{estima3}
%&\nonumber\\
%&\leq C_6 \kappa_B^7\left(\kappa_B^2\sum\limits_{\tiny\begin{array}{cc}&l_1+l_2+l_3+l_4=4,\\
%		 &i_1,i_2,i_3,i_4\in\{1,2\}\end{array}} \|u_{i_1}^{(l_1)}u_{i_2}^{(l_2)}u_{i_3}^{(l_3)}u_{i_4}^{(l_4)}\|_{L^1}+\kappa_B\kappa_A \sum_{\tiny l_1+l_2=3, i_1,i_2\in\{1,2\}} \|u_{i_1}^{(l_1)}u_{i_2}^{(l_2)}\|_{L^1}\right)\nonumber\\
%&\leq C_7 \kappa_B^9\kappa_A\left( \sum_{l_1+l_2=4, i_1,i_2\in\{1,2\}} \|u_{i_1}^{(l_1)}u_{i_2}^{(l_2)}\|_{L^1}+\frac{ \kappa^3 \kappa_B^{18}\de_0^3}{t^{\frac{3}{2}}}\right).
\end{align}
Similarly we have using \eqref{estimw} and \eqref{2.24},
\begin{align}
&\|\widetilde{w}_{i,xx}(t,\cdot)X_{i,t}(t,X_i^{-1}(t,\cdot))\|_{L^1}\nonumber\\
&\leq \|\widetilde{w}_{i,xx}(t,\cdot)\|_{L^1}\|X_{i,t}(t,X_i^{-1}(t,\cdot))\|_{L^\infty}\nonumber\\
&=\mathcal{O}(1)\left(\delta_0+\sum_{l=1}^2\|u_{l,xx}(t,\cdot)\|_{L^1}\right)\left(\|w_{i,xx}(t,\cdot)\|_{L^1}+\sum_{l=1}^2\|u_{l,xx}(t,\cdot)\|_{L^1}\|w_{i,x}(t,\cdot)\|_{L^1}\right)\nonumber\\
&=\mathcal{O}(1)\left(\delta_0+ \|u_{xx}(t,\cdot)\|_{L^1}\right)\Big[\|u_{xxx}(t,\cdot)\|_{L^1}+\|u_{x}(t,\cdot)\|_{L^\f}\|u_{xx}(t,\cdot)\|_{L^1}\nonumber\\
&+\|u_{x}(t,\cdot)\|_{L^\f}\|u_{xx}(t,\cdot)\|_{L^1}+\norm{u_{xx}(t,\cdot)}_{L^1}^2+\norm{u_{xx}(t,\cdot)}_{L^1}\norm{u_x(t,\cdot)}_{L^\f}\norm{u_x(t,\cdot)}_{L^1}\Big]\nonumber\\
&=\mathcal{O}(1)\left(\delta_0+\frac{\delta_0}{\sqrt{t}}\right)\left(\frac{\delta_0^2}{t}+\sum_{l=1}^2\|u_{l,xxx}(t,\cdot)\|_{L^1}\right).\label{estima4}
%&\leq C_8\kappa_B^9\kappa_A\big(\delta_0+\sum_{l=1}^2\|u_{l,xx}(t,\cdot)\|_{L^1}\big)\Big[\norm{u_{xxx}}_{L^1}+\kappa_B 
%\sum_{\tiny\begin{array}{cc}&l_1+l_2+l_3=2,\\
%		& i_1,i_2,i_3\in\{1,2\}\end{array}}\norm{\pa_x^{l_1}u_{i_1}\pa_x^{l_2}u_{i_2}\pa_x^{l_3}u_{i_3}}_{L^1}\nonumber\\
%&\hspace*{1.5cm}+\sum_{l=1}^2\|u_{l,xx}(t,\cdot)\|_{L^1}\big(\|u_{i,xx}(t,\cdot)\|_{L^1}+\kappa_B \|u_{1,x}u_{2,x}(t,\cdot)\|_{L^1}+\|u_{1,x}^2(t,\cdot)\|_{L^1}\big)\Big]\nonumber\\
%&\leq C_9\left(\delta_0+\frac{\delta_0}{\sqrt{t}}\right)\left(\frac{\delta_0^2}{t}+\sum_{l=1}^2\|u_{l,xxx}(t,\cdot)\|_{L^1}\right).\label{estima4}
\end{align}
Taking derivative of \eqref{eqn-X-i-t} with respect to $x$ we get
\begin{align*}
	X_{i,tx}(x)
	&=-\frac{1}{2}\alpha_i(u(t,x))^{-\frac{3}{2}}\Big[\frac{\pa}{\pa u_1}\alpha_i(u(t,x))\big(-f'(u_1(t,x))u_{1,x}(t,x)+(\alpha_1(u_1)u_{1,x})_x(t,x)\big)
	\nonumber\\
	&\hspace*{2cm}+
	\frac{\pa}{\pa u_2}\alpha_i(u(t,x))\big(-(g(u))_x(t,x)+(\beta(u)u_{1,x}+\alpha_2(u)u_{2,x})_x(t,x)\big)\Big].
\end{align*}
For sufficiently large constant $C_{10},C_{11}>0$ we have
\begin{align*}
	\norm{X_{i,tx}}_{L^1}&=\mathcal{O}(1)\Big[\norm{u_{1,xx}}_{L^1}+\norm{u_{1,x}}_{L^1}\norm{u_{1,x}}_{L^\f}+\norm{u_{1,x}}\norm{u_{2,x}}_{L^\f}\\
	&\hspace*{2cm}+\norm{u_{2,x}}_{L^1}\norm{u_{2,x}}_{L^\f}+\norm{u_{2,xx}}_{L^1}\Big]\\
	&=\mathcal{O}(1)\left[\norm{u_{xx}}_{L^1}+\norm{u_{x}}_{L^1}\right].
\end{align*}
It yields for large enough constants $C_{12},C_{13},C_{14}>0$,
\begin{align}
	&\norm{\widetilde{w}_{i,x}X_{i,t x}(\cdot,X_i^{-1})(X_i^{-1})_x}_{L^1}\nonumber\\
	&=\mathcal{O}(1)\norm{w_{i,x}}_{L^\f}\norm{X_{i,tx}(\cdot,X_i^{-1})}_{L^1}\norm{(X_i^{-1})_x}_{L^\f}\nonumber\\
	&=\mathcal{O}(1)\norm{w_{i,xx}}_{L^1}\norm{X_{i,tx}}_{L^1}\norm{X_{i,xx}}_{L^1}\nonumber\\
	&=\mathcal{O}(1)\left(\frac{\delta_0^2}{t}+\sum_{l=1}^2\|u_{l,xxx}(t,\cdot)\|_{L^1}\right)\left(\norm{u_{1,xx}}_{L^1}+\norm{u_{2,xx}}_{L^1}\right)\norm{u_x}_{L^1}\nonumber\\
	&=\mathcal{O}(1)\left(\frac{\delta_0^2}{t}+\sum_{l=1}^2\|u_{l,xxx}(t,\cdot)\|_{L^1}\right)\frac{\de_0^2}{\sqrt{t}}.\label{est-w-i-x-X-i-t}
\end{align}
From the definition of $\mathcal{H}_i$ we get
\begin{equation*}
	\norm{\mathcal{H}_i}_{L^\f}=\mathcal{O}(1)\norm{u_{xx}}_{L^1}\mbox{ and }\norm{\mathcal{H}_{i,x}}_{L^1}=\mathcal{O}(1)\norm{u_{xx}}_{L^1},
\end{equation*}
for some large constant $C_{15}>0$. Hence, it follows for $C_{16},C_{17}>0$ large enough,
\begin{align}
	&\norm{{\cal H}_i \widetilde{w}_{i,xx}}_{L^1}+\norm{{\cal H}_{i,x} \widetilde{w}_{i,x}}_{L^1} \nonumber\\
	&\leq \norm{\mathcal{H}_i}_{L^\f}\norm{\widetilde{w}_{i,xx}}_{L^1}+\norm{\mathcal{H}_{i,x}}_{L^1}\norm{\widetilde{w}_{i,x}}_{L^\f}\nonumber\\
	&=\mathcal{O}(1)\norm{u_{xx}}_{L^1}\norm{w_{i,xx}}_{L^1}\nonumber\\
	&=\mathcal{O}(1)\norm{u_{xx}}_{L^1}\left(\frac{\delta_0^2}{t}+\sum_{l=1}^2\|u_{l,xxx}(t,\cdot)\|_{L^1}\right)\nonumber\\
	&=\mathcal{O}(1)\left(\frac{\delta_0^2}{t}+\sum_{l=1}^2\|u_{l,xxx}(t,\cdot)\|_{L^1}\right)\frac{\de_0}{\sqrt{t}}.\label{est-Hw-i}
\end{align}
Similarly, from \eqref{defn:A-i} we have
\begin{align}
	\norm{\mathcal{A}_{i,x}}_{L^1}&=\mathcal{O}(1)\norm{u_x}_{L^\f}\norm{\widetilde{w}_{i,xx}}_{L^1}\nonumber\\
		&=\mathcal{O}(1)\norm{u_{xx}}_{L^1}\left(\frac{\delta_0^2}{t}+\sum_{l=1}^2\|u_{l,xxx}(t,\cdot)\|_{L^1}\right)\nonumber\\
		&=\mathcal{O}(1)\left(\frac{\delta_0^2}{t}+\sum_{l=1}^2\|u_{l,xxx}(t,\cdot)\|_{L^1}\right)\frac{\de_0}{\sqrt{t}}.\label{est-A-i}
\end{align}
From \eqref{estimate-u-xxx-1}, \eqref{estima211}, \eqref{estima3}, \eqref{estima4}, \eqref{est-w-i-x-X-i-t}, \eqref{est-Hw-i} and \eqref{est-A-i} we get for $C_{6}>0$ large enough
\begin{align}
\sum_{i=1}^2 \|u_{i,xxx}(t,\cdot)\|_{L^1} &\leq \left(\frac{C_2\delta_0}{2t}+C_{6}\int_{\frac{t}{2}}^t\frac{1}{\sqrt{t-s}}\left(\frac{\delta_0^2}{s}+\frac{\delta_0^2}{s^{\frac{3}{2}}}\right)ds\right)\nonumber\\
&\leq \big(\frac{C_2\delta_0}{2t}+\frac{2\sqrt{2}C_{6}\delta_0^2}{\sqrt{t}}+\frac{4C_{6}\delta_0^2}{t})\leq\frac{C_2}{t},\label{estima6}
\end{align}
if we take $C_2\geq 16C_{6}$ sufficiently large and such that $\delta_0\leq\frac{1}{\sqrt{\hat{t}}}$ which gives the result. 

	  We apply similar arguments to prove the fourth order estimate of $u_{xxxx}$ as in \eqref{estimate:parabolic-1}. From \eqref{ux3} and \eqref{eqn:w-i-xxx} we have
	  \begin{align*}
	  	\|u_{i,xxxx}(t,\cdot)\|_{L^1}&=\mathcal{O}(1)\bigg(\norm{w_{xxx}}_{L^1}+\norm{u_x}_{L^\f}\norm{u_{xxx}}_{L^1}+\norm{u_x}^2_{L^\f}\norm{u_{xx}}_{L^1}\\
	  	&\hspace{1cm} +\norm{u_x}_{L^\f}^3\norm{u_x}_{L^1}+\norm{u_{xx}}_{L^\f}\norm{u_{xx}}_{L^1}\bigg)\\
	  	&=\mathcal{O}(1)\bigg(\norm{w_{xxx}}_{L^1}+\norm{u_{xx}}_{L^1}\norm{u_{xxx}}_{L^1}+\norm{u_{xx}}^2_{L^1}\norm{u_{xx}}_{L^1}\\
	  	&\hspace{1cm} +\norm{u_{xx}}_{L^1}^3\norm{u_x}_{L^1}+\norm{u_{xxx}}_{L^1}\norm{u_{xx}}_{L^1}\bigg)\\
	  	&=\mathcal{O}(1)\bigg(\norm{\widetilde{w}_{1,xxx}}_{L^1}+\norm{\widetilde{w}_{2,xxx}}_{L^1}+\norm{u_{x}}_{L^\f}(\norm{\widetilde{w}_{1,xx}}_{L^1}+\norm{\widetilde{w}_{2,xx}}_{L^1})\\
	  	&\hspace{1cm} +(\norm{u_{xx}}_{L^\f}+\norm{u_{x}}_{L^\f}^2) (\norm{\widetilde{w}_{1,x}}_{L^1}+\norm{\widetilde{w}_{2,x}}_{L^1})\\
	  	&\hspace{1cm} +\norm{u_{xx}}_{L^1}\norm{u_{xxx}}_{L^1}+\norm{u_{xx}}^2_{L^1}\norm{u_{xx}}_{L^1}\\
	  	& \hspace{1cm} +\norm{u_{xx}}_{L^1}^3\norm{u_x}_{L^1}+\norm{u_{xxx}}_{L^1}\norm{u_{xx}}_{L^1}\bigg).
	  \end{align*}
  From \eqref{estimw} we have
  \begin{align*}
  	\norm{\widetilde{w}_{i,x}}_{L^1}&=\mathcal{O}(1)\norm{{w}_{i,x}}_{L^1}=\mathcal{O}(1)(\norm{u_{xx}}_{L^1}+\norm{u_x}_{L^\f}\norm{u_x}_{L^1}),\\
  	\norm{\widetilde{w}_{i,xx}}_{L^1}&=\mathcal{O}(1)\left(\norm{w_{i,xx}}_{L^1}+\norm{u_x}_{L^\f}\norm{w_{i,x}}_{L^1}\right)\\
  	&=\mathcal{O}(1)\left(\norm{u_{xxx}}_{L^1}+\norm{u_x}_{L^\f}\norm{u_{xx}}_{L^1}+\norm{u_x}^2_{L^\f}\norm{u_{x}}_{L^1}\right).
  \end{align*}
  Then we obtain
  \begin{align}
  	\|u_{i,xxxx}(t,\cdot)\|_{L^1}&=\mathcal{O}(1)\bigg(\norm{\widetilde{w}_{1,xxx}}_{L^1}+\norm{\widetilde{w}_{2,xxx}}_{L^1}+\norm{u_{x}}_{L^\f}\norm{u_{xxx}}_{L^1}+\norm{u_x}^2_{L^\f}\norm{u_{xx}}_{L^1}\nonumber\\
  	&\hspace{1.5cm} +(\norm{u_{xx}}_{L^\f}+\norm{u_{x}}_{L^\f}^2) \norm{u_{xx}}_{L^1}+\norm{u_{xx}}_{L^1}\norm{u_{xxx}}_{L^1}\nonumber\\
  	& \hspace{1.5cm} +\norm{u_{xx}}^2_{L^1}\norm{u_{xx}}_{L^1}+\norm{u_{xx}}_{L^1}^3\norm{u_x}_{L^1}+\norm{u_{xxx}}_{L^1}\norm{u_{xx}}_{L^1}\bigg)\nonumber\\
  	&=\mathcal{O}(1)\bigg(\norm{\widetilde{w}_{1,xxx}}_{L^1}+\norm{\widetilde{w}_{2,xxx}}_{L^1}+\frac{\de_0^2}{t^{3/2}}\bigg).
  \end{align}
Similarly, we have
\begin{equation*}
	\norm{\widetilde{w}_{xxx}}_{L^1}=\mathcal{O}(1)\left(\norm{u_{xxxx}}_{L^1}+\frac{\de_0^2}{t^{3/2}}\right).
\end{equation*}
To estimate $\sum_{i=1}^2\|\widetilde{w}_{i,xx}(t,\cdot)\|_{L^1}$, again we have
\begin{align}
	\widetilde{w}_{i,xxx}(t,\cdot)&=G_{i,x}({t}/{2},\cdot)*\widetilde{w}_{i,xxx}({t}/{2},\cdot)\nonumber\\
	&-\frac{1}{2}\int_{\frac{t}{2}}^tG_{i,x}(t-s)*\Big[{\cal H}_i \widetilde{w}_{i,xxx}+2{\cal H}_{i,x} \widetilde{w}_{i,xx} +{\cal H}_{i,xx} \widetilde{w}_{i,x}+2{\cal A}_{i,xx}\nonumber\\
	&\hspace{3cm}+{\cal G}_{i,xx}(\cdot,X_i^{-1})(X_i^{-1})^2_x+{\cal G}_{i,x}(\cdot,X_i^{-1})(X_i^{-1})_{xx}\nonumber\\
	&\hspace{3cm}-\widetilde{w}_{i,xxx}X_{i,t}(\cdot,X_i^{-1})-2\widetilde{w}_{i,xx}X_{i,t x}(\cdot,X_i^{-1})(X_i^{-1})_x \nonumber\\
	&\hspace{2.5cm}-\widetilde{w}_{i,x}X_{i,t xx}(\cdot,X_i^{-1})(X_i^{-1})^2_x-\widetilde{w}_{i,x}X_{i,t x}(\cdot,X_i^{-1})(X_i^{-1})_{xx}\Big]ds.\label{estima31}
\end{align}
We calculate
\begin{align*}
	{\cal G}_{xx}&=u_{xx}\cdot DB_1(u){\cal R}_1+(u_x\otimes u_x):D^2B_1(u){\cal R}_1+2u_x\cdot DB_1(u){\cal R}_{1,x}+B_1(u){\cal R}_{1,xx}\\
	&\quad +u_{xx}\cdot DP(u){\cal R}_2+(u_x\otimes u_x):D^2P(u){\cal R}_2+2u_x\cdot DP(u){\cal R}_{2,x}+P(u){\cal R}_{2,xx}\\
	&\quad +{\cal R}_{3,xx}-{\cal R}_{4,xx}+u_{xx}\cdot DP(u){\cal R}+(u_x\otimes u_x):D^2P(u){\cal R}\\
	&\quad +2u_x\cdot DP(u){\cal R}_{x}+P(u){\cal R}_{xx}.
\end{align*}
Note that
\begin{align*}
	{\cal R}_{1,xx}&=\begin{pmatrix}
		0\\
		2\frac{\pa }{\pa x}(\gamma(u))u_{1,xxxx}+4\frac{\pa^2 }{\pa x^2}(\gamma(u))u_{1,xxx}+2\frac{\pa^3 }{\pa x^3}(\gamma(u))u_{1,xx}
	\end{pmatrix}\\
&+\begin{pmatrix}
	0\\
	\frac{\pa ^4}{\pa x^4}(\gamma(u))u_{1,x}+2\frac{\pa ^3}{\pa x^3}(\gamma(u))u_{1,xx}+\frac{\pa ^2}{\pa^2 x}(\gamma(u))u_{1,xxx}
\end{pmatrix}\\
&=\mathcal{O}(1)\left(\abs{u_{x}}\abs{u_{1,xxxx}}+\abs{u_{xx}}\abs{u_{1,xxx}}+\abs{u_{x}}^2\abs{u_{1,xxx}}+\abs{u_{xxx}}\abs{u_{1,xx}}\right)\\
&+\mathcal{O}(1)\left(\abs{u_{xx}}\abs{u_x}\abs{u_{1,xx}}+\abs{u_x}^3\abs{u_{1,xx}}+\abs{u_{xxxx}}\abs{u_{1,x}}+\abs{u_{xxx}}\abs{u_{x}}\abs{u_{1,x}}\right)\\
&+\mathcal{O}(1)\left(\abs{u_{xx}}\abs{u_{x}}^2\abs{u_{1,x}}+\abs{u_{xx}}^2\abs{u_{1,x}}+\abs{u_x}^4\abs{u_{1,x}}\right),\\
	{\cal R}_{2,xx}&=\begin{pmatrix}
		\alpha^{\p\p\p}_1(u_1)(u_{1,x}^2+u_{1,xx})u_{1,x} u_{1,xx}+2\alpha^{\p\p}_1(u_1)u_{1,x}[u_{1,x} u_{1,xxx}+u_{1,xx}^2]\\
		\frac{\pa^3}{\pa x^3}(\beta(u))u_{1,xx}+\frac{\pa^3 }{\pa x^3}(\alpha_2(u))u_{2,xx}+2\frac{\pa^2}{\pa x^2}(\beta(u))u_{1,xxx}+2\frac{\pa^2 }{\pa x^2}(\alpha_2(u))u_{2,xxx}
	\end{pmatrix}\\
&+\begin{pmatrix}
	\alpha^{\p}_1(u_1)[3u_{1,xxx} u_{1,xx} +u_{1,x} u_{1,xxxx}]+\alpha^{\p}_1(u_1)  [3u_{1,xxx} u_{1,xx} +u_{1,x} u_{1,xxxx}]\\
	\frac{\pa}{\pa x}(\beta(u))u_{1,xxxx}+\frac{\pa }{\pa x}(\alpha_2(u))u_{2,xxxx}
\end{pmatrix}\\	
&=\mathcal{O}(1)\left(\abs{u_{x}}\abs{u_{xxxx}}+\abs{u_{xx}}\abs{u_{xxx}}+\abs{u_{x}}^2\abs{u_{xxx}}+\abs{u_{xx}}^2\abs{u_x} +\abs{u_x}^3\abs{u_{xx}}\right),\\
 {\cal R}_{3,xx}&=\begin{pmatrix}
		0\\
		\frac{\pa^3 }{\pa t\pa x^2}(\gamma(u))u_{1,x}+2\frac{\pa^2 }{\pa t\pa x}(\gamma(u))u_{1,xx}+\frac{\pa }{\pa t}(\gamma(u))u_{1,xxx}
	\end{pmatrix}\\
&=\mathcal{O}(1)\left(\abs{u_{txx}}\abs{u_x}+\abs{u_{tx}}\abs{u_x}^2+\abs{u_t}\abs{u_{xx}}\abs{u_x}+\abs{u_t}\abs{u_x}^3+\abs{u_{tx}}\abs{u_{xx}}+\abs{u_t}\abs{u_{xxx}}\right)\\
&=\mathcal{O}(1)\left(\abs{u_{xxxx}}\abs{u_x}+\abs{u_{xxx}}\abs{u_{x}}^2+\abs{u_{xx}}^2\abs{u_x}+\abs{u_{xx}}\abs{u_x}^3+\abs{u_x}^4+\abs{u_{xxx}}\abs{u_{x}}\right)\\
&+\mathcal{O}(1)\left(\abs{u_{xx}}\abs{u_x}^2+\abs{u_x}^3+\abs{u_{xxx}}\abs{u_x}^2+\abs{u_{xx}}\abs{u_x}^3+\abs{u_x}^5+\abs{u_{xx}}\abs{u_x}^2+\abs{u_x}^4\right)\\
&+\mathcal{O}(1)\left( \abs{u_{xx}}^2\abs{u_x}+\abs{u_{xx}}\abs{u_x}^3+\abs{u_{xx}}\abs{u_x}^2+\abs{u_{xx}}\abs{u_x}^3+\abs{u_x}^5++\abs{u_x}^4\right)\\
&+\mathcal{O}\left(\abs{u_{xxx}}\abs{u_{xx}}+\abs{u_{xx}}^2\abs{u_x}+\abs{u_{xx}}\abs{u_x}^3+\abs{u_{xx}}^2+\abs{u_{xx}}\abs{u_x}^2+\abs{u_x}\abs{u_{xxx}}\right)\\
&+\mathcal{O}(1)\left(\abs{u_{xx}}\abs{u_{xxx}}+\abs{u_x}^2\abs{u_{xxx}}
\right),\\
	{\cal R}_{4,xx}&=[u_x\otimes u_x:D^2A_1(u)+u_{xx}\cdot DA_1(u)]\begin{pmatrix}
		0\\
		\frac{\pa }{\pa x}(\gamma(u))u_{1,x}
	\end{pmatrix}\\
&+u_{x}\cdot DA_1(u)\begin{pmatrix}
	0\\
	\frac{\pa^2 }{\pa x^2}(\gamma(u))u_{1,x}+\frac{\pa }{\pa x}(\gamma(u))u_{1,xx}
\end{pmatrix}\\
&+ A_1(u)\begin{pmatrix}
	0\\
	\frac{\pa^3 }{\pa x^3}(\gamma(u))u_{1,x}+2\frac{\pa^2 }{\pa x^2}(\gamma(u))u_{1,xx}+\frac{\pa }{\pa x}(\gamma(u))u_{1,xxx}
\end{pmatrix}\\
&=\mathcal{O}(1)\left(\abs{u_x}^4+\abs{u_x}^2\abs{u_{xx}}+\abs{u_{xxx}}\abs{u_x}+\abs{u_{xx}}^2\right),\\
{\cal R}_{xx}&=\begin{pmatrix}
	(\alpha_1'(u_1)(u_{1,x})^2)_{xxx}-\left(f''(u_1)(u_{1,x})^2\right)_{xx}\\
	\left(\frac{\pa }{\pa x}(\beta(u))u_{1,x}\right)_{xxx}+\left(\frac{\pa }{\pa x}(\alpha_2(u))u_{2,x}\right)_{xxx}
\end{pmatrix}\\
&=\mathcal{O}(1)\left(\abs{u_x}^5+\abs{u_{xx}}\abs{u_{x}}^3+\abs{u_{xxx}}\abs{u_x}^2+\abs{u_{xxxx}}\abs{u_x}+\abs{u_{xxx}}\abs{u_{xx}}+\abs{u_x}\abs{u_{xx}}^2\right).
\end{align*}
Subsequently, we get
\begin{align*}
	&\norm{{\cal G}_{i,xx}(\cdot,X_i^{-1})(X_i^{-1})^2_x}_{L^1}\\
	&=\mathcal{O}(1)\left(\norm{u_{xx}}_{L^\f}\norm{\mathcal{R}_1}_{L^1}+\norm{u_{x}}^2_{L^\f}\norm{\mathcal{R}_1}_{L^1}+\norm{u_{x}}_{L^\f}\norm{\mathcal{R}_{1,x}}_{L^1}+\norm{\mathcal{R}_{1,xx}}_{L^1}\right)\\
	&+\mathcal{O}(1)\left(\norm{u_{xx}}_{L^\f}\norm{\mathcal{R}_2}_{L^1}+\norm{u_{x}}^2_{L^\f}\norm{\mathcal{R}_2}_{L^1}+\norm{u_{x}}_{L^\f}\norm{\mathcal{R}_{2,x}}_{L^1}+\norm{\mathcal{R}_{2,xx}}_{L^1}\right)\\
	&+\mathcal{O}(1)\left(\norm{u_{xx}}_{L^\f}\norm{\mathcal{R}}_{L^1}+\norm{u_{x}}^2_{L^\f}\norm{\mathcal{R}}_{L^1}+\norm{u_{x}}_{L^\f}\norm{\mathcal{R}_{x}}_{L^1}+\norm{\mathcal{R}_{xx}}_{L^1}+\norm{\mathcal{R}_{3,xx}}_{L^1}+\norm{\mathcal{R}_{4,xx}}_{L^1}\right).
\end{align*}
Then we get
\begin{align*}
	&\norm{{\cal G}_{i,xx}(\cdot,X_i^{-1})(X_i^{-1})^2_x}_{L^1}\\
	&=\mathcal{O}(1)\left(\norm{u_x}_{L^\f}\norm{u_{xx}}_{L^1} +\norm{u_x}^2_{L^\f}\norm{u_{x}}_{L^1}\right)\left[\norm{u_{xx}}_{L^\f}+\norm{u_x}_{L^\f}^2\right]\\
	&+\mathcal{O}(1)\Big(\norm{u_x}^2_{L^\f}\norm{u_{xxx}}_{L^1}+\norm{u_x}_{L^\f}\norm{u_{xx}}_{L^\f}\norm{u_{xx}}_{L^1}+\norm{u_x}^3_{L^\f}\norm{u_{xx}}_{L^1}+\norm{u_x}^4_{L^\f}\norm{u_{x}}_{L^1}\Big)\\
	&+\mathcal{O}(1)\left(\norm{u_{x}}_{L^\f}\norm{u_{xxxx}}_{L^1}+\norm{u_{xx}}_{L^\f}\abs{u_{xxx}}_{L^1}+\norm{u_{x}}_{L^\f}^2\norm{u_{xxx}}_{L^1}+\norm{u_{xxx}}_{L^1}\norm{u_{xx}}_{L^\f}\right)\\
	&+\mathcal{O}(1)\left(\norm{u_{xx}}_{L^1}\norm{u_x}_{L^\f}\norm{u_{xx}}_{L^\f}+\norm{u_x}_{L^\f}^3\norm{u_{xx}}_{L^1}+\norm{u_{xxxx}}_{L^1}\norm{u_{x}}_{L^\f}+\norm{u_{xxx}}_{L^1}\norm{u_{x}}^2_{L^\f}\right)\\
	&+\mathcal{O}(1)\left(\norm{u_{xx}}_{L^1}\norm{u_{x}}_{L^\f}^2\norm{u_{x}}_{L^\f}+\norm{u_{xx}}_{L^\f}^2\norm{u_{x}}_{L^1}+\norm{u_x}_{L^\f}^4\norm{u_{x}}_{L^1}\right)\\
	&+\mathcal{O}(1) \norm{u_x}_{L^\f}\norm{u_{xx}}_{L^1}\left[\norm{u_{xx}}_{L^\f}+\norm{u_x}_{L^\f}^2\right]\\
	&+\mathcal{O}(1)\Big(\norm{u_x}^2_{L^\f}\norm{u_{xxx}}_{L^1}+\norm{u_x}_{L^\f}\norm{u_{xx}}_{L^\f}\norm{u_{xx}}_{L^1}+\norm{u_x}^3_{L^\f}\norm{u_{xx}}_{L^1}\Big)\\
	&+\mathcal{O}(1)\left(\norm{u_{x}}_{L^\f}\norm{u_{xxxx}}_{L^1}+\norm{u_{xx}}_{L^\f}\norm{u_{xxx}}_{L^1}+\norm{u_{x}}_{L^\f}^2\norm{u_{xxx}}_{L^1}\right)\\
	&+\mathcal{O}(1)\left(\norm{u_{xx}}_{L^\f}^2\norm{u_x}_{L^1} +\norm{u_x}_{L^\f}^3\norm{u_{xx}}_{L^1}\right)\\
	&+\mathcal{O}(1)\left(\norm{u_x}_{L^\f}\norm{u_{xx}}_{L^1}+\norm{u_x}_{L^\f}\norm{u_{x}}_{L^1}+\norm{u_x}^2_{L^\f}\norm{u_{x}}_{L^1}\right)\left[\norm{u_{xx}}_{L^\f}+\norm{u_x}_{L^\f}^2\right]\\
	&+\mathcal{O}(1)\Big(\norm{u_x}^2_{L^\f}\norm{u_{xxx}}_{L^1}+\norm{u_x}_{L^\f}\norm{u_{xx}}_{L^\f}\norm{u_{xx}}_{L^1}+\norm{u_x}^3_{L^\f}\norm{u_{xx}}_{L^1}\\
	&\hspace{1cm}+\norm{u_x}^4_{L^\f}\norm{u_{x}}_{L^1}+\norm{u_x}^2_{L^\f}\norm{u_{xx}}_{L^1}+\norm{u_x}^3_{L^\f}\norm{u_{x}}_{L^1}\Big)\\
	&+\mathcal{O}(1)\left(\norm{u_x}_{L^\f}^4\norm{u_x}_{L^1}+\norm{u_{xx}}_{L^1}\norm{u_{x}}_{L^\f}^3+\norm{u_{xxx}}_{L^1}\norm{u_x}_{L^\f}^2+\norm{u_{xxxx}}_{L^1}\norm{u_x}_{L^\f}\right)\\
	&+\mathcal{O}(1)\left(\norm{u_{xxx}}_{L^1}\norm{u_{xx}}_{L^\f}+\norm{u_x}_{L^1}\norm{u_{xx}}_{L^\f}^2\right)\\
	&+\mathcal{O}(1)\Big(\norm{u_{xxxx}}_{L^1}\norm{u_x}_{L^\f}+\norm{u_{xxx}}_{L^1}\norm{u_{x}}_{L^\f}^2+\norm{u_{xx}}_{L^\f}^2\norm{u_x}_{L^1}+\norm{u_{xx}}_{L^1}\norm{u_x}_{L^\f}^3\\
	&\hspace{1cm}+\norm{u_x}_{L^\f}^3\norm{u_x}_{L^1}+\norm{u_{xxx}}_{L^1}\norm{u_{x}}_{L^\f}\Big)\\
	&+\mathcal{O}(1)\Big(\norm{u_{xx}}_{L^1}\norm{u_x}_{L^\f}^2+\norm{u_x}_{L^\f}^2\norm{u_x}_{L^1}+\norm{u_{xxx}}_{L^1}\norm{u_x}_{L^\f}^2+\norm{u_{xx}}_{L^1}\norm{u_x}_{L^\f}^3\\
	&\hspace{1cm}+\norm{u_x}_{L^\f}^4\norm{u_x}_{L^1}+\norm{u_{xx}}_{L^\f}\norm{u_x}_{L^\f}^2+\norm{u_x}^3_{L^\f}\norm{u_x}_{L^1}\Big)\\
	&+\mathcal{O}(1)\Big( \norm{u_{xx}}_{L^\f}^2\norm{u_x}_{L^1}+\norm{u_{xx}}_{L^1}\norm{u_x}_{L^\f}^3+\norm{u_{xx}}_{L^1}\norm{u_x}_{L^\f}^2+\norm{u_{xx}}_{L^1}\norm{u_x}_{L^\f}^3\\
	&\hspace{1cm}+\norm{u_x}_{L^\f}^4\norm{u_x}_{L^1}+\norm{u_x}_{L^\f}^3\norm{u_x}_{L^1}\Big)\\
	&+\mathcal{O}(1)\Big(\norm{u_{xxx}}_{L^1}\norm{u_{xx}}_{L^\f}+\norm{u_{xx}}_{L^\f}^2\norm{u_x}_{L^1}+\norm{u_{xx}}_{L^1}\norm{u_x}_{L^\f}^3+\norm{u_{xx}}_{L^1}\norm{u_{xx}}_{L^\f}\\
	&\hspace{1cm}+\norm{u_{xx}}_{L^1}\norm{u_x}_{L^\f}^2+\norm{u_x}_{L^\f}\norm{u_{xxx}}_{L^1}\Big)\\
	&+\mathcal{O}(1)\left(\norm{u_{xx}}_{L^\f}\norm{u_{xxx}}_{L^1}+\norm{u_x}_{L^\f}^2\norm{u_{xxx}}_{L^1}\right)\\
	&+\mathcal{O}(1)\left(\norm{u_x}^3_{L^\f}\norm{u_x}_{L^1}+\norm{u_x}_{L^\f}^2\norm{u_{xx}}_{L^1}+\norm{u_{xxx}}_{L^1}\norm{u_x}_{L^\f}+\norm{u_{xx}}_{L^\f}\norm{u_{xx}}_{L^1}\right),
\end{align*}
which gives
\begin{equation}\label{estimate-3-1}
	\norm{{\cal G}_{i,xx}(\cdot,X_i^{-1})(X_i^{-1})^2_x}_{L^1}=\mathcal{O}(1)\left( \frac{\de^2_0}{t^{3/2}}+\frac{\de^2_0}{t^{2}}\right).
\end{equation}
From the definition of $\mathcal{H}_i$ we have
\begin{equation*}
	\norm{\mathcal{H}_{i,xx}}_{L^1}=\mathcal{O}(1)\left(\norm{u_{xxx}}_{L^1}+\norm{u_{x}}_{L^\f}\norm{u_{xx}}_{L^1}+\norm{u_x}_{L^\f}^2\norm{u_x}_{L^1}\right)=\mathcal{O}(1)\frac{\de_0}{t}.
\end{equation*}
Then we get
\begin{align}
		&\norm{{\cal H}_i \widetilde{w}_{i,xxx}}_{L^1}+2\norm{{\cal H}_{i,x} \widetilde{w}_{i,xx}}_{L^1} +\norm{{\cal H}_{i,xx} \widetilde{w}_{i,x}}_{L^1} \nonumber\\
		&\leq \norm{{\cal H}_i}_{L^\f}\norm{\widetilde{w}_{i,xxx}}_{L^1}+2\norm{{\cal H}_{i,x}}_{L^\f}\norm{\widetilde{w}_{i,xx}}_{L^1} +\norm{{\cal H}_{i,xx}}_{L^1}\norm{ \widetilde{w}_{i,x}}_{L^\f}\nonumber\\
		&\leq\mathcal{O}(1)\frac{\de_0}{\sqrt{t}}\left(\norm{u_{xxxx}}_{L^1}+\frac{\de_0^2}{t^{3/2}}\right)+3\norm{\mathcal{H}_{i,xx}}_{L^1}\norm{\widetilde{w}_{xx}}_{L^1}\nonumber\\
		&=\mathcal{O}(1)\frac{\de_0}{\sqrt{t}}\left(\norm{u_{xxxx}}_{L^1}+\frac{\de_0^2}{t^{3/2}}\right)+\mathcal{O}(1)\frac{\de_0}{t}\left(\norm{u_{xxx}}_{L^1}+\norm{u_x}_{L^\f}\norm{u_{xx}}_{L^1}+\norm{u_x}^2_{L^\f}\norm{u_{x}}_{L^1}\right)\nonumber\\
		&=\mathcal{O}(1)\frac{\de_0}{\sqrt{t}}\left(\norm{u_{xxxx}}_{L^1}+\frac{\de_0^2}{t^{3/2}}\right).\label{estimate-3-2}
\end{align}
Furthermore, we estimate
\begin{align}
	\norm{{\cal G}_{i,x}(\cdot,X_i^{-1})(X_i^{-1})_{xx}}_{L^1}&=\mathcal{O}(1)\norm{\mathcal{G}_{i,xx}}_{L^1}\norm{u_x}_{L^1} \nonumber\\
	&=\mathcal{O}(1)\left(\frac{\de_0^3}{t^{3/2}}+\frac{\de_0^3}{t^2}\right), \label{estimate-3-3}\\
	\norm{\widetilde{w}_{i,xxx}X_{i,t}(\cdot,X_i^{-1})}_{L^1}&=\norm{\widetilde{w}_{i,xxx}}_{L^1}\norm{X_{i,t}(\cdot,X_i^{-1})}_{L^\f}\nonumber\\
	&=\mathcal{O}(1)\left(\norm{u_{xxxx}}_{L^1}+\frac{\de_0^2}{t^{3/2}}\right)\left(\de_0+\frac{\de_0}{\sqrt{t}}\right), \label{estimate-3-4}\\
	\norm{\widetilde{w}_{i,xx}X_{i,t x}(\cdot,X_i^{-1})(X_i^{-1})_x}_{L^1}&=\mathcal{O}(1)\norm{\widetilde{w}_{i,xx}}_{L^1}\norm{X_{i,t x}(\cdot,X_i^{-1})}_{L^\f}\norm{(X_i^{-1})_x}_{L^\f}\nonumber\\
	&=\mathcal{O}(1)\norm{\widetilde{w}_{i,xx}}_{L^1}\norm{X_{i,t x}(\cdot,X_i^{-1})}_{L^\f}\nonumber\\
	&=\mathcal{O}(1)\left(\norm{u_{xxx}}_{L^1}+\frac{\de_0^2}{t}\right)\frac{\de_0}{\sqrt{t}},\label{estimate-3-5}\\
	\norm{\widetilde{w}_{i,x}X_{i,t xx}(\cdot,X_i^{-1})(X_i^{-1})^2_x}_{L^1}&\leq \norm{\widetilde{w}_{i,x}}_{L^\f}\norm{X_{i,t xx}(\cdot,X_i^{-1})}_{L^1}\norm{(X_i^{-1})^2_x}_{L^\f}\nonumber\\
	&\leq \norm{\widetilde{w}_{i,xx}}_{L^1}\norm{X_{i,t xx}(\cdot,X_i^{-1})}_{L^1}\nonumber\\
	&=\mathcal{O}(1)\frac{\de_0}{t}\left(\norm{u_{xxx}}_{L^1}+\norm{u_{xx}}_{L^1}+\norm{u_x}_{L^\f}^2\right)\nonumber\\
	&=\mathcal{O}(1)\frac{\de_0}{t^2},\label{estimate-3-6}\\
	\norm{\widetilde{w}_{i,x}X_{i,t x}(\cdot,X_i^{-1})(X_i^{-1})_{xx}}_{L^1}
	&\leq 	\norm{\widetilde{w}_{i,x}}_{L^\f}\norm{X_{i,t x}(\cdot,X_i^{-1})}_{L^1}\norm{(X_i^{-1})_{xx}}_{L^\f}\nonumber\\
	&=\mathcal{O}(1)\norm{\widetilde{w}_{i,xx}}_{L^1}(\norm{u_{xx}}_{L^1}+\norm{u_x}_{L^1})\norm{u_x}_{L^\f}\nonumber\\
	&=\mathcal{O}(1)\frac{\de_0^2}{t^2}.\label{estimate-3-7}
\end{align}
We also have
\begin{align}	
	\norm{{\cal A}_{i,xx}}_{L^1}&=\mathcal{O}(1)\bigg(\norm{u_{xx}}_{L^1}\norm{\widetilde{w}_x}_{L^\f}+\norm{u_{x}}_{L^\f}^2\norm{\widetilde{w}_x}_{L^1}+\norm{u_x}_{L^\f}\norm{\widetilde{w}_{xx}}_{L^1}\nonumber\\
	&\hspace{4cm}+\norm{u-u^*}_{L^\f}\norm{\widetilde{w}_{xxx}}_{L^1}\bigg)\nonumber\\
	&=\mathcal{O}(1)\left(\de_0\norm{u_{xxxx}}_{L^1}+\frac{\de_0^2}{t^{3/2}}\right).\label{estimate-3-8}
\end{align}
From \eqref{estima31}, \eqref{estimate-3-1}, \eqref{estimate-3-2}, \eqref{estimate-3-3}, \eqref{estimate-3-4}, \eqref{estimate-3-5}, \eqref{estimate-3-6}, \eqref{estimate-3-7}, and \eqref{estimate-3-8} we get for $C_{7}>0$ large enough
\begin{align}
	\sum_{i=1}^2 \|u_{i,xxx}(t,\cdot)\|_{L^1} &\leq \left(\frac{C_3\delta_0}{2t^{3/2}}+C_{7}\int_{\frac{t}{2}}^t\frac{1}{\sqrt{t-s}}\left(\frac{\delta_0^2}{s^2}+\frac{\delta_0^2}{s^{3/2}}\right)ds\right)\nonumber\\
	&\leq  \bigg(\frac{C_3\delta_0}{2t^{3/2}}+\frac{4\sqrt{2}C_{7}\delta_0^2}{t^{3/2}}+\frac{4C_{7}\delta_0^2}{t})\leq\frac{C_3}{t^{3/2}},\label{estima-7}
\end{align}
	if $C_3\geq 32C_{7}$ is chosen sufficiently large and such that $\delta_0\leq\frac{1}{\sqrt{\hat{t}}}$ which gives the result. This completes the proof of Proposition \ref{prop:parabolic}.
\end{proof}

\begin{corollary}
\label{coro2.2} Let $T>\hat{t}$. Assume that the solution $u$ of  \eqref{eqn-main} satisfies on $[0,T]$
$$\|u_{1,x}(t,\cdot)\|_{L^1}+\|u_{2,x}(t,\cdot)\|_{L^1}\leq \delta_0,$$
then for all $t\in[\hat{t},T]$ we have
\begin{equation}%\label{estimate:parabolic-2}
	\begin{aligned}
		&\sum_{i=1}^2\norm{u_{i,xx}(t,\cdot)}_{L^1}, \sum_{i=1}^2\norm{u_{i,x}(t,\cdot)}_{L^\infty},=O(1)\delta_0^2\\
		%\frac{16 \kappa \kappa_B^6\de_0}{\sqrt{t}},\\
		&\sum_{i=1}^2 \norm{u_{i,xxx}(t,\cdot)}_{L^1},\sum_{i=1}^2 \norm{u_{i,xx}(t,\cdot)}_{L^\infty}=O(1)\delta_0^3\\%\frac{2500\kappa_1^6\kappa^3\kappa_A\kappa_B^7\kappa_P^{13}\delta_0}{\sqrt{t}.
		%\frac{C_2\de_0}{t},\\
		&\sum_{i=1}^2\norm{u_{i,xxxx}(t,\cdot)}_{L^1},\sum_{i=1}^2\norm{u_{i,xxx}(t,\cdot)}_{L^\infty}=O(1)\delta_0^4%\frac{2500\kappa_1^6\kappa^3\kappa_A\kappa_B^7\kappa_P^{13}\delta_0}{\sqrt{t}.
		%\frac{C_3}{\sqrt{t}t}.
		\end{aligned}
		\end{equation}
\end{corollary}
\begin{proof} It suffices to apply Proposition \eqref{prop:parabolic} on $[t-\hat{t},t]$.
\end{proof}
\begin{proposition}
\label{prop2.3}
There exists $C>0$ large enough depending on $\kappa_A$, $\kappa_B$, $\kappa$ such that
 if the solution $u$ of  \eqref{eqn-main} verifies:
 \begin{equation}
 \mbox{Tot.Var.}(u(0,\cdot))\leq\frac{\delta_0}{C},
 \label{2.16}
 \end{equation}
 then $u$ satisfies on the whole interval $[0,\hat{t}]$:
 \begin{equation}
 \|u_x(t,\cdot)\|_{L^1}\leq \frac{\delta_0}{2}.
 \label{2.17}
 \end{equation}
\end{proposition}
Proof of the proposition follows from estimates \eqref{estimate:parabolic-1} with the argument as in \cite[Proposition 2.3]{BB-vv-lim-ann-math}, \cite[Proposition 3.6]{HJ}. We omit the proof here.
\section{Gradient decomposition}\label{sec:grad-decomp}
	We would like to decompose $u_x$ in terms of travelling waves of \eqref{eqn-main}. Let $u(t,x)=U(x-\si_1t)$ be a travelling wave corresponding to 1-family. Then if satisfies
	\begin{equation}
		\left\{\begin{array}{rl}
			&u_x=v\\
			&v_x=B^{-1}(A(u)-\si_1\mathbb{I})-B^{-1}(u)[v\cdot DB(u)]v,\\
			&\si_{1,x}=0.
		\end{array}\right.
	\end{equation} 
As it is done in Appendix \ref{App-1}, we note that the gradient of 1-family travelling waves can be written under the following form
\begin{equation}
	u_x=v_1\tilde{r}_1\mbox{ where }\tilde{r}_1=\begin{pmatrix}
		1\\
		s(u,v_1,\si_1)
	\end{pmatrix}
	\label{vecteur1}
\end{equation}
for some $C^2$ function $s$. Furthermore, we can calculate that $\tilde{r}_1=r_1(u)+\tilde{\psi}_2(u,v_1,\si_1)r_2$ where $\tilde{\psi}_2$ is as in Appendix \ref{App-1}. We note from \eqref{eigenv} that:
$$s(u,v_1,\si_1)=\tilde{\psi}_2(u,v_1,\si_1)+\frac{\frac{\pa g}{\pa u_1}(u)}{f^\p(u_1)-\frac{\pa g}{\pa u_2}(u)}.$$
Therefore, from \eqref{tech1} we deduce that
\begin{equation}
\begin{aligned}
&s(u,v_1,\sigma_1)=\mathcal{O}(1)v_1,\,\,s_\si(u,v_1,\sigma_1)=\mathcal{O}(1)v_1,\,\,s_{\si\si}(u,v_1,\sigma_1)=\mathcal{O}(1)v_1,\\
&s_{u\si}(u,v_1,\sigma_1)=\mathcal{O}(1)v_1.
\end{aligned}
	\label{cle1}
\end{equation}
Hence, we get
\begin{equation}
	\tilde{r}_{1,\si}(u,v_1,\sigma_1)=\mathcal{O}(1)v_1,\quad 	\tilde{r}_{1,\si\si}(u,v_1,\sigma_1)=\mathcal{O}(1)v_1,\quad 	\tilde{r}_1\bullet\tilde{r}_{1,\si}=\mathcal{O}(1)v_1.
	\label{cle2}
\end{equation}
Furthermore, we can deduce some identities. Assume that $u(t,x)=U(x-\sigma_1 t)$ is a 1 travelling wave with speed $\sigma_1\in\R$ since $U'(x)=v_1(x)\tilde{r}_1(U(x),v_1(x),\sigma_1)$ and $u$ satisfies the equation \eqref{eqn-main}, we get
\begin{equation}\label{cal-1}
	(v_1B(u)\tilde{r}_1)_x=v_1(A(u)-\si_1)\tilde{r}_1.
\end{equation}
We observe that $B(u)\tilde{r}_1(u,v_1,\sigma_1)=\al_1\tilde{r}_1+s_1r_2$ where $s_1=(\al_2(u)-\al_1(u_1))\tilde{\psi}_2(u,v_1,\si_1)$. Therefore, we have
\begin{align}
	(v_1B(u)\tilde{r}_1)_x&=(\al_1v_1\tilde{r}_1+v_1s_1r_2)_x\\
	&=\al_1v_{1,x}\tilde{r}_1+v_1^2\al_1^\p(u_1)\tilde{r}_1+\al_1v_1^2\tilde{r}_1\bullet \tilde{r}_1+\al_1v_1v_{1,x}\tilde{r}_{1,v}\\
	&\quad\quad+v_{1,x}s_1r_2+v_1^2\tilde{r}_1\bullet s_1 r_2+v_1v_{1,x}s_{1,v}r_2.
\end{align}
Comparing the first coordinate of \eqref{cal-1} %and using the fact that $\tilde{r}_1\bullet \tilde{r}_1=\tilde{r}_1\bullet s r_2$ 
we get $\al_1v_{1,x}+v_1^2 \al_1^\p=(\la_1-\si_1)v_1$. Hence, we have
\begin{align}
	&v_1A(u)\tilde{r}_1-v_1^2\al_1^\p\tilde{r}_1-\al_1v_1^2\tilde{r}_1\bullet \tilde{r}_1-v_1^2\tilde{r}_1\bullet s_1 r_2\nonumber\\
	&= v_{1}(\la_1-v_1 \al_1^\p)\tilde{r}_1+ v^2_{1}(\la_1-v_1  \al_1^\p-\si_1)\tilde{r}_{1,v}+\frac{1}{\al_1}v_{1}(\la_1-v_1\al_1^\p-\si_1)s_1r_2\nonumber\\
	&+\frac{1}{\al_1}v^2_{1}(\la_1-v_1  \al_1^\p-\si_1)s_{1,v}r_2.\label{rel:A-v}
\end{align}
We want to have a decomposition of $u_x=\begin{pmatrix}
		u_{1,x}\\
		u_{2,x}
	\end{pmatrix}$ in the basis $(\tilde{r}_1(u,v_1,\sigma_1),r_2)$ as 
\begin{equation}\label{decomp:u-x}
	u_x(t,x)=v_1(t,x)\tilde{r}_1(u(t,x),v_1(t,x),\sigma_1(t,x))+v_2(t,x) r_2.
\end{equation}
\begin{remark}
We point out on the fact that we will use the decomposition \eqref{decomp:u-x} only for $t\geq \hat{t}$, indeed for large time $t$ we expect that $|v_1(t,x)=u_{1,x}(t,x)|$ is sufficiently small in $L^\infty$ norm  such that $|v_1(t,x)|\leq\e$ with $\e>0$ defined in the Appendix. In particular we can apply the Central Manifold Theorem since $(u(t,x),v_1(t,x),\sigma_1(t,x))$ is in a neighbourhood of $(u^*,0,\lambda_1^*)$.
\end{remark}
\begin{remark}
We claim that the decomposition of $u_x$ corresponds to a sum of gradient of viscous travelling wave. Indeed we recall that the $2$ Lax curve associated to \eqref{hyperbo} is of Temple type since it is a straight line. In this case (see \cite{HJ}), we can observe that the gradient of the $2$ viscous travelling waves take the direction of $r_2$.
\end{remark}
\begin{remark} Under the previous decomposition, due to the form of the vector $\tilde{r}_1(u,v_1,\sigma_1)$ (see \eqref{vecteur1}), we deduce that 
\begin{equation}
\label{formulev}
v_1(t,x)=u_{1,x}(t,x)\;\;\mbox{and}\;\;v_2(t,x)=u_{2,x}(t,x)-s(u(t,x),u_{1,x}(t,x),\sigma_1(t,x))u_{1,x}(t,x). 
\end{equation}
\end{remark}
Therefore, we can have $Bu_x=\al_1v_1\tilde{r}_1+v_1s_1r_2+\al_2v_2r_2$. Differentiating w.r.t. $x$ we obtain
\begin{align*}
	(B(u)u_x)_x&=\al_1v_{1,x}\tilde{r}_1+\al_1v^2_{1}\tilde{r}_1\bullet \tilde{r}_1+\al_1v_{1}v_2r_2\bullet \tilde{r}_1+\al_1v_{1,x}v_1\tilde{r}_{1,v}\\
	&+\al_1\si_{1,x}v_1\tilde{r}_{1,\si}+(v_1s_1)_xr_2+\al_2v_{2,x}r_2+v_1^2 \al^\p_1\tilde{r}_1+v_2^2(r_2\bullet \al_2)r_2\\
	&+v_1v_2(\tilde{r}_1\bullet \al_2)r_2.
\end{align*} 
By using \eqref{rel:A-v} we obtain
\begin{align*}
	&(B(u)u_x)_x-A(u)u_x\\
	&=\al_1v_{1,x}\tilde{r}_1+\al_1v^2_{1}\tilde{r}_1\bullet \tilde{r}_1+v_1^2 \al^\p_1\tilde{r}_1-v_1A(u)\tilde{r}_1-v_2A(u)r_2+\al_1v_{1}v_2r_2\bullet \tilde{r}_1\\
	&+\al_1v_{1,x}v_1\tilde{r}_{1,v}+\al_1\si_{1,x}v_1\tilde{r}_{1,\si}+v_{1,x}s_1r_2+v^2_{1}\tilde{r}_1\bullet s_1r_2\\
	&+v_1v_2r_2\bullet s_1r_2+v_1v_{1,x}s_{1,v}r_2+v_1\si_{1,x}s_{1,\si}r_2+\al_2v_{2,x}r_2\\
	&+v_2^2(r_2\bullet \al_2)r_2+v_1v_2(\tilde{r}_1\bullet \al_2)r_2\\
	&=\al_1v_{1,x}\tilde{r}_1-\la_2v_2r_2+\al_1v_{1}v_2r_2\bullet \tilde{r}_1+\al_1v_{1,x}v_1\tilde{r}_{1,v}+\al_1\si_{1,x}v_1\tilde{r}_{1,\si}\\
	&+v_{1,x}s_1r_2+v_1v_2r_2\bullet s_1r_2+v_1v_{1,x}s_{1,v}r_2+v_1\si_{1,x}s_{1,\si}r_2+\al_2v_{2,x}r_2\\
	&- v_{1}(\la_1-v_1  \al_1^\p)\tilde{r}_1-v^2_{1}(\la_1-v_1  \al_1^\p-\si_1)\tilde{r}_{1,v}\\
	&-\frac{1}{\al_1}v_{1}(\la_1-v_1 \al_1^\p-\si_1)s_1r_2-\frac{1}{\al_1}v^2_{1}(\la_1-v_1  \al_1^\p-\si_1)s_{1,v}r_2\\
	&+v_2^2(r_2\bullet \al_2)r_2+v_1v_2(\tilde{r}_1\bullet \al_2)r_2.
\end{align*}
We simplify
\begin{equation}
\begin{aligned}
	&(B(u)u_x)_x-A(u)u_x\\
	&=(\al_1v_{1,x}-(\la_1-v_1\al_1^\p) v_1)[\tilde{r}_1+v_1\tilde{r}_{1,v}]+\frac{1}{\al_1}(\al_1v_{1,x}-(\la_1-v_1\al_1^\p)v_1)(s_1+v_1s_{1,v})r_2\\
	&+(\al_2v_{2,x}- (\la_2-v_2({r}_2\bullet \al_2))v_2)r_2+v^2_{1}\si_1\tilde{r}_{1,v}+\frac{1}{\al_1}v_{1}\si_1s_1r_2\\
	&+\frac{1}{\al_1}v^2_{1}\si_1s_{1,v}r_2+v_1\si_{1,x}s_{1,\si}r_2+\al_1\si_{1,x}v_1\tilde{r}_{1,\si} \\
	&+v_1v_2(\tilde{r}_1\bullet \al_2)r_2+v_1v_2r_2\bullet s_1r_2+\al_1v_{1}v_2r_2\bullet \tilde{r}_1.
\end{aligned}
\label{ut1}
\end{equation}
Following \cite{BB-triangular,BB-vv-lim-ann-math} we set $w_1:=\al_1v_{1,x}-(\la_1-v_1\al_1^\p)v_1$. Then we have
\begin{equation}
\begin{aligned}
	&(B(u)u_x)_x-A(u)u_x\\
	&=w_1[\tilde{r}_1+v_1\tilde{r}_{1,v}]\\
	&+(\al_2v_{2,x}- (\la_2-v_2({r}_2\bullet \al_2))v_2)r_2\\
	&+\frac{1}{\al_1}(w_1+\si_1v_1)(s_1+v_1s_{1,v})r_2+v_1\si_{1,x}s_{1,\si}r_2+v^2_{1}\si_1\tilde{r}_{1,v}\\
	&+\al_1\si_{1,x}v_1\tilde{r}_{1,\si}+v_1v_2r_2\bullet s_1r_2+\al_1v_{1}v_2r_2\bullet \tilde{r}_1+v_1v_2(\tilde{r}_1\bullet \al_2)r_2.\label{eqn-B-A-2}
\end{aligned}
\end{equation}
Now, we calculate
\begin{align*}
	&((B(u)u_x)_x-A(u)u_x)_x\\
	&=(\al_1v_{1,x}-(\la_1-v_1\al_1^\p)v_1)_x[\tilde{r}_1+v_1\tilde{r}_{1,v}]+(\al_2v_{2,x}- (\la_2-v_2({r}_2\bullet \al_2))v_2)_xr_2\\
	&+(\al_1v_{1,x}-(\la_1-v_1\al_1^\p)v_1)[v_1\tilde{r}_1\bullet\tilde{r}_1+v_2r_2\bullet\tilde{r}_1+2v_{1,x}\tilde{r}_{1,v}+\si_{1,x}\tilde{r}_{1,\si}]\\
	&+(\al_1v_{1,x}-(\la_1-v_1\al_1^\p)v_1)v_1[v_1\tilde{r}_1\bullet\tilde{r}_{1,v}+v_2r_2\bullet\tilde{r}_{1,v}+v_{1,x}\tilde{r}_{1,vv}+\si_{1,x}\tilde{r}_{1,v\si}]\\
	&+\frac{1}{\al_1}((w_1+\si_1v_1)(s_1+v_1s_{1,v}))_xr_2-\frac{v_{1} \al^\p_1}{\al_1^2}((w_1+\si_1v_1)(s_1+v_1s_{1,v}))r_2\\
	&+(v_1\si_{1,x}s_{1,\si})_xr_2+2v_{1}v_{1,x}\si_1\tilde{r}_{1,v}+v^2_{1}\si_{1,x}\tilde{r}_{1,v}\\
	&+v^2_{1}\si_1[v_1\tilde{r}_1\bullet\tilde{r}_{1,v}+v_2r_2\bullet \tilde{r}_{1,v}+v_{1,x}\tilde{r}_{1,vv}+\si_{1,x}\tilde{r}_{1,v\si}]+\al_1(\si_{1,x}v_1)_x\tilde{r}_{1,\si}+(v_{1} \al_1^\p)\si_{1,x}v_1\tilde{r}_{1,\si}\\
	&+\al_1\si_{1,x}v_1[v_1\tilde{r}_1\bullet \tilde{r}_{1,\si}+v_2r_2\bullet \tilde{r}_{1,\si}+v_{1,x}\tilde{r}_{1,v\si}+\si_{1,x}\tilde{r}_{1,\si\si}]+(v_{1,x}v_2+v_1v_{2,x})r_2\bullet s_1r_2\\
	&+v_1v_2\left[v_1\tilde{r}_1\bullet(r_2\bullet s_1)+v_2r_2\bullet(r_2\bullet s_1)+v_{1,x}r_2\bullet s_{1,v}+\si_{1,x}r_2\bullet s_{1,\si}\right]r_2\\
	&+\al_1(v_{1,x}v_2+v_1v_{2,x})r_2\bullet \tilde{r}_1+(v_1\al^\p_1)v_{1}v_2 r_2\bullet \tilde{r}_1\\
	&+\al_1v_{1}v_2\left[v_1\tilde{r}_1\bullet(r_2\bullet \tilde{r}_1)+v_2r_2\bullet (r_2\bullet \tilde{r}_1)+v_{1,x}r_2\bullet \tilde{r}_{1,v}+\si_{1,x}r_2\bullet \tilde{r}_{1,\si}\right] \\
	&+[v_{1,x}v_2+v_1v_{2,x}](\tilde{r}_1\bullet \al_2)r_2+v_1v_2((u_x\otimes \tilde{r}_1):D^2 \al_2)r_2+v_1v_2((u_x\bullet \tilde{r}_1)\bullet \al_2)r_2\\
	&+v_1v_2v_{1,x}( \tilde{r}_{1,v}\bullet \al_2)r_2+v_1v_2\si_{1,x}(\tilde{r}_{1,\si}\bullet \al_2)r_2.
\end{align*}
Similarly using \eqref{eqn-B-A-2}, we have
\begin{align*}
	u_{tx}&=v_{1,t}[\tilde{r}_1+v_1\tilde{r}_{1,v}]+v_{2,t}r_2+v_1u_t\bullet \tilde{r}_1+v_1\si_{1,t}\tilde{r}_{1,\si}\\
	&=v_{1,t}[\tilde{r}_1+v_1\tilde{r}_{1,v}]+v_{2,t}r_2+v_1\si_{1,t}\tilde{r}_{1,\si}\\
	&+(\al_1v_{1,x}-(\la_1 - v_1\al_1^\p)v_1)v_1[\tilde{r}_1\bullet \tilde{r}_1+v_1\tilde{r}_{1,v}\bullet \tilde{r}_1]\\
	&+(\al_2v_{2,x}-(\la_2-r_2\bullet \al_2 v_2)v_2)v_1r_2\bullet \tilde{r}_1\\
	&+\frac{1}{\al_1}(w_1+\si_1v_1)(s_1+v_1s_{1,v})v_1r_2\bullet \tilde{r}_1\\
	&+v^3_{1}\si_1\tilde{r}_{1,v}\bullet \tilde{r}_1+v^2_1\si_{1,x}s_{1,\si}r_2\bullet \tilde{r}_1+\al_1\si_{1,x}v^2_1\tilde{r}_{1,\si}\bullet \tilde{r}_1\\
	&+v^2_1v_2r_2\bullet s_1r_2\bullet \tilde{r}_1+\al_1v^2_{1}v_2(r_2\bullet \tilde{r}_1)\bullet \tilde{r}_1\\
	&+v^2_1v_2(\tilde{r}_1\bullet \al_2)(r_2\bullet \tilde{r}_1).
\end{align*}
%After a rearrangement we obtain
%\begin{align*}
%	u_{tx}&=v_{1,t}[\tilde{r}_1+v_1\tilde{r}_{1,v}]+v_{2,t}r_2+v_1u_t\bullet \tilde{r}_1+v_1\si_{1,t}\tilde{r}_{1,\si}\\
%	&=v_{1,t}[\tilde{r}_1+v_1\tilde{r}_{1,v}]+v_{2,t}r_2+v_1\si_{1,t}\tilde{r}_{1,\si}\\
%	&+(\al_1v_{1,x}-(\la_1 -\al_1^\p v_1)v_1)v_1\tilde{r}_1\bullet \tilde{r}_1\\
%	&+(\al_1v_{1,x}-(\la_1 - \al_1^\p v_1)v_1+\si_1v_1)v^2_1\tilde{r}_{1,v}\bullet \tilde{r}_1\\
%	&+(\al_2v_{2,x}-(\la_2-r_2\bullet \al_2 v_2)v_2)v_1r_2\bullet \tilde{r}_1\\
%	&+\frac{1}{\al_1}(w_1+\si_1v_1)(s_1+v_1s_{1,v})v_1r_2\bullet \tilde{r}_1\\
%	&+v^2_1\si_{1,x}s_{1,\si}r_2\bullet \tilde{r}_1+\al_1\si_{1,x}v^2_1\tilde{r}_{1,\si}\bullet \tilde{r}_1\\
%	&+v^2_1v_2r_2\bullet s_1r_2\bullet \tilde{r}_1+\al_1v^2_{1}v_2r_2\bullet \tilde{r}_1\bullet \tilde{r}_1\\
%	&+v^2_1v_2(\tilde{r}_1\bullet \al_2)(r_2\bullet \tilde{r}_1).
%\end{align*}
Set $\tilde{\la}_1,\tilde{\la}_2$ as follows,
\begin{equation*}
	\tilde{\la}_1=\la_1-\al_1^\p v_1,\quad \tilde{\la}_2=\la_2-(r_2\bullet\al_2)v_2.
\end{equation*}
Since $u_{tx}=((B(u)u_x)_x-A(u)u_x)_x$, we can obtain
\begin{align*}
	&(v_{1,t}+(\tilde{\la}_1v_1)_x-(\al_1v_{1,x})_x)\Big[\tilde{r}_1+v_1\tilde{r}_{1,v}\Big]+(v_{2,t}+(\tilde{\la}_2v_2)_x-(\al_2v_{2,x})_x)r_2\\
	&=\frac{1}{\al_1}((w_1+\si_1v_1)(s_1+v_1s_{1,v}))_xr_2+(v_1\si_{1,x}s_{1,\si})_xr_2-\frac{v_{1} \al^\p_1}{\al_1^2}((w_1+\si_1v_1)(s_1+v_1s_{1,v}))r_2\\
	&+[w_1\sigma_{1,x}+\al_1(\si_{1,x}v_1)_x-\si_{1,t}v_1+v^2_1\alpha'_1 \sigma_{1,x} ]\tilde{r}_{1,\si}+\mathcal{E},
\end{align*} 
where $\mathcal{E}$ is defined as follows
\begin{align*}
	\mathcal{E}&=(\al_1v_{1,x}-(\la_1-v_1\al_1^\p)v_1)[v_1\tilde{r}_1\bullet\tilde{r}_1+v_2r_2\bullet\tilde{r}_1+2v_{1,x}\tilde{r}_{1,v}]\\
	&+(\al_1v_{1,x}-(\la_1-v_1\al_1^\p)v_1)v_1[v_1\tilde{r}_1\bullet\tilde{r}_{1,v}+v_2r_2\bullet\tilde{r}_{1,v}+v_{1,x}\tilde{r}_{1,vv}+\si_{1,x}\tilde{r}_{1,v\si}]\\
	&+2v_{1}v_{1,x}\si_1\tilde{r}_{1,v}\\
	&+v^2_{1}\si_{1,x}\tilde{r}_{1,v}+v^2_{1}\si_1[v_1\tilde{r}_1\bullet\tilde{r}_{1,v}+v_2r_2\bullet \tilde{r}_{1,v}+v_{1,x}\tilde{r}_{1,vv}+\si_{1,x}\tilde{r}_{1,v\si}]\\
	&+\al_1\si_{1,x}v_1[v_1\tilde{r}_1\bullet \tilde{r}_{1,\si}+v_2r_2\bullet \tilde{r}_{1,\si}+v_{1,x}\tilde{r}_{1,v\si}+\si_{1,x}\tilde{r}_{1,\si\si}]+(v_{1,x}v_2+v_1v_{2,x})r_2\bullet s_1r_2\\
	&+v_1v_2\left[v_1\tilde{r}_1\bullet(r_2\bullet s_1)+v_2r_2\bullet(r_2\bullet s_1)+v_{1,x}r_2\bullet s_{1,v}+\si_{1,x}r_2\bullet s_{1,\si}\right]r_2\\
	&+\al_1(v_{1,x}v_2+v_1v_{2,x})r_2\bullet \tilde{r}_1+(v_1\al^\p_1)v_{1}v_2 r_2\bullet \tilde{r}_1\\
	&+\al_1v_{1}v_2\left[v_1\tilde{r}_1\bullet(r_2\bullet \tilde{r}_1)+v_2r_2\bullet (r_2\bullet \tilde{r}_1)+v_{1,x}r_2\bullet \tilde{r}_{1,v}+\si_{1,x}r_2\bullet \tilde{r}_{1,\si}\right] \\
	&+[v_{1,x}v_2+v_1v_{2,x}](\tilde{r}_1\bullet \al_2)r_2+v_1v_2((u_x\otimes \tilde{r}_1):D^2 \al_2)r_2+v_1v_2((u_x\bullet \tilde{r}_1)\bullet \al_2)r_2\\
	&+v_1v_2v_{1,x}( \tilde{r}_{1,v}\bullet \al_2)r_2+v_1v_2\si_{1,x}( \tilde{r}_{1,\si}\bullet \al_2)r_2\\
	&-(\al_1v_{1,x}-(\la_1 - v_1\al_1^\p)v_1)v_1[\tilde{r}_1\bullet \tilde{r}_1+v_1\tilde{r}_{1,v}\bullet \tilde{r}_1]\\
	&-(\al_2v_{2,x}-(\la_2-r_2\bullet \al_2 v_2)v_2)v_1r_2\bullet \tilde{r}_1\\
	&-\frac{1}{\al_1}(w_1+\si_1v_1)(s_1+v_1s_{1,v})v_1r_2\bullet \tilde{r}_1\\
	&-v^3_{1}\si_1\tilde{r}_{1,v}\bullet \tilde{r}_1-v^2_1\si_{1,x}s_{1,\si}r_2\bullet \tilde{r}_1-\al_1\si_{1,x}v^2_1\tilde{r}_{1,\si}\bullet \tilde{r}_1\\
	&-v^2_1v_2r_2\bullet s_1r_2\bullet \tilde{r}_1-\al_1v^2_{1}v_2(r_2\bullet \tilde{r}_1)\bullet \tilde{r}_1\\
	&-v^2_1v_2(\tilde{r}_1\bullet \al_2)(r_2\bullet \tilde{r}_1).
\end{align*}
We have in fact obtained
\begin{align*}
	\mathcal{E}&=w_1[v_2r_2\bullet\tilde{r}_1+2v_{1,x}\tilde{r}_{1,v}]\\
	&+w_1 v_1[v_1\tilde{r}_1\bullet\tilde{r}_{1,v}+v_2r_2\bullet\tilde{r}_{1,v}+v_{1,x}\tilde{r}_{1,vv}+\si_{1,x}\tilde{r}_{1,v\si}]\\
	&+2v_{1}v_{1,x}\si_1\tilde{r}_{1,v}\\
	&+v^2_{1}\si_{1,x}\tilde{r}_{1,v}+v^2_{1}\si_1[v_1\tilde{r}_1\bullet\tilde{r}_{1,v}+v_2r_2\bullet \tilde{r}_{1,v}+v_{1,x}\tilde{r}_{1,vv}+\si_{1,x}\tilde{r}_{1,v\si}]\\
	&+\al_1\si_{1,x}v_1[v_1\tilde{r}_1\bullet \tilde{r}_{1,\si}+v_2r_2\bullet \tilde{r}_{1,\si}+v_{1,x}\tilde{r}_{1,v\si}+\si_{1,x}\tilde{r}_{1,\si\si}]+(v_{1,x}v_2+v_1v_{2,x})r_2\bullet s_1r_2\\
	&+v_1v_2\left[v_1\tilde{r}_1\bullet(r_2\bullet s_1)+v_2r_2\bullet(r_2\bullet s_1)+v_{1,x}r_2\bullet s_{1,v}+\si_{1,x}r_2\bullet s_{1,\si}\right]r_2\\
	&+\al_1(v_{1,x}v_2+v_1v_{2,x})r_2\bullet \tilde{r}_1+(v_1\al^\p_1)v_{1}v_2 r_2\bullet \tilde{r}_1\\
	&+\al_1v_{1}v_2\left[v_1\tilde{r}_1\bullet(r_2\bullet \tilde{r}_1)+v_2r_2\bullet (r_2\bullet \tilde{r}_1)+v_{1,x}r_2\bullet \tilde{r}_{1,v}+\si_{1,x}r_2\bullet \tilde{r}_{1,\si}\right] \\
	&+[v_{1,x}v_2+v_1v_{2,x}](\tilde{r}_1\bullet \al_2)r_2+v_1v_2((u_x\otimes \tilde{r}_1):D^2 \al_2)r_2+v_1v_2((u_x\bullet \tilde{r}_1)\bullet \al_2)r_2\\
	&+v_1v_2v_{1,x}( \tilde{r}_{1,v}\bullet \al_2)r_2+v_1v_2\si_{1,x}( \tilde{r}_{1,\si}\bullet \al_2)r_2\\
	&-w_1 v_1^2\tilde{r}_{1,v}\bullet \tilde{r}_1\\
	&-(\al_2v_{2,x}-(\la_2-r_2\bullet \al_2 v_2)v_2)v_1r_2\bullet \tilde{r}_1\\
	&-\frac{1}{\al_1}(w_1+\si_1v_1)(s_1+v_1s_{1,v})v_1r_2\bullet \tilde{r}_1\\
	&-v^3_{1}\si_1\tilde{r}_{1,v}\bullet \tilde{r}_1-v^2_1\si_{1,x}s_{1,\si}r_2\bullet \tilde{r}_1-\al_1\si_{1,x}v^2_1\tilde{r}_{1,\si}\bullet \tilde{r}_1\\
	&-v^2_1v_2r_2\bullet s_1r_2\bullet \tilde{r}_1-\al_1v^2_{1}v_2(r_2\bullet \tilde{r}_1)\bullet \tilde{r}_1\\
	&-v^2_1v_2(\tilde{r}_1\bullet \al_2)(r_2\bullet \tilde{r}_1).
\end{align*}
It gives then
\begin{align}
		\mathcal{E}&=(2v_{1,x}(w_1+\sigma_1 v_1)+\sigma_{1,x}v_1^2)\tilde{r}_{1,v}+v_1^2(w_1+\sigma_1 v_1)(\tilde{r}_1\bullet\tilde{r}_{1,v}-\tilde{r}_{1,v}\bullet \tilde{r}_1)\nonumber \\
	&+v_1v_{1,x}(w_1+\sigma_1 v_1)\tilde{r}_{1,vv}+v_{1}\sigma_{1,x}(w_1+\alpha_1v_{1,x}+v_1\sigma_1)\tilde{r}_{1,v\sigma}\nonumber \\
	&+
	\alpha_1\sigma_{1,x} v_1^2(\tilde{r}_1\bullet\tilde{r}_{1,\sigma}-\tilde{r}_{1,\sigma}\bullet \tilde{r}_1)+\alpha_1v_1(\sigma_{1,x})^2\tilde{r}_{1,\sigma\sigma}\nonumber \\
	&+\biggl[
	w_1v_2+\al_1(v_{1,x}v_2+v_1v_{2,x})+(v_1\al^\p_1)v_{1}v_2-v^2_1\si_{1,x}s_{1,\si}-(\al_2v_{2,x}-(\la_2-r_2\bullet \al_2 v_2)v_2)v_1\nonumber\\
	&\hspace{2cm}-\frac{1}{\al_1}(w_1+\si_1v_1)(s_1+v_1s_{1,v})v_1-v^2_1v_2(\tilde{r}_1\bullet \al_2)\biggl]   r_2\bullet \tilde{r}_1 \nonumber \\
	&+(w_1 v_1v_2+v^2_{1}\si_1v_2)   r_2\bullet\tilde{r}_{1,v} +2\al_1\si_{1,x}v_1v_2r_2\bullet \tilde{r}_{1,\si}+(v_{1,x}v_2+v_1v_{2,x})r_2\bullet s_1r_2\nonumber \\
	&+v_1v_2\left[v_1\tilde{r}_1\bullet(r_2\bullet s_1)+v_2r_2\bullet(r_2\bullet s_1)+v_{1,x}r_2\bullet s_{1,v}+\si_{1,x}r_2\bullet s_{1,\si}\right]r_2\nonumber \\
	&+\al_1v_{1}v_2\left[v_1\tilde{r}_1\bullet(r_2\bullet \tilde{r}_1)+v_2r_2\bullet (r_2\bullet \tilde{r}_1)+v_{1,x}r_2\bullet \tilde{r}_{1,v}\right] \nonumber \\
	&+[v_{1,x}v_2+v_1v_{2,x}](\tilde{r}_1\bullet \al_2)r_2+v_1v_2((u_x\otimes \tilde{r}_1):D^2 \al_2)r_2+v_1v_2((u_x\bullet \tilde{r}_1)\bullet \al_2)r_2\nonumber \\
	&+v_1v_2v_{1,x}( \tilde{r}_{1,v}\bullet \al_2)r_2+v_1v_2\si_{1,x}( \tilde{r}_{1,\si}\bullet \al_2)r_2\nonumber \\
	&-v^2_1v_2r_2\bullet s_1r_2\bullet \tilde{r}_1-\al_1v^2_{1}v_2(r_2\bullet \tilde{r}_1)\bullet \tilde{r}_1.\label{defE}
\end{align}
Observe that $\langle l_1,\xi\bullet \tilde{r}_1\rangle=0$ for all $\xi\in\R^2$ since $(\xi\bullet \tilde{r}_1)_1=0$. Then we have $\langle l_1,\mathcal{E}\rangle=0$. Hence, taking scalar product with $l_1$ we obtain
\begin{equation}
	v_{1,t}+(\tilde{\la}_1v_1)_x-(\al_1v_{1,x})_x=0.
	\label{v1}
\end{equation}
%\begin{equation}
%	v_{1,t}+((\la_1-(\tilde{r}_1\bullet \al_1)v_1)v_1)_x-\al_1v_{1,xx}=(v_1v_2(r_2\bullet \al_1))_x.
%\end{equation}
%We write
%\begin{equation}
%	v_{1,t}+(\tilde{\la}_1v_1)_x-(\al_1v_{1,x})_x=0\mbox{ where }\tilde{\la}_1=\la_1-(\tilde{r}_1\bullet \al_1)v_1-v_2(r_2\bullet \al_1).
%\end{equation}
Similarly, multiplying by ${l}_2$ we get
\begin{equation}
\begin{aligned}
	&v_{2,t}+(\tilde{\la}_2v_{2})_x-(\al_2v_{2,x})_x\\
	&=\frac{1}{\al_1}((w_1+\si_1v_1)(s_1+v_1s_{1,v}))_x +(v_1\si_{1,x}s_{1,\si})_x-\frac{v_{1} \al^\p_1}{\al_1^2}((w_1+\si_1v_1)(s_1+v_1s_{1,v}))\\
	&+[w_1\sigma_{1,x}+\al_1(\si_{1,x}v_1)_x-\si_{1,t}v_1+v^2_1\alpha'_1 \sigma_{1,x} ](l_2\cdot \tilde{r}_{1,\si})+l_2\cdot \mathcal{E}.\label{eqn-v2-1}
\end{aligned}
\end{equation}
The previous computation are satisfied for any parameter $\sigma_1$, it is then important at this level to fix our choice of $\sigma_1$. As in \cite{BB-triangular} since we consider a solution such that $\lim_{x\rightarrow-\infty}u(t,x)=u^*$, we set
\begin{equation}
\begin{aligned}
&\tilde{w}_1=w_1+\lambda_1(u^*)v_1,\\
&\si_1=\la_1(u^*)+\theta(-\frac{\tilde{w}_1}{v_1}),
\end{aligned}
	%\la_1(u^*)-\theta\left(\tilde{\la}_1(u_1)-\la_1(u^*)-\frac{\al_1 v_{1,x}}{v_1}\right)=%\la_1(u^*)-\theta\left(-\frac{w_1}{v_1}-\la_1(u^*)\right).
	\label{formulesigma}
\end{equation}
with $\theta:\R\rightarrow [-2\delta_1,2\delta_1]$ an odd function such that
$$
\theta(x)=\begin{cases}
\begin{aligned}
&x\;\;\;\mbox{for}&|x|\leq \delta_1,\\
&\mbox{smooth connection}&\delta_1\leq|x|\leq 3\delta_1,\\
&0&|x|\geq 3\delta_1.
\end{aligned}
\end{cases}
$$
We take here a cubic interpolation satisfying:
$$|\theta(x)|\leq 2\delta_1,\;|\theta'(x)|\leq 1,\;|\theta''(x)|\leq\frac{1}{\delta_1}.$$
\begin{remark}
We can note that $\delta_1$ is chosen sufficiently small such that we can define $\tilde{r}_1(u,v_1,\sigma_1)$, indeed from the Central Manifold Theorem we have seen that $\sigma_1$ lives in a neighbourhood of $0$.
\end{remark}
\begin{remark}
The previous choice on $\sigma_1$ can be explained as follows. Assume that $u(t,x)=U(x-\sigma t)$ is a 1 travelling wave solution of \eqref{eqn-main} with $\sigma$ close from $\lambda_1$, then from the Appendix we know that we have $U'(x)=v_1(x)\tilde{r}_1(U(x),v_1(x),\sigma)$, in addition we have:
$ u_t(t,x)=-\sigma  u_x(t,x)=-\sigma v_1(x-\sigma t)\tilde{r}_1(U(x-\sigma t),v_1(x-\sigma t),\sigma)$ but from \eqref{eqn-B-A-2} we deduce also that:
\begin{equation*}
	u_t=w_1\tilde{r}_1+\frac{1}{\al_1}(w_1+\si_1v_1)(s_1+v_1s_{1,v})r_2+(w_1+v_{1}\si_1)v_1\tilde{r}_{1,v}.
\end{equation*}
It is then natural to expect that $\sigma=-\frac{w_1}{v_1}$ when $\sigma$ is close from $\lambda_1(u^*)$ which is exactly what expressed the truncated expression \eqref{formulesigma} when $|\frac{\tilde{w}_1}{v_1}|\leq \delta_1$. In the sequel in order to simplify the notation, we will assume that $\lambda_1(u^*)=0$ such that $\sigma_1=\theta(-\frac{w_1}{v_1})$.
\end{remark}
In the sequel since we will prove that the solution $u(t,x)$ satisfies $\|v_1(t,\cdot)\|_{L^1}\leq 2 \delta_0$ for any $t\geq 0$ and $\| \pa_x v_1(t,x)\|_{L^1}\leq 2\delta_0^2$, we choose $\delta_1$ such that:
$$
\begin{aligned}
&\|{1}/{\alpha_1}\|_{L^\infty[-2\delta_0,2\delta_0]}
\max{(|\lambda_1(u(t,x))|+|\alpha'_1(u_1(t,x))||v_1(t,x)|)}_{(t,x)}\\
&\hspace{2cm}\leq \| {1}/{\alpha_1}\|_{L^\infty[-2\delta_0,2\delta_0]} (2\|\lambda'_1\|_{L^\infty}\delta_0+2\|\alpha'_1\|_{L^\infty}\delta_0^2)\leq\frac{1}{5}\delta_1.
\end{aligned}
$$
We deduce then that if $w_1+\sigma_1 v_1\ne 0$ then we must have $|\frac{w_1}{v_1}|\geq\delta_1$ which implies in particular that
\begin{equation}\label{5.10}
|\frac{v_{1,x}}{v_1}|\geq\frac{4}{5}\delta_1.
\end{equation}
It yields that when the ratio $|\frac{w_1}{v_1}|$ is large, we can bound $|v_1|$ by $|v_{1,x}|$. Similarly we can choose $\delta_1$ such that when $|\frac{w_1}{v_1}|\leq \delta_1$ then:
\begin{equation}
|\frac{v_{1,x}}{v_1}|\leq\frac{6}{5}\delta_1.
\end{equation}
As previously it implies that when the ratio  $|\frac{w_1}{v_1}|$ is small then we can control $|v_{1,x}|$ by $|v_1|$.
%{\color{red}\begin{remark}PITCH SUR LA REGULARITE DE $\tilde{r}_1$.
%\end{remark}}

We wish now to define the equation satisfied by $w_1$. Since $w_1=\al_1v_{1,x}-(\la_1-v_1\al_1^\p)v_1$ we calculate
\begin{equation*}
	w_{1,x}=(\al_1v_{1,x}-(\la_1-v_1\al_1^\p)v_1)_x=v_{1,t}.
\end{equation*}
Then we have
\begin{align*}
	w_{1,t}&=\al_1 v_{1,tx}+u_{1,t}\al_1^\p v_{1,x}-(\la_1-v_1\al_1^\p)v_{1,t}-(\la_1^\p-v_1\al_1^{\p\p})u_{1,t}v_{1}+v_{1,t}\al_1^\p v_1\\
	&=(\al_1w_{1,x}-(\la_1-v_1\al^\p_1)w_1)_x-\al_1^\p u_{1,x}w_{1,x}+u_{1,t}\al_1^\p v_{1,x}\\
	&+(\la_1^\p-v_1\al_1^{\p\p})u_{1,x}w_{1}-(\la_1^\p-v_1\al_1^{\p\p})u_{1,t}v_{1}+v_{1,t}\al_1^\p v_1-v_{1,x}\al_1^\p w_1.
\end{align*}
Observe that $u_{1,t}=w_1$ and $u_{1,x}=v_1$. Hence,
\begin{align*}
	w_{1,t}
	&=(\al_1w_{1,x}-(\la_1-v_1\al^\p_1)w_1)_x-\al_1^\p v_1w_{1,x}+w_1\al_1^\p v_{1,x}\\
	&+(\la_1^\p-v_1\al_1^{\p\p})v_1w_{1}-(\la_1^\p-v_1\al_1^{\p\p})w_1v_{1}+w_{1,x}\al_1^\p v_1-v_{1,x}\al_1^\p w_1\\
	&=(\al_1w_{1,x}-(\la_1-v_1\al^\p_1)w_1)_x.
\end{align*}
We have obtained in particular that
\begin{equation}
	w_{1,t}+(\tilde{\la}_1w_1)_x-(\al_1 w_{1,x})_x=0.
	\label{w1}
\end{equation}
Let us simplify now the right hand side of \eqref{eqn-v2-1}. We denote in the sequel by $\theta^{(k)}_1$ the value $\theta^{(k)}\left(-\frac{w_1}{v_1}\right)$ with $k\in\mathbb{N}$. Now we can calculate
\begin{align*}
	&\al_1(\si_{1,x}v_1)_x-\si_{1,t}v_1\\
	&=-\al_1\theta_1^\p\left(w_{1,x}-\frac{w_1}{v_1}v_{1,x}\right)_x+\al_1\theta_1^{\p\p}v_1\left(\frac{w_1}{v_1}\right)_x^2+\theta_1^\p \left(w_{1,t}-\frac{w_1}{v_1}v_{1,t}\right)\\
	&=\theta_1^\p \left(w_{1,t}-\al_1w_{1,xx}-\frac{w_1}{v_1}(v_{1,t}-\al_1v_{1,xx})\right)+\al_1\theta_1^\p\left(\frac{w_1}{v_1}\right)_xv_{1,x}+\al_1\theta_1^{\p\p}v_1\left(\frac{w_1}{v_1}\right)_x^2.
\end{align*}
%with $\theta'_1=\theta'(-\frac{w_1}{v_1}-\la_1(u^*))$ and $\theta''_1=\theta''(-\frac{w_1}{v_1}-\la_1(u^*))$.
Note that
\begin{align*}
	w_{1,t}-\al_1w_{1,xx}&=-\tilde{\la}_{1,x}w_1-\tilde{\la}_1w_{1,x}+\al_{1}^\p v_1w_{1,x},\\
	v_{1,t}-\al_1v_{1,xx}&=-\tilde{\la}_{1,x}v_1-\tilde{\la}_1v_{1,x}+\al_{1}^\p v_1v_{1,x}.
\end{align*}
This implies,
\begin{align*}
	&\left(w_{1,t}-\al_1w_{1,xx}-\frac{w_1}{v_1}(v_{1,t}-\al_1v_{1,xx})\right)\\
	&=-\tilde{\la}_1v_1\left(\frac{w_1}{v_1}\right)_x+\al_{1}^\p (v_1w_{1,x}-w_1v_{1,x}).
\end{align*}
It gives in particular
$$
\begin{aligned}
&\al_1(\si_{1,x}v_1)_x-\si_{1,t}v_1=-\theta'_1 \tilde{\la}_1v_1\left(\frac{w_1}{v_1}\right)_x+\al_{1}^\p \theta'_1 v_1^2 \left(\frac{w_1}{v_1}\right)_x+\al_1\theta_1^\p\left(\frac{w_1}{v_1}\right)_xv_{1,x}+\al_1\theta_1^{\p\p}v_1\left(\frac{w_1}{v_1}\right)_x^2.
\end{aligned}
$$
It yields using the fact that $w_1=\alpha_1 v_{1,x}-\tilde{\lambda}_1 v_1$
$$
\begin{aligned}
&w_1\sigma_{1,x}+\al_1(\si_{1,x}v_1)_x-\si_{1,t}v_1+\alpha'_1 v_1^2\sigma_{1,x}\\
&=-\theta'_1 \tilde{\la}_1v_1\left(\frac{w_1}{v_1}\right)_x+\al_{1}^\p \theta'_1 v_1^2 \left(\frac{w_1}{v_1}\right)_x+\al_1\theta_1^\p\left(\frac{w_1}{v_1}\right)_xv_{1,x}+\al_1\theta_1^{\p\p}v_1\left(\frac{w_1}{v_1}\right)_x^2\\
&-\alpha_1 v_{1,x}\theta_1^\p\left(\frac{w_1}{v_1}\right)_x+\tilde{\lambda}_1 v_1 \theta_1^\p\left(\frac{w_1}{v_1}\right)_x-\alpha'_1 v_1^2\theta_1^\p\left(\frac{w_1}{v_1}\right)_x\\
&=\al_1\theta_1^{\p\p}v_1\left(\frac{w_1}{v_1}\right)_x^2.
\end{aligned}
$$
Then we have from \eqref{eqn-v2-1}
\begin{equation}
\begin{aligned}
&v_{2,t}+(\tilde{\la}_2v_{2})_x-(\al_2v_{2,x})_x\\
	&=\frac{1}{\al_1}((w_1+\si_1v_1)(s_1+v_1s_{1,v}))_x +(v_1\si_{1,x}s_{1,\si})_x\\
	&+\al_1\theta_1^{\p\p}v_1\left(\frac{w_1}{v_1}\right)_x^2 l_2\cdot \tilde{r}_{1,\si}-\frac{v_{1} \al^\p_1}{\al_1^2}((w_1+\si_1v_1)(s_1+v_1s_{1,v}))+l_2\cdot \mathcal{E}.
\end{aligned} 
\label{eqv2}
\end{equation}
In the sequel we will note by:
\begin{equation}
\begin{aligned}
\phi_2=&\frac{1}{\al_1}((w_1+\si_1v_1)(s_1+v_1s_{1,v}))_x +(v_1\si_{1,x}s_{1,\si})_x
	+\al_1\theta_1^{\p\p}v_1\left(\frac{w_1}{v_1}\right)_x^2 l_2\cdot \tilde{r}_{1,\si}\\
	&-\frac{v_{1} \al^\p_1}{\al_1^2}((w_1+\si_1v_1)(s_1+v_1s_{1,v}))+l_2\cdot \mathcal{E}.
	\end{aligned}
	\label{eqphi2}
\end{equation}
\section{Form of the remainder terms}\label{section:rem}
In this section, we wish to prove that a solution $u(t,x)$ of \eqref{eqn-main} satisfies for any $t> 0$
\begin{equation}
\|u_{1,x}(t,\cdot)\|_{L^1}+\|u_{2,x}(t,\cdot)\|_{L^1}\leq \delta_0,
\label{cru}
\end{equation}
for $\delta_0>0$ fixed and 
provided that $\mbox{TV}(u(0,\cdot))$ is sufficiently small in terms of $\delta_0$. 
Let us denote by $\kappa_1$ the following constant
\begin{equation}
\kappa_1=\max\left(1,\sup_{|u-u^*|\leq\e, |v_1|\leq \e,|\sigma_1|\leq 2\delta_1}s(u,v_1,\sigma_ 1)\right),
\end{equation}
with $\e>0$ defined as in the Appendix. In addition we can take again $\mbox{Tot.Var.}(u(0,\cdot))$ sufficiently small such that the solution $(u(t,x),v_1(t,x),\sigma_1(t,x))$ satisfies all along the time $u(t,x)\in B(u^*,\e)$ and $v_{1}(t,x)=u_{1,x}(t,x)\in B(0,\e)$ for $t\geq \hat{t}$.
Using the Proposition \ref{prop2.3}, we know that \eqref{cru} is satisfied for any $t\in[0,\hat{t}]$ provided that ${Tot.Var.}(u(0,\cdot))\leq\frac{\delta_0}{C}$ with $C>0$ sufficiently large depending on $\kappa,\kappa_A$ and $\kappa_B$. 

Our main goal now is to prove similar result when $t\geq \hat{t}$, to do this as in \cite{BB-vv-lim-ann-math} it will require hyperbolic type estimates relying to the interactions of wave based on the previous decomposition of $u_x$ as a sum of travelling wave. To do this we start by giving regularity estimates on $v_1$, $v_2$ and $w_1$.
\begin{lemma}  Let $T>\hat{t}$. Assume that the solution $u$ of  \eqref{eqn-main} satisfies on $[0,T]$
$$\|u_{1,x}(t,\cdot)\|_{L^1}+\|u_{2,x}(t,\cdot)\|_{L^1}\leq \delta_0,$$
then for all $t\in[\hat{t},T]$ we have
\begin{equation}\label{estimate:parabolic-2}
	\begin{aligned}
	&\sum_{i=1}^2\norm{v_i(t,\cdot)}_{L^1}+\|w_1(t,\cdot)\|_{L^1}=O(1)\delta_0\\
		&\sum_{i=1}^2\norm{v_{i,x}(t,\cdot)}_{L^1}+\|w_{1,x}(t,\cdot)\|_{L^1}, \sum_{i=1}^2\norm{v_{i}(t,\cdot)}_{L^\infty}+\|w_{1}(t,\cdot)\|_{L^\infty} =O(1)\delta_0^2\\
		%\frac{16 \kappa \kappa_B^6\de_0}{\sqrt{t}},\\
		&\sum_{i=1}^2 \norm{v_{i,xx}(t,\cdot)}_{L^1}+\|w_{1,xx}(t,\cdot)\|_{L^1},\sum_{i=1}^2 \norm{v_{i,x}(t,\cdot)}_{L^\infty}+\|w_{1,x}(t,\cdot)\|_{L^\infty} =O(1)\delta_0^3%\frac{2500\kappa_1^6\kappa^3\kappa_A\kappa_B^7\kappa_P^{13}\delta_0}{\sqrt{t}.
		,\\
		&\sum_{i=1}^2\norm{v_{i,xxx}(t,\cdot)}_{L^1}, \sum_{i=1}^2\norm{v_{i,xx}(t,\cdot)}_{L^\infty}=O(1)\delta_0^4.%\frac{2500\kappa_1^6\kappa^3\kappa_A\kappa_B^7\kappa_P^{13}\delta_0}{\sqrt{t}.
		%\frac{C_3}{\sqrt{t}t}.
		\end{aligned}
		\end{equation}
\label{lemme4.1}
\end{lemma}
\begin{proof} It suffices to apply the Corollary \ref{coro2.2} and the formula \eqref{formulev}.
\end{proof}
We have seen that the solution $u$ satisfies \eqref{cru} for $t\in[0,\hat{t}]$ and we know that the solution can be extended as long as its total variation remains small. Let us denote the time $T$ satisfying,
\begin{equation}
T=\sup \left\{t,\int^{t}_{\hat{t}}|\phi_2(s,x)| ds dx\leq\frac{\delta_0}{5}\right\}.
\label{10.3}
\end{equation}
We want to prove that $T=+\infty$. We proceed then by contradiction and assume now that $T<+\infty$.
 To do this we observe by applying maximum principle that for any $t\in[\hat{t},T]$ and the formula \eqref{formulev}  we have
\begin{equation}
\begin{aligned}
\|v_1(t,\cdot)\|_{L^1}&=\|u_{1,x}(t,\cdot)\|_{L^1}\leq \|u_{1,x}(\hat{t},\cdot)\|_{L^1}\\\
\|v_2(t,\cdot)\|_{L^1}&\leq \|v_{2}(\hat{t},\cdot)\|_{L^1}+\int^t_{\hat{t}}\int_{\R}|\phi_2(s,x)|ds dx,\\
&\leq (\kappa_1\|u_{1,x}(\hat{t},\cdot)\|_{L^1}+ \|u_{2,x}(\hat{t},\cdot)\|_{L^1})+\frac{\delta_0}{5}.
%&\leq \frac{\delta_0}{2\kappa_1}.
\end{aligned}
\end{equation}
It implies now that for $t\in[\hat{t},T]$, we get using \eqref{formulev} and \eqref{cru}
\begin{align*}
& \|u_{1,x}(t,\cdot)\|_{L^1}+\|u_{2,x}(t,\cdot)\|_{L^1}\\
&\leq (1+\kappa_1) \|u_{1,x}(\hat{t},\cdot)\|_{L^1}+ (\kappa_1\|u_{1,x}(\hat{t},\cdot)\|_{L^1}+ \|u_{2,x}(\hat{t},\cdot)\|_{L^1})+\frac{\delta_0}{5}\\
&\leq\delta_0.
 \end{align*}
 It means that we can apply Corollary \ref{coro2.2} and Lemma \ref{lemme4.1} for any $t\in[\hat{t},T]$ and in particular we control the $L^\infty$ norm of $v_{i,x}$, $v_{i,xx}$ and $w_1, w_{1,x}$.
 
 We wish now to estimate $\int^t_{\hat{t}}\int_{\R}|\phi_2(s,x)|ds dx$ and to prove that
 \begin{equation}
 \int^t_{\hat{t}}\int_{\R}|\phi_2(s,x)|ds dx=O(1)\delta_0^2
 \label{10.5}
 \end{equation}
 which would contradict the fact that $T$ is a supremum. From \eqref{defE}, we observe that each of the terms of $\mathcal{E}$ can be estimated by the following expression
\begin{enumerate}
	\item wrong speed
	\begin{equation}\label{eqn:wrong-speed}
		O(1)v_{1,x}(w_1+\si_1v_1),
	\end{equation}
    \item change in strength
    \begin{equation}\label{eqn:change-strength}
    	O(1)w_1v_{1,x}-w_{1,x}v_1,
    \end{equation}
\item change in speed
\begin{equation}\label{eqn:change-speed}
	O(1)v_1^2\left(\frac{w_1}{v_1}\right)_x^2\chi_{\abs{w_1/v_1}\leq3\de_1},
\end{equation}
\item transversal interaction
\begin{equation}\label{eqn:transversal}
	O(1)v_{1}v_2,\quad v_{1,x}v_2, \quad v_{2,x}v_1.
\end{equation}
\end{enumerate}
For completeness we are going to estimate some of the terms of $\langle{\cal E},l_2\rangle$, essentially the more complicated. We will use in a crucial way the estimate \eqref{cle1} and \eqref{cle2}. We note that some terms are completely similar to those appearing in \cite{BB-triangular} and some others one are new due to the fact that $\alpha_1$, $\alpha_2$ are not constant and also because $s_1$ is not necessary zero since $\alpha_1\ne \alpha_2$.
\begin{itemize}
\item We have then using the fact that $w_1+\sigma_1 v_1\ne 0$ only if $|\frac{w_1}{v_1}|\geq \delta_1$ which implies that $|v_1|=O(1)|v_{1,x}|$ then it gives
$$v_1^2(w_1+\sigma_1 v_1)\langle(\tilde{r}_1\bullet\tilde{r}_{1,v}-\tilde{r}_{1,v}\bullet \tilde{r}_1) ,l_2\rangle=O(1) v_1 v_{1,x}(w_1+\sigma_1 v_1).$$
\item  Using now the fact that $\theta'_1\ne 0$ if $|\frac{w_1}{v_1}|\leq 3\delta_1$ which implies that $v_{1,x}=0(1)v_1$ we have
$$
\begin{aligned}
&v_{1}\sigma_{1,x}(w_1+\alpha_1v_{1,x}+v_1\sigma_1)\langle \tilde{r}_{1,v\sigma},l_2\rangle\\
&=-\theta'_1(w_{1,x}v_1-v_{1,x}w_1)(\frac{w_1}{v_1}+\alpha_1\frac{v_{1,x}}{v_1}+\sigma_1) )\langle \tilde{r}_{1,v\sigma},l_2\rangle\\
&=O(1)(w_{1,x}v_1-v_{1,x}w_1).
\end{aligned}
$$
\item The term in $\langle\tilde{r}_{1,\sigma\sigma},l_2 \rangle$
satisfies since $\tilde{r}_{1,\sigma\sigma}=O(1)v_1$,
$$\alpha_1v_1(\sigma_{1,x})^2\langle\tilde{r}_{1,\sigma\sigma},l_2 \rangle=O(1)v_1^2\left(\frac{w_1}{v_1}\right)_x^2\chi_{\abs{w_1/v_1}\leq3\de_1}.$$
\item The term in $\langle r_2\bullet \tilde{r}_{1,\si},l_2\rangle$ verifies since $r_2\bullet \tilde{r}_{1,\si}=O(1)v_1$,
$$2\al_1\si_{1,x}v_1v_2\langle r_2\bullet \tilde{r}_{1,\si},l_2\rangle=-2\alpha_1\theta'_1 v_2v_1(w_{1,x} -\frac{v_{1,x}}{v_1})\langle \frac{r_2\bullet \tilde{r}_{1,\si}}{v_1},l_2\rangle=O(1)v_1v_2. $$
\item The term $\langle v_1v_2\si_{1,x}( \tilde{r}_{1,\si}\bullet \al_2)r_2,l_2\rangle$ can be dealt as previously.
\end{itemize}
Let us estimate now the main new terms compared with \cite{BB-triangular}on the right hand side of \eqref{eqphi2} which are $(v_1\si_{1,x}s_{1,\si})_x$, $
-\frac{v_{1} \al^\p_1}{\al_1^2}((w_1+\si_1v_1)(s_1+v_1s_{1,v}))$ and $-\frac{v_{1} \al^\p_1}{\al_1^2}((w_1+\si_1v_1)(s_1+v_1s_{1,v}))$.
Let us start with the term $\frac{1}{\al_1}((w_1+\si_1v_1)(s_1+v_1s_{1,v}))_xr_2$. We note now that since $w_1+\sigma_1 v_1$ is different from $0$ if $|\frac{w_1}{v_1}|\geq \delta_1$ then we have
\begin{align*}
 &((w_1+\si_1v_1)(s_1+v_1s_{1,v}))_x\\
 &=(w_{1,x}+\si_1v_{1,x})(s_1+v_1s_{1,v})\chi_{\left\{\abs{\frac{w_1}{v_1}}\leq\de_1\right\}}+(w_{1,x}+\si_1v_{1,x})(s_1+v_1s_{1,v})\chi_{\left\{\de_1<\abs{\frac{w_1}{v_1}}\right\}}\\
 &-\theta_1^\p\left(\frac{w_1}{v_1}\right)_xv_1(s_1+v_1s_{1,v})+(w_1+\si_1v_1)v_1(\tilde{r}_1\bullet s_1+v_1\tilde{r}_1\bullet s_{1,v})\chi_{\left\{\de_1<\abs{\frac{w_1}{v_1}}\right\}}\\
 &+(w_1+\si_1v_1)v_2(r_2\bullet s_1+v_1r_2\bullet s_{1,v})\chi_{\left\{\de_1<\abs{\frac{w_1}{v_1}}\right\}}+(w_1+\si_1v_1)v_{1,x}( 2s_{1,v}+v_1  s_{1,vv})\chi_{\left\{\de_1<\abs{\frac{w_1}{v_1}}\right\}}\\
 &-(w_1+\si_1v_1)\theta'_1\left(\frac{w_1}{v_1}\right)_x( s_{1,\si}+v_1  s_{1,v\si})\chi_{\left\{\de_1<\abs{\frac{w_1}{v_1}}\right\}}\\[2mm]
  &=(\frac{s_1}{v_1}+s_{1,v})[(w_{1,x}v_1-v_{1,x}w_1)+(w_1+\sigma_1 v_1)v_{1,x}]\\
 &-\theta_1^\p\left(\frac{w_1}{v_1}\right)_x v^2_1(\frac{s_1}{v_1}+s_{1,v})+(w_1+\si_1v_1)v_1(\tilde{r}_1\bullet s_1+v_1\tilde{r}_1\bullet s_{1,v})\chi_{\left\{\de_1<\abs{\frac{w_1}{v_1}}\right\}}\\
 &+(w_1+\si_1v_1)v_2(r_2\bullet s_1+v_1r_2\bullet s_{1,v})\chi_{\left\{\de_1<\abs{\frac{w_1}{v_1}}\right\}}+(w_1+\si_1v_1)v_{1,x}( 2s_{1,v}+v_1  s_{1,vv})\chi_{\left\{\de_1<\abs{\frac{w_1}{v_1}}\right\}}\\
 &-(w_1+\si_1v_1)\theta'_1\left(\frac{w_1}{v_1}\right)_x( s_{1,\si}+v_1  s_{1,v\si})\chi_{\left\{\de_1<\abs{\frac{w_1}{v_1}}\right\}}.
 \end{align*}
Since we have seen that $s_1=O(1)|v_1|$, $s_{1,\sigma}=O(1)|v_1|$, $\mbox{supp}\theta'$ is included in $\{|\frac{w_1}{v_1}|\leq 3\delta_1\}$ and  $|v_1|= O(1)|v_{1,x}|$ when $|\frac{w_1}{v_1}|\geq\delta_1$ we deduce that
 \begin{align*}
 ((w_1+\si_1v_1)(s_1+v_1s_{1,v}))_x&=\mathcal{O}(1)\abs{w_{1,x}v_1-w_1v_{1,x}}+\mathcal{O}(1)|v_{1,x}(w_1+\sigma_1 v_1)|\\
 &+\mathcal{O}(1)(|v_1 v_2|+|v_2 v_{1,x}|)\\
 &+\mathcal{O}(1) v_1^2\left(\frac{w_1}{v_1}\right)^2\chi_{\abs{w_1/v_1}\leq3\de_1}
 +\mathcal{O}(1)(w_1^2+v_1^2)\chi_{\left\{\de_1<\abs{\frac{w_1}{v_1}}\right\}}.
\end{align*}
Let us study now the term $-\frac{v_{1} \al^\p_1}{\al_1^2}((w_1+\si_1v_1)(s_1+v_1s_{1,v}))$, which is of the form $\mathcal{O}(1)|v_{1,x}(w_1+\sigma_1 v_1)|$ since $s_1+v_1s_{1,v}=\mathcal{O}(1)|v_1|$ and $|v_1|=\mathcal{O}(1)|v_{1,x}|$ when $w_1+\sigma_1 v_1\ne 0$.
We would like now to rewrite \eqref{eqn-v2-1} as follows
\begin{equation}
	v_{2,t}+(\tilde{\la}_2v_{2})_x-(\al_2v_{2,x})_x=\varphi_2+(v_1\si_{1,x}s_{1,\si})_x,
	\label{4th.13}
\end{equation}
where $\varphi_2$ is defined as follows
\begin{equation*}
	\varphi_2=\phi_2-(v_1\si_{1,x}s_{1,\si})_x.
\end{equation*}
By the analysis so far we conclude that $\varphi_2$ can be estimated by \eqref{eqn:wrong-speed}, \eqref{eqn:change-strength}, \eqref{eqn:change-speed} and \eqref{eqn:transversal}. Now, we can calculate
\begin{align*}
	(v_1\si_{1,x}s_{1,\si})_x&=-v_{1,x}\theta_1^\p\left(\frac{w_1}{v_1}\right)_xs_{1,\si}+v_1\theta_1^{\p\p}\left(\frac{w_1}{v_1}\right)_x^2s_{1,\si}\\
	&-\theta_1^\p v_1\left(\frac{w_1}{v_1}\right)_{xx}s_{1,\si}-\theta_1^\p v^2_1\left(\frac{w_1}{v_1}\right)_{x}\tilde{r}_1\bullet s_{1,\si}-\theta_1^\p v_1v_2\left(\frac{w_1}{v_1}\right)_{x}r_2\bullet s_{1,\si}\\
	&-\theta_1^\p v_1v_{1,x}\left(\frac{w_1}{v_1}\right)_xs_{1,v\si}+(\theta_1^\p)^2v_1\left(\frac{w_1}{v_1}\right)_x^2s_{1,\si\si}.
\end{align*}
Note that
\begin{equation}
	\left(\frac{w_1}{v_1}\right)_{xx}=\frac{w_{1,xx}v_1-v_{1,xx}w_1}{v_1^2}-\frac{2v_{1,x}}{v_1}\left(\frac{w_1}{v_1}\right)_x.
	\label{formule2}
\end{equation}
In particular we deduce since $s_{1,\si}=\mathcal{O}(1)|v_{1}|$ and because $\mbox{supp}\theta'$ is included in $\{|\frac{w_1}{v_1}|\leq 3\delta_1\}$ that
\begin{equation}
\theta_1^\p v_1\left(\frac{w_1}{v_1}\right)_{xx}s_{1,\si}=\mathcal{O}(1)\abs{w_{1,xx}v_1-v_{1,xx}w_1}+\mathcal{O}(1)\abs{w_{1,x}v_1-v_{1,x}w_1}.
\end{equation}
Since $s_{1,\si}=\mathcal{O}(1)|v_{1}|$ and $v_{1,x}=\mathcal{O}(1)v_1$ when $\theta'_1\ne 0$, we obtain
\begin{align*}
	(v_1\si_{1,x}s_{1,\si})_x&=\mathcal{O}(1)\abs{w_{1,xx}v_1-v_{1,xx}w_1}+\mathcal{O}(1)v_1^2\left(\frac{w_1}{v_1}\right)_x^2\chi_{\left\{\abs{\frac{w_1}{v_1}}\leq 3\delta_1\right\}}\\
	&+\mathcal{O}(1)\abs{w_{1,x}v_1-v_{1,x}w_1}.
\end{align*}
To summarize we have seen that all the terms on the right hand side of \eqref{4th.13} are of the form
\begin{equation}
\begin{aligned}
&O(1)v_{1,x}(w_1+\si_1v_1),\;
    	O(1)(w_1v_{1,x}-w_{1,x}v_1),\;
	\;O(1)v_1^2\left(\frac{w_1}{v_1}\right)_x^2\chi_{\abs{w_1/v_1}\leq3\de_1},\\
	&
	O(1)v_{1}v_2,\;v_{1,x}v_2, \quad v_{2,x}v_1\;\;\mbox{and}\;\;w_{1,xx}v_1-v_{1,xx}w_1.
\end{aligned}
\label{4.19}
\end{equation}
We note that the term $w_{1,xx}v_1-v_{1,xx}w_1$ does not appear in \cite{BB-triangular}. We wish now to estimate $L^1$ norm of each of these terms on $[\hat{t},T]\times\R$.
%Using \cite{BB-triangular} it can be shown that
%\begin{equation*}
%	\int\limits_{0}^{T}\int\limits_{\R}\abs{\varphi_2(t,x)}\,dxdt=\mathcal{O}(1)\de_0.
%\end{equation*}
%Note that above calculation works for any smooth function $\si$. Now, we specify the choice as in \cite{BB-triangular}. We consider a smooth odd function $\theta:\R\rr[-\de_1,\de_1]$ such that
%\begin{equation}
%	\theta(s)=\left\{\begin{array}{rl}
%		s&\mbox{ if }\abs{s}\leq \frac{\de_1}{2},\\
%		0&\mbox{ if }\abs{s}\geq \de_1,
%	\end{array}\right.\quad \abs{\theta^\p}\leq 1\mbox{ and  }\abs{\theta^{\p\p}}\leq 4/\de_1.
%\end{equation}

\section{Shortening curves}\label{sec:shortening}
Your goal now is to deal with the new term $\mathcal{O}(1)\abs{w_{1,xx}v_1-v_{1,xx}w_1}$ which does not appear in the analysis of \cite{BB-triangular,BB-vv-lim-ann-math} when we consider constant viscosity coefficients.
\subsection{Scalar conservation laws}
In this section we mainly discuss the application of shortening curves estimates for scalar conservation laws which reads as follows
\begin{equation}
	u_t+f^\p(u)u_x=(\alpha (u)u_{x})_x.
\end{equation} 
Though it has been already used in \cite{BB-triangular,BB-vv-lim-ann-math} to show that $u_{xx}u_{t}-u_{tx}u_x$ is $L^1_{t,x}$ when $\alpha(u)=1$, here we wish to prove that $u_{xxx}u_t-u_{txx}u_x$ is $L^1_{t,x}$. We consider $v:=u_x$ satisfying the following equation
\begin{equation}
	v_t+(\tilde{\la}(u)v)_x=(\alpha(u)v_{x})_x\mbox{ where }\tilde{\la}(u)=f^\p(u)-\alpha'(u)v\;\;\mbox{and}\;\;\lambda(u)=f^\p(u).
\end{equation}
We define $w=\alpha v_x-\tilde{\lambda} v$. Observe that $w=u_t$ and $v_t=w_x$. Then we deduce 
\begin{align*}
	w_t=(\alpha v_x)_t-(\tilde{\lambda} v)_t&=\alpha'(u)u_t v_x+\alpha w_{xx}
	-\la^\p u_tv+(\alpha'(u)v)_t v-\la v_t+\alpha'(u)v v_t\\
	&=\alpha'(u)w v_x+\alpha w_{xx}
	-\la^\p w v+\alpha''(u)u_t v^2+\alpha'(u) v_t v-\la v_t+\alpha'(u)v v_t\\
	&=\alpha'(u)w v_x+\alpha w_{xx}
	-\la^\p w v+\alpha''(u)w v^2+\alpha'(u) w_x v-\la w_x+\alpha'(u)v w_x\\
	&=(\alpha w_{x})_x-(\tilde{\lambda}w)_x.
	%&=(\alpha w_{x})_x-\la_x w-\la w_x.
\end{align*}
Hence, $w$ solves
\begin{equation}
	w_t+(\tilde{\la} w)_x-(\alpha(u)w_{x})_x=0.
\end{equation}
We introduce a new variable $z:=\alpha(u)w_x-\tilde{\la} w$. We observe that $z_x=w_t$. Furthermore,
\begin{equation}
\begin{aligned}
	z_t&=\alpha(u)w_{tx}+(\alpha(u))_t w_x-(\tilde{\la})_t w-\tilde{\la} w_t\\
	&=\alpha(u)z_{xx}+(\alpha(u))_t w_x-(\tilde{\la})_t w-\tilde{\la} z_x.
	%&=z_{xx}-\la^\p w(v_x-\la v)-\tilde{\la} z_x\\
	%&=z_{xx}-\la^\p wv_{x}+\la \la^\p wv-\la z_x\\
	%&=z_{xx}-\la^\p wv_{x}+ \la_x (\la w)-\la z_x\\
	%&=z_{xx}-\la^\p wv_{x}+ \la_x (w_x-z)-\la z_x\\
	%&=z_{xx}-(\la z)_x+\la^\p w_x v-\la^\p wv_{x}.
\end{aligned}
\label{eq1}
\end{equation}
Let us compute now $-(\tilde{\la})_t w$, it gives:
\begin{align*}
	&-(\tilde{\la})_t w=(-\lambda'(u)u_t+\alpha''(u)u_u v+\alpha'(u)v_t)w\\
	&=-\lambda'(u) w^2+\alpha''(u)w^2 v+\alpha'(u) w_x w\\
	&=-\lambda'(u) w\alpha v_x+\lambda'(u) w\tilde{\lambda} v+\alpha''(u)w\alpha v_x v-\alpha''(u)w\tilde{\lambda} v^2+\alpha'(u) w_x w\\
	&=-\lambda'(u) w\alpha v_x+(\lambda(u))_x (\alpha w_x-z) +\alpha''(u)w\alpha v_x v-(\alpha'(u))_x  (\alpha w_x-z) v+\alpha'(u) w_x w\\
	&=-(\lambda(u))_x z+\lambda'(u)\alpha(w_x v-v_x w)-(\alpha'(u))_x v \alpha w_x+(\alpha'(u))_x v z
	 +\alpha''(u)w\alpha v_x v\\
	 &\hspace{0,5cm}+\alpha'(u) w_x (\alpha v_x-\tilde{\lambda}v)\\
	 &=-(\lambda(u))_x z+\lambda'(u)\alpha(w_x v-v_x w)-(\alpha'(u))_x v \alpha w_x+(\alpha'(u))_x v z
	 +\alpha''(u)w\alpha v_x v\\
	 &\hspace{0,5cm}-\alpha'(u) w_x \tilde{\lambda}v+\alpha'(u)v_x(z+\tilde{\lambda}w)\\[2mm]
	 &=-(\lambda(u))_x z+\lambda'(u)\alpha(w_x v-v_x w)+(\alpha'(u))_x v z+\alpha'(u) v_x z-(\alpha'(u))_x v \alpha w_x+\alpha''(u)w\alpha v_x v\\
	 &\hspace{0,5cm}+\alpha'(u)v_x\tilde{\lambda} w-\alpha'(u)w_x\tilde{\lambda}v.%&=-\lambda'(u) w\alpha v_x+(\lambda(u))_x w\tilde{\lambda} +\alpha''(u)w\alpha v_x v-(\alpha'(u))_x (\alphaw_x-z) v+\alpha'(u) w_x w\\
	%&=z_{xx}-\la^\p w(v_x-\la v)-\tilde{\la} z_x\\
	%&=z_{xx}-\la^\p wv_{x}+\la \la^\p wv-\la z_x\\
	%&=z_{xx}-\la^\p wv_{x}+ \la_x (\la w)-\la z_x\\
	%&=z_{xx}-\la^\p wv_{x}+ \la_x (w_x-z)-\la z_x\\
	%&=z_{xx}-(\la z)_x+\la^\p w_x v-\la^\p wv_{x}.
\end{align*}
We have finally obtained:
\begin{equation}
-(\tilde{\la})_t w=-\tilde{\lambda}_x z+\lambda'(u)\alpha(u)(w_xv-v_x w)+\alpha'(u)\tilde{\lambda}(v_x w-w_x v)+\alpha''(u)v \alpha(u)(wv_x-v w_x).
\label{eq2}
\end{equation}
Let us estimate now $(\alpha(u))_t w_x$,  we have then:
\begin{equation}
\begin{aligned}
	&(\alpha(u))_t w_x=\alpha'(u)w w_x=\alpha'(u)w_x(\alpha v_x-\tilde{\lambda}v)\\
	&=\alpha'(u) v_x(z+\tilde{\lambda}w)-\alpha'(u) w_x \tilde{\lambda}v\\
	&=\alpha'(u) v_x z+\alpha'(u)\tilde{\lambda}( v_x w-w_x v)\\
	&=(\alpha(u))_x z_x+\alpha'(u)(v_x z-v z_x)+\alpha'(u)\tilde{\lambda}(v_x w-w_x v).
	\end{aligned}
	\label{eq3}
	\end{equation}
Combining \eqref{eq1},  \eqref{eq2} and  \eqref{eq3} we obtain that $z$ solves the following equation
\begin{equation*}
	z_t+(\tilde{\la} z)_x-(\alpha(u) z_x)=\phi,
\end{equation*}
with
\begin{equation}
\begin{aligned}
&\phi=\big(\lambda'(u)\alpha(u)-2\alpha'(u)\tilde{\lambda}-\alpha''(u)v \alpha(u)\big)
(w_xv-v_x w)+\alpha'(u)(v_x z-v z_x).
\end{aligned}
\label{phi}
\end{equation}
Now, let us recall the following crucial result from \cite{BB-vv-lim-ann-math}.
\begin{lemma}\label{lemma-sc-1}
	Let $\zeta_1,\zeta_2$ be solutions of the following equations for some $\varphi_1$ and $\varphi_2$.
	\begin{equation}
	\begin{aligned}
		&\zeta_{1,t}+(\la \zeta_1)_x-(\alpha_1\zeta_{1,x})_x=\varphi_1,\\
		&\zeta_{2,t}+(\la \zeta_2)_x-(\alpha_1\zeta_{2,x})_x=\varphi_2.
	\end{aligned}
	\label{8.6}
	\end{equation}
	For each $t$, we assume that $x\mapsto \zeta_1(t,x),\,x\mapsto \zeta_2(t,x)$ and $x\mapsto \la(t,x)$ are $C^{1,1}$. Then, we have
	\begin{equation}
	\begin{aligned}
		\frac{d}{dt}\mathcal{A}(t)\leq& -\int\limits_{\R}\alpha_1(t,x)\abs{\zeta_{1,x}(t,x)\zeta_2(t,x)-\zeta_1(t,x)\zeta_{2,x}(t,x)}\,dx\\
		&+\norm{\zeta_1(t)}_{L^1}\norm{\varphi_2(t)}_{L^1}+\norm{\zeta_2(t)}_{L^1}\norm{\varphi_1(t)}_{L^1},
	\end{aligned}
	\label{8.11}
	\end{equation}
	where $\mathcal{A}$ is defined as below
	\begin{equation}
		\mathcal{A}(t)=\frac{1}{2}\int\int\limits_{x<y}\abs{\zeta_1(t,x)\zeta_2(t,y)-\zeta_1(t,y)\zeta_2(t,x)}\,dxdy.
	\end{equation}
\end{lemma}
\begin{proof} The proof follows the same line as in \cite{BB-vv-lim-ann-math} except that we deal with a variable viscosity coefficient $\alpha_1$. For completeness we recall the proof. Let us set:
\begin{equation}
\gamma(t,x)=\big(\int^x_{-\infty} \zeta_1(t,y)dy,\int^x_{-\infty} \zeta_2(t,y)dy\big)
\end{equation}
Integrating \eqref{8.6}, we get:
\begin{equation}
\gamma_t+\lambda \gamma_x-\alpha_1 \gamma_{xx}=\varphi(t,x)=\big(\int^x_{-\infty} \varphi_1(t,y)dy,\int^x_{-\infty} \varphi_2(t,y)dy\big)
\label{8.8}
\end{equation}
As in \cite{BB-vv-lim-ann-math}, for the curve $\gamma$, at each point where $\gamma_x\ne 0$ we consider the unit vector $n(x)$ oriented so that $\gamma_x (x)\wedge n=|\gamma_x(x)|>0$. In particular for every vector $w\in\R^2$ we have
$$\gamma_x(x)\wedge v=|\gamma_x (x)|\;\langle n(x),v\rangle.$$
We define also the projection of $\gamma$ along $n$ by:
$$y\rightarrow \chi^n(y)=\langle n,\gamma(y)\rangle.$$
Following the same computations as \cite{BB-vv-lim-ann-math}, we can show that:
$$
\begin{aligned}
&\frac{d\mathcal{A}(t) }{dt}=\frac{1}{2}\int|\gamma_x(t,x)|\frac{d}{dt}(TV(\chi^{n(t,x)}))dx
\end{aligned}
$$
where the derivative with respect to time of the total variation is considered assuming t $n(t,x)$ is like a constant vector in time. If we assume that $\chi^{n(t,x)}$ has a finite number of local minima and maxima $y_{-p},\cdots, y_{-1},y_0=x,y_1,\cdots y_q$ then we can check that:
\begin{equation}
\frac{d}{dt}(TV(\chi^{n(t,x)}))=-\mbox{sign}\langle n(t,x),\gamma_{xx}(t,x)\rangle 2\sum_{-p\leq \alpha\leq q}(-1)^\alpha\langle n(t,x), \gamma_t(t,y_\alpha(t))\rangle
\label{8.13}
\end{equation}
Furthermore we have as in \cite{BB-vv-lim-ann-math}:
$$\langle n(t,x),\gamma_x(t,x)\rangle=0,\;\;\mbox{sign}\langle n(t,x),\gamma_{xx}(t,y_\alpha)\rangle=(-1)^{\alpha}\mbox{sign}\langle n(t,x),\gamma_{xx}(t,x)\rangle,$$
One obtain then:
\begin{equation}
\begin{aligned}
\frac{d \mathcal{A}(t)}{dt}&=-\int|\gamma_x(t,x)|\mbox{sign}\langle n(t,x),\gamma_{xx}(t,x)\rangle\left(\sum_{-p\leq\alpha\leq q}(-1)^{\alpha}\langle n(t,x),\gamma_t(t,y_\alpha(t))\rangle\right) dx\\
&\leq-\int\sum_\alpha \alpha_1(t,y_\alpha(t))|\gamma_x(t,x)\wedge\gamma_{xx}(t,y_\alpha(t))|dx\\
&+\int |\gamma_x(t,x)||\sum_\alpha (-1)^{\alpha}\langle n(t,x),\varphi(t,y_\alpha(t))\rangle| dx\\
&\leq-\int \alpha_1(t,x)|\gamma_x(t,x)\wedge\gamma_{xx}(t,x)|dx+\int \int |\gamma_x(t,x)\wedge\varphi_x(t,y)| dx dy.
\end{aligned}
\end{equation}
This last inequality implies \eqref{8.11}. As in \cite{BB-vv-lim-ann-math} using an approximation argument we can assume that the function $\chi^{n(t,x)}$ has the previous regularity.
\end{proof}

\noi\textbf{Applications:}
\begin{enumerate}
	\item We may take $\zeta_1=v_1$ and $\zeta_2=w_1$. From Lemma \ref{lemma-sc-1} we can obtain
	\begin{equation*}
		\frac{d}{dt}\mathcal{A}(t)\leq -\int\limits_{\R}\alpha_1(u_1(t,x))\abs{v_{1,x}(t,\cdot)w_1(t,\cdot)-w_{1,x}(t,\cdot)v_1(t,\cdot)}\,dx.
	\end{equation*}
	Hence, we get for all $0\leq T\leq T'$
	\begin{equation}\label{estimate-1}
		\int\limits_{T}^{T'}\int\limits_{\R}\alpha_1(u_1)\abs{v_{1,x}(t,\cdot)w_1(t,\cdot)-w_{1,x}(t,\cdot)v_1(t,\cdot)}\,dx\leq \norm{v_1(T)}_{L^1}\norm{w_1(T)}_{L^1}.
	\end{equation}
	
	\item We may take $\zeta_1=v_1$ and $\zeta_2=z_1=\alpha_1(u_1)w_{1,x}-\tilde{\lambda} w_1$. Define $\phi$ as in \eqref{phi}, from Lemma \ref{lemma-sc-1} we can obtain
	\begin{align}
		\frac{d}{dt}\mathcal{A}(t)&\leq -\int\limits_{\R}\alpha_1(u_1(t,x))\abs{v_{1,x}(t,\cdot)z_1(t,\cdot)-z_{1,x}(t,\cdot)v_1(t,\cdot)}\,dx\nonumber\\
		&+\norm{v_1(t)}_{L^1}\norm{\phi(t)}_{L^1}.\label{estimate-2}
	\end{align}
	We first note that
	\begin{align}
		&z_{1,x}v_1-v_{1,x}z_1
		=(\alpha_1 w_{1,x}-\tilde{\lambda}_1w_1)_xv_1-v_{1,x}(\alpha_1w_{1,x}-\tilde{\lambda}_1w_1)  \nonumber\\
		&=\alpha_{1,x} w_{1,x}v_1+\alpha_1 w_{1,xx} v_1-\tilde{\lambda}_{1,x} w_1 v_1-\tilde{\lambda}_1w_{1,x}v_1-\alpha_1 v_{1,x}w_{1,x}+\tilde{\lambda}_1 v_{1,x}w_1 \nonumber\\
		&=\alpha_{1} w_{1,xx}v_1-(w_1+\tilde{\lambda}_1v_1)w_{1,x}+\alpha_{1,x} w_{1,x}v_1-\tilde{\lambda}_{1,x} w_1 v_1+\tilde{\lambda}_1 v_{1,x}w_1-\tilde{\lambda}_1 v_{1}w_{1,x}
		\nonumber\\
		&=\alpha_{1} w_{1,xx}v_1-w_1w_{1,x}-\tilde{\lambda}_1v_1w_{1,x}+\alpha_{1,x} w_{1,x}v_1-\tilde{\lambda}_{1,x} w_1 v_1+\tilde{\lambda}_1 v_{1,x}w_1-\tilde{\lambda}_1 v_{1}w_{1,x} \nonumber\\
		&=\alpha_{1} w_{1,xx}v_1-w_1(\alpha_1 v_{1,x}-\tilde{\lambda}_1v_1)_x-\tilde{\lambda}_1v_1w_{1,x}+\alpha_{1,x} w_{1,x}v_1-\tilde{\lambda}_{1,x} w_1 v_1\nonumber\\
		&+\tilde{\lambda}_1 v_{1,x}w_1-\tilde{\lambda}_1 v_{1}w_{1,x} \nonumber\\
		&=\alpha_1(w_{1,xx} v_1-w_1v_{1,xx})+(2\tilde{\lambda}_1-\alpha'_1(u_1) v_1)(w_{1,x}v_1-v_{1,x}w_1).\label{identity-1}
	\end{align}
	From \eqref{estimate-2} we get for all $0\leq T\leq T'$
	\begin{align*}
		\int\limits_{T}^{T'}\int\limits_{\R}\alpha_1(u_1)\abs{z_{1,x} v_1-v_{1,x}z_1}\,dxdt&\leq \norm{z_1(T)}_{L^1}\norm{v_1(T)}_{L^1}\\
		&+\norm{v_1}_{L^\f((T,T'),L^1(\R))}\norm{\phi}_{L^1([T,T']\times\R)}.
	\end{align*}
	Now, we apply \eqref{phi}, \eqref{estimate-1}, \eqref{estimate-2} and \eqref{identity-1} to get
	\begin{align}
		&\int\limits_{T}^{T'}\int\limits_{\R}\alpha_1^2(u_1)\abs{w_{1,xx} v_1-w_1v_{1,xx}}\,dxdt \nonumber\\
		&\leq \int\limits_{T}^{T'}\int\limits_{\R}\alpha_1(u_1)\abs{z_{1,x} v_1-v_{1,x}z_1}\,dxdt \nonumber\\
		&\hspace{2cm}+\|2\tilde{\lambda}_1-\alpha'_1(u_1) v_1\|_{L^\infty([T,T']\times\R)}\int\limits_{T}^{T'}\int\limits_{\R}\alpha_1(u_1)\abs{w_{1,x} v_1-v_{1,x}w_1}\,dxdt \nonumber\\
		&\leq \norm{z_1(T)}_{L^1}\norm{v_1(T)}_{L^1}+\norm{v_1}_{L^\f((T,T');L^1(\R))}\norm{\phi}_{L^1([T,T']\times\R)} \nonumber\\
		&\hspace{4cm}+\|2\tilde{\lambda}_1-\alpha'_1(u_1) v_1\|_{L^\infty([T,T']\times\R)}\norm{v_1(T)}_{L^1}\norm{w_1(T)}_{L^1} \nonumber\\
	&\leq \norm{z_1(T)}_{L^1}\norm{v_1(T)}_{L^1}+\|2\tilde{\lambda}_1-\alpha'_1(u_1) v_1\|_{L^\infty([T,T']\times\R)}\norm{v_1(T)}_{L^1}\norm{w_1(T)}_{L^1} \nonumber\\
	&+\norm{v_1}_{L^\f((T,T');L^1(\R))}\|\lambda_1'(u_1)-2\frac{\alpha_1'(u_1)}{\alpha_1(u_1)}\tilde{\lambda}-\alpha_1''(u_1)v_1 \|_{L^\infty((T,T')\times\R)} \norm{v_1(T)}_{L^1}\norm{w_1(T)}_{L^1} \nonumber\\
	&+\norm{v_1}_{L^\f((T,T');L^1(\R))} \|\alpha_1'(u_1)\|_{L^\infty((T,T')\times\R)} \|\alpha_1(w_{1,xx} v_1-w_1v_{1,xx})\|_{L^1((T,T')\times\R)} \nonumber\\
	&+	\norm{v_1}_{L^\f((T,T');L^1(\R))} \|\alpha_1'(u_1)\|_{L^\infty((T,T')\times\R)}\times \nonumber\\
	&\hspace*{3cm}\norm{\frac{1}{\alpha_1(u_1)}(2\tilde{\lambda}_1-\alpha'_1(u_1) v_1)}_{L^\infty((T,T')\times\R)}  \norm{v_1(T)}_{L^1}\norm{w_1(T)}_{L^1}.\label{5.19}
	\end{align}
\end{enumerate}
We take now $T=\hat{t}$ and $T'=T$ with $T$ defined in \eqref{10.3}. We can note in particular that 
$\abs{w_{1,xx} v_1-w_1v_{1,xx}}$ and $|z_{1,x}v_1-v_{1,x}z_1|$ are integrable on $[\hat{t},T]\times\R$ due to the estimates of the Lemma \ref{lemme4.1}. From the Lemma \ref{lemme4.1} and using \eqref{condialpha} we deduce that
\begin{equation}
\int\limits_{T}^{T'}\int\limits_{\R}\abs{w_{1,xx} v_1-w_1v_{1,xx}}\,dxdt=O(1)\delta_0^2+O(1)\delta_0
 \|\alpha_1(w_{1,xx} v_1-w_1v_{1,xx})\|_{L^1((\hat{t},T)\times\R)}.
\label{estim1cru}
\end{equation}
We deduce then since $\delta_0>0$ is chosen sufficiently small that
%\begin{proposition}
	\begin{equation}
\int\limits_{T}^{T'}\int\limits_{\R}\abs{w_{1,xx} v_1-w_1v_{1,xx}}\,dxdt=O(1)\delta_0^2.
\label{estim2cru}
\end{equation}
Similarly using \eqref{estimate-1} we have
\begin{equation}
\int\limits_{T}^{T'}\int\limits_{\R}\abs{w_{1,x} v_1-w_1v_{1,x}}\,dxdt=O(1)\delta_0^2.
\label{estim2crubis}
\end{equation}
\subsection{Length Functional}
As in \cite{BB-triangular}, we wish now to deal with the terms of the form $O(1)v_1^2\left(\frac{w_1}{v_1}\right)_x^2\chi_{\abs{w_1/v_1}\leq3\de_1}$. We follow the main argument of the proof of \cite{BB-vv-lim-ann-math} except that we have variable viscosity coefficient.
\begin{lemma}
Under the assumptions of Lemma \ref{lemme4.1} with $\zeta_{1}=v_1$, $\zeta_{2}=w_1$, $\lambda=\tilde{\lambda}_1$, $\varphi_1=\varphi_2=0$, at a fixed time $t$ assume that $\gamma_x(t,x)\ne 0$ for every $x$ with $\gamma(t,x)=\big(\int^x_{-\infty} v_1(t,y)dy,\int^x_{-\infty} w_1(t,y)dy\big)$. Then
\begin{equation} \label{utile}
\frac{d}{dt}{\cal L}(t)\leq -\frac{1}{(1+9\delta_1^2)^{\frac{3}{2}}}\int_{\{|\frac{w_1}{v_1}|\leq 3\delta_1\}}|\alpha_1 v_1||(\frac{w_1}{v_1})_x|^2 dx,
\end{equation}
with
$${\cal L}(t)=\int_{\R}\sqrt{v_1^2(t)+w_1^2(t)} dx.$$
\end{lemma}
\begin{proof} Let us recall that
\begin{equation}
\gamma_{t,x}+(\tilde{\lambda}_1 \gamma_x)_x-(\alpha_1 \gamma_{xx})_x=0.
\label{8.8bis}
\end{equation}
Since $\gamma_x$ never vanishes, we get after integrations by parts
$$
\begin{aligned}
\frac{d}{dt}{\cal L}(t)=&\int\frac{\langle\gamma_x,\gamma_{xt}\rangle}{\sqrt{|\gamma_x|^2}}dx\\
=&\int \frac{\langle\gamma_x,( \alpha_ 1\gamma_{xx})_x\rangle}{|\gamma_x|}- \frac{\langle\gamma_x,( \tilde{\lambda}_ 1\gamma_{x})_x\rangle}{|\gamma_x|}dx\\
=&\int(\alpha_1|\gamma_x|)_x-\frac{\alpha_1|\gamma_{xx}|^2}{|\gamma_x|}+\alpha_1\frac{\langle\frac{\gamma_x}{|\gamma_x|},\gamma_{xx}\rangle^2}{|\gamma_x|}-(\tilde{\lambda}_1|\gamma_x|)_x.
\end{aligned}
$$
As in \cite{BB-vv-lim-ann-math}, we have $\alpha_1\frac{\langle\frac{\gamma_x}{|\gamma_x|},\gamma_{xx}\rangle^2}{|\gamma_x|}=-\frac{ \alpha_1 |v_1||(\frac{w_1}{v_1})_x|^2}{(1+(\frac{w_1}{v_1})^2)^{\frac{3}{2}}}$ which implies that
$$
\begin{aligned}
\frac{d}{dt}{\cal L}(t)=-\int \frac{ \alpha_1 |v_1||(\frac{w_1}{v_1})_x|^2}{(1+(\frac{w_1}{v_1})^2)^{\frac{3}{2}}}dx.
\end{aligned}
$$
This last inequality implies clearly \eqref{utile}.
\end{proof}
From \eqref{utile}, \eqref{condialpha} and applying the Lemma \ref{lemme4.1}, we deduce that
\begin{equation}
\int_{\hat{t}}^T\int_{\R} v_1^2\left(\frac{w_1}{v_1}\right)_x^2\chi_{\abs{w_1/v_1}\leq3\de_1} dx ds =O(1)\delta_0^2.
\label{estim3cru}
\end{equation}
%\end{proposition}
\section{Transversal interaction and energy estimates}\label{sec:transversal}
In this subsection we recall the estimate for transversal interaction (see \cite{BB-vv-lim-ann-math,HJ})
\begin{proposition}
	Let $v,v^\#$ be solutions of the following equations
	\begin{equation}
	\begin{aligned}
		v_t+(\la v)_x-(\al v)_{xx}&=\phi,\\
		v^\#_t+(\la^\# v^\#)_x-(\al^\# v^\#)_{xx}&=\phi^\#.
	\end{aligned}
	\label{5.22}
	\end{equation}
We assume that $\inf \la^\# -\sup\la\geq c_0>0$ and that $\|(\alpha,\alpha^\#)\|_{L^\infty}<+\infty$. Then the following estimate holds true for any $0\leq T\leq T'$
\begin{align*}
	&\int\limits_T^{T'}\int\limits_{\R}\abs{vv^\#}\,dxdt\\
	&\leq \frac{1}{c_1}\left(\norm{v(T,\cdot)}_{L^1(\R)}+\norm{\phi}_{L^1([T,T']\times\R)}\right)\cdot \left(\norm{v^\#(T,\cdot)}_{L^1(\R)}+\norm{\phi^\#}_{L^1([T,T']\times\R)}\right),
\end{align*}
where $c_1=2c_0$. 
\label{prop5.2}
\end{proposition}
Similarly we have the following proposition.
%\subsection{Energy estimate}
\begin{proposition}
	Let $v,v^\#$ be solutions of \eqref{5.22} and we assume that $\inf \la^\# -\sup\la\geq c_0$ holds  with in addition the following estimates
	\begin{align}
		&\int\limits_{T}^{T'}\int\limits_{\R}\abs{\phi(t,x)}dxdt\leq \de_0,\quad\quad 	\int\limits_{T}^{T'}\int\limits_{\R}\abs{\phi^\#(t,x)}dxdt\leq \de_0,\\
		&\norm{v(t)}_{L^1},\,\norm{v^\#(t)}_{L^1}\leq\de_0,\quad\quad \norm{v_x(t)}_{L^1},\,\norm{v^\#(t)}_{L^\f}\leq C_*\de^2_0,\\
		&\norm{\la_x(t)}_{L^\f},\,\norm{\la_x(t)}_{L^1}\leq C_*\de_0,\quad\quad \lim\limits_{x\rr-\f}\la(t,x)=0,
	\end{align}
	for all $t\in[T,T']$. Then we have
	\begin{equation}\label{est:transversal-2}
		\int\limits_{0}^{T}\int\limits_{\R}\abs{v_x(t,x)}\abs{v^\#(t,x)}\,dxdt=\mathcal{O}(1)\de^2_0.
	\end{equation}
\label{prop5.3}
\end{proposition}
We can apply directly the Proposition \ref{prop5.2} to $v=v_1$ and $v^\#=v_2$. We have then $\lambda=\lambda_1$, $\phi=0$, $ \la^\#=\la_2-(r_2\bullet\al_2)v_2+u_x\bullet\al_2=\la_2+(r_1\bullet\al_2)v_1$, $\phi^\#=\varphi_2$. We can observe that using \eqref{condhyperbo}, the Lemma \ref{lemme4.1} and taking $\delta_0$ sufficiently small we deduce that $\inf \la^\# -\sup\la\geq c_0>0$ with $c_0=\frac{c}{2}$. It gives by applying again Lemma \ref{lemme4.1} that
\begin{equation}
\int\limits_{T}^{T'}\int\limits_{\R}\abs{v_1v_2}\,dxdt=O(1)\delta_0^2.
\label{estim4cru}
\end{equation}
Similarly by applying the Proposition \ref{prop5.3} we can show that
\begin{equation}\label{estim5cru}
\int\limits_{T}^{T'}\int\limits_{\R}\abs{v_1v_{2,x}}\,dxdt=O(1)\delta_0^2,\;\;\int\limits_{T}^{T'}\int\limits_{\R}\abs{v_{1,x}v_{2}}\,dxdt=O(1)\delta_0^2.
\end{equation}
As in \cite{BB-vv-lim-ann-math}, we are going to prove energy estimate. Define now  a cut-off positive function $\hat{\theta}$ satisfying %with $c_1$ defined in \eqref{condialpha}
\begin{equation}
\begin{aligned}
\hat{\theta}(x)=\begin{cases}
\begin{aligned}
&0\;\;\;|x|\leq \frac{3\delta_1}{5}, \\%\max(\frac{3\delta_1}{5},\frac{3}{c_1})\\
&\mbox{smooth connection}\;\;\  \frac{3\delta_1}{5}  % \max(\frac{3\delta_1}{5},\frac{3}{c_1})
\leq |x|\leq \frac{4\delta_1}{5},\\% \max(\frac{4}{c_1},\frac{4\delta_1}{5})\\
&1\;\;\;|x|\geq\frac{4\delta_1}{5}.% \max(\frac{4}{c_1},\frac{4\delta_1}{5}).
\end{aligned}
\end{cases}
\end{aligned}
\label{6.11}
\end{equation}
We assume here that $\delta_1|\hat{\theta}'|,\delta_1^2|\hat{\theta}''|\leq 16$. 
Let us multiply now \eqref{v1} by $v_1\theta\left(\frac{w_1}{v_1}\right)$ and integrate by parts, we obtain
$$
\begin{aligned}
&\int_{\R}\biggl((\frac{v_1^2\hat{\theta}}{2})_t-\frac{v_1^2}{2}(\hat{\theta}_t+2\tilde{\lambda}_1\hat{\theta}_x-(\alpha_1\hat{\theta}_x)_x)+2\alpha_1 v_{1,x}v_1\hat{\theta}_x+\hat{\theta} v_{1,x}(\alpha_1 v_{1,x}-\tilde{\lambda}_1 v_1)\biggl) dx=0.
\end{aligned}
$$
It implies that
\begin{align}
\int_\R \hat{\theta} v_{1,x}(\alpha_1 v_{1,x}-\tilde{\lambda}_1 v_1)=&- \int_{\R}(\frac{v_1^2\hat{\theta}}{2})_t dx+\int_\R \frac{v_1^2}{2}(\hat{\theta}_t+\tilde{\lambda}_1\hat{\theta}_x-(\alpha_1\hat{\theta}_x)_x)dx\nonumber\\
&-\int_{\R} 2\alpha_1 v_{1,x}v_1\hat{\theta}_x+\int_\R \frac{v_1^2}{2}\tilde{\lambda}_1\hat{\theta}_x.
\label{formenergie}
\end{align}
A direct computation gives using \eqref{formule2},
\begin{equation*}
\begin{aligned}
&\hat{\theta}_t+\tilde{\lambda}_1\hat{\theta}_x-(\alpha_1\hat{\theta}_x)_x\\
&=\hat{\theta}'(\frac{w_{1,t}}{v_1}-\frac{v_{1,t}w_1}{v_1^2})+\tilde{\lambda}_1(\frac{w_{1,x}}{v_1}-\frac{v_{1,x}w_1}{v_1^2})\\
&-\alpha'_1 v_1\hat{\theta}'\left(\frac{w_1}{v_1}\right)_x-\alpha_1\hat{\theta}''\left(\frac{w_1}{v_1}\right)_x^2
+2\alpha_1\hat{\theta}'\frac{v_{1,x}}{v_1}\left(\frac{w_1}{v_1}\right)_x-\alpha_1\hat{\theta}'(\frac{w_{1,xx}v_1-v_{1,xx}w_1}{v_1^2})\\
&=\frac{\hat{\theta}'}{v_1}(w_{1,t}+\tilde{\lambda}_1 w_{1,x}-(\alpha_1 w_{1,x})_x)-\frac{\hat{\theta}' w_1}{v_1^2}(
v_{1,t}+\tilde{\lambda}_1 v_{1,x}-(\alpha_1 v_{1,x})_x)-\alpha_1\hat{\theta}''\left(\frac{w_1}{v_1}\right)_x^2
\\
&+2\alpha_1\hat{\theta}'\frac{v_{1,x}}{v_1}\left(\frac{w_1}{v_1}\right)_x\\
&=-\alpha_1\hat{\theta}''\left(\frac{w_1}{v_1}\right)_x^2
+2\alpha_1\hat{\theta}_x\frac{v_{1,x}}{v_1}.
\end{aligned}
\end{equation*}
Putting this last expression in \eqref{formenergie}, it yields
\begin{equation}\label{formenergie1}
\begin{aligned}
\int_\R \hat{\theta} v_{1,x}(\alpha_1 v_{1,x}-\tilde{\lambda}_1 v_1)&=- \int_{\R}(\frac{v_1^2\hat{\theta}}{2})_t dx-\frac{1}{2}\int_\R \alpha_1\hat{\theta}'' v_1^2\left(\frac{w_1}{v_1}\right)_x^2dx\\
&-\int_{\R} \alpha_1 v_{1,x}v_1\hat{\theta}_x+\int_\R \frac{v_1^2}{2}\tilde{\lambda}_1\hat{\theta}_x.
\end{aligned}
\end{equation}
We observe that $\hat{\theta}(\frac{w_1}{v_1})\ne 0$ if
$|w_1|=|\alpha_1 v_{1,x}-\tilde{\lambda}_1 v_1|\geq \frac{3\delta_1}{5}|v_1|$ which gives
\begin{equation}
|\alpha_1 v_{1,x}|\geq \frac{3\delta_1}{5} |v_1|-\|\tilde{\lambda}_1\|_{L^\infty}|v_1|\geq 2\|\tilde{\lambda}_1\|_{L^\infty}|v_1|,
\label{impoti}
\end{equation}
if we assume that $\delta_1$ satisfies:
\begin{equation}
\delta_1\geq 10\|\lambda'_1\|_{L^\infty}\delta_0+2\delta_0\|\alpha'_1\|_{L^\infty}\geq 5\|\tilde{\lambda}_1\|_{L^\infty}.
\label{5.7}
\end{equation}
From \eqref{formenergie1}, \eqref{impoti}, \eqref{condialpha} we get for $t\in[\hat{t},T]$
\begin{equation}
\begin{aligned}
\frac{c_1}{2 }\int_\R \hat{\theta} v_{1,x}^2&\leq- \int_{\R}\left(\frac{v_1^2\hat{\theta}}{2}\right)_t dx+O(1)\int_{\{ |\frac{w_1}{v_1}|\leq \delta_1\} } v_1^2 | \hat{\theta}'' |\left(\frac{w_1}{v_1}\right)_x^2dx\\
&+O(1)\int_\R(1+|\frac{v_{1,x}}{v_1}|  \chi_{\{|\frac{w_1}{v_1}|\leq\frac{4\delta_1}{5}\}} )|w_{1,x}v_1-v_{1,x}w_1|  dx.
\end{aligned}
\label{formenergie3}
\end{equation}
We can now estimate the term $v_{1,x}(w_1+\si_1v_1)$, using the fact that this term is different from zero only if $|\frac{w_1}{v_1}|\geq\delta_1$ we get from \eqref{formenergie1}, \eqref{estim3cru}, \eqref{estim2crubis} and Lemma \ref{lemme4.1}
\begin{equation}
\begin{aligned}
\int^T_{\hat{t}}\int_{\R}|v_{1,x}(w_1-\sigma_1 v_1)| dx ds&\leq O(1)\int^T_{\hat{t}}\int_{\R} \chi_{\{|\frac{w_1}{v_1}|\geq\delta_1\}} (v_{1,x}^2+|v_1v_{1,x}|) ds dx\\
&\leq O(1)\int^T_{\hat{t}}\int_{\R}  \chi_{\{|\frac{w_1}{v_1}|\geq\delta_1\}} v_{1,x}^2 ds dx\leq O(1)\delta_0^2.
\end{aligned}
 \label{estim6cru}
\end{equation}
Finally using \eqref{4.19}, \eqref{estim2cru}, \eqref{estim2crubis}, \eqref{estim3cru}, \eqref{estim4cru}, \eqref{estim5cru} and \eqref{estim6cru} we have proved that
$$
\int_{\hat{t}}^T \int_{\R}|\phi_2(s,x)| ds dx\leq O(1) \delta_0^2,$$
which contradict the definition of $T$ provided that $\delta_0$ is chosen sufficiently small. In particular since $T=+\infty$ we deduce by maximum principle that we have for any $t\in\R$
\begin{equation}
\|u_{1,x}(t,\cdot)\|_{L^1}+\|u_{2,x}(t,\cdot)\|_{L^1}\leq\delta_0.
\end{equation}
\section{Stability}\label{section:stability}
We wish now to get stability estimate, in other words let us consider two initial data $\bar{u}$ and $\bar{v}$ of the equation \eqref{eqn-main} satisfying the smallness assumption on $\mbox{Tot.Var.}\bar{u}$ and $\mbox{Tot.Var.}\bar{v}$ such that \eqref{eqn-main} admits respectively a global solution $u$ and $v$ associated to the initial data $\bar{u}$ and $\bar{v}$ with small Total Variation. In addition we assume that $(\bar{u}-\bar{v})\in L^1(\R)$ then we are going to prove that there exists a constant $L$ depending on $\delta_0$ such that for any $t\in\R^+$ we have
 \begin{equation}
 \|u(t)-v(t)\|_{L^1}\leq\|\bar{u}-\bar{v}\|_{L^1}.
 \label{stab1}
 \end{equation}
 To do this following \cite{BB-vv-lim-ann-math}, we employ an homotopy argument, setting $\bar{u}^\theta=\theta\bar{u}+(1-\theta)\bar{v}$ with $\theta\in[0,1]$ and denoting by $u^\theta$ the solution of 
  \eqref{eqn-main} with this initial data (this solution exists due to the Theorem \ref{theo1}) we are going to consider the unknown $z^\theta(t,x)=\frac{d u^\theta}{d\theta}(t,x)$ which satisfies the following system
\begin{equation}
\begin{cases}
\begin{aligned}
&	z^\theta_t+(A(u)z^\theta)_x=\left(B(u)z^{\theta}_x+z^\theta\bullet B(u^\theta) u^\theta_x \right)_x,\\
&z^\theta(0,\cdot)=\bar{u}-\bar{v}.
\end{aligned}
\end{cases}
\end{equation}
Indeed we can observe that $u^\theta_x\bullet A(u^\theta) z^\theta= z^\theta\bullet A(u^\theta) u^\theta_x$. It remains now to estimate the $L^1$ norm of $z^\theta$ all along the time, in particular if we show that for any $t\geq 0$
\begin{equation}
\|z^\theta(t,\cdot)\|_{L^1}\leq L \|z^\theta(0,\cdot)\|_{L^1},
\label{stab2}
\end{equation}
 then we deduce that for any $t\geq 0$
 $$
 \begin{aligned}
 &\|u(t)-v(t)\|_{L^1}\leq \int^1_0\|z^\theta(s)\| ds\leq L\|\bar{u}-\bar{v}\|_{L^1},
 \end{aligned}
 $$
 which is exactly the estimate \eqref{stab1} what conclude the proof of the stability. It remains now to prove the estimate \eqref{stab2}. We are reduced to study the $L^1$ norm of the solution of the following equation with $u$ satisfying  \eqref{eqn-main} 
\begin{equation}
	h_t+(Ah)_x=(Bh)_{xx}+(h\bullet B u_x-u_x\bullet B  h)_x.
	\label{equationh}
\end{equation}
\subsection{Parabolic estimates}
As in the previous section, we can get regularity estimates on $h$, by applying similar arguments (it is clear since $h$ satisfies the same equation than $u_x$ except that the initial data is a priori different).
\begin{proposition}\label{prop:parabolic2}
	There exists $C_1,C_2,C_3>0$ depending on $\kappa_A,\kappa_B,\kappa$ such that if $h$ is the solution of the equation \eqref{equationh} satisfying 
	\begin{equation}\label{assumption:h-L1}
		\norm{(h)_1(t,\cdot)}_{L^1}+\norm{(h)_2(t,\cdot)}_{L^1}\leq \de_0\mbox{ for all }t\in[0,\hat{t}]\mbox{ where }\hat{t}\leq \frac{1}{(C_3\delta_0)^2},
	\end{equation}
	for $C_3$ sufficiently large depending on  $\kappa_A,\kappa_B, \kappa$ and $\delta_0$ sufficiently small in terms of $C_1,C_2, C_3$, $\kappa_A$, $\kappa_B$ and $\kappa$ then
	%for some $\de_0<\min(\frac{1}{2},\frac{2-\sqrt{2}}{96\sqrt{2} C_3\kappa_B^13\kappa^2\kappa_A^2},\frac{1}{12 \kappa_B^3})$ with $C_3>0$ sufficiently large and and $\kappa,\kappa_A,\kappa_B$ are as in \eqref{def:kappa}--\eqref{def:kappa-P}. Then 
	we have for $i\in\{1,2\}$
	\begin{equation}\label{estimate:parabolic-h-1}
		\norm{h_{i,x}(t,\cdot)}_{L^1}\leq \frac{C_1\delta_0}{\sqrt{t}}, \norm{h_{i,xx}(t,\cdot)}_{L^1}\leq 
		\frac{C_2\de_0}{t},\mbox{ and }\norm{h_{i,xxx}(t,\cdot)}_{L^1}\leq %\frac{2500\kappa_1^6\kappa^3\kappa_A\kappa_B^7\kappa_P^{13}\delta_0}{\sqrt{t}.
		\frac{C_3}{t^{3/2}}.
	\end{equation}
\end{proposition}
\begin{corollary}
\label{coro2.3} Let $T>\hat{t}$. Assume that the solution $h$ of  \eqref{equationh} satisfies on $[0,T]$
$$\|(h)_1(t,\cdot)\|_{L^1}+\|(h)_2(t,\cdot)\|_{L^1}\leq \delta_0,$$
then for all $t\in[\hat{t},T]$ we have
\begin{equation}\label{estimate:parabolic-2-h}
	\begin{aligned}
		&\sum_{i=1}^2\norm{h_{i,x}(t,\cdot)}_{L^1}, \sum_{i=1}^2\norm{h_{i}(t,\cdot)}_{L^\infty},=O(1)\delta_0^2,\\
		&\sum_{i=1}^2 \norm{h_{i,xx}(t,\cdot)}_{L^1},\sum_{i=1}^2 \norm{h_{i,x}(t,\cdot)}_{L^\infty}=O(1)\delta_0^3,\\
		&\sum_{i=1}^2\norm{h_{i,xxx}(t,\cdot)}_{L^1},\sum_{i=1}^2\norm{h_{i,xx}(t,\cdot)}_{L^\infty}=O(1)\delta_0^4.
		\end{aligned}
		\end{equation}
\end{corollary}
\begin{proof} It suffices to apply Proposition \eqref{prop:parabolic} on $[t-\hat{t},t]$.
\end{proof}
\begin{proposition}
\label{prop2.3-h}
There exists $C>0$ large enough depending on $\kappa_A$, $\kappa_B$, $\kappa$ such that
 if the solution $h$ of  \eqref{equationh} verifies:
 \begin{equation}
 \|h(0,\cdot)\|_{L^1}\leq\frac{\delta_0}{C},
 \label{2.16-h}
 \end{equation}
 then $h$ satisfies on the whole interval $[0,\hat{t}]$:
 \begin{equation}
 \|h(t,\cdot)\|_{L^1}\leq \frac{\delta_0}{2}.
 \label{2.17-h}
 \end{equation}
\end{proposition}

\subsection{Gradient decomposition}
We wish now to  estimate on $h(t,\cdot)$ in $L^1$ when $t\geq\hat{t}$. As previously we decompose
$h$ as a sum of viscous travelling waves
%We decompose as in \eqref{decomp:u-x}
\begin{equation}\label{decomp-h}
	h=h_1\tilde{r}_1(u,v_1,\sigma_1)+h_2r_2,
\end{equation}
with the same choice on $\sigma_1$ as in \eqref{formulesigma}. In the sequel we distinguish $(h)_2$ which is the second coordinate of $h$ whereas $h_2$ is the new variable introduced in \eqref{decomp-h}. From the previous estimate on $h$ we deduce that
\begin{lemma}  Let $T>\hat{t}$. Assume that the solution $h$ of  \eqref{equationh} satisfies on $[0,T]$
$$\|(h)_{1}(t,\cdot)\|_{L^1}+\|(h)_2(t,\cdot)\|_{L^1}\leq \delta_0,$$
then for all $t\in[\hat{t},T]$ we have
\begin{equation}%\label{estimate:parabolic-2}
	\begin{aligned}
	&\sum_{i=1}^2\norm{h_i(t,\cdot)}_{L^1}=O(1)\delta_0\\
		&\sum_{i=1}^2\norm{h_{i,x}(t,\cdot)}_{L^1}, \sum_{i=1}^2\norm{h_{i}(t,\cdot)}_{L^\infty} =O(1)\delta_0^2\\
		%\frac{16 \kappa \kappa_B^6\de_0}{\sqrt{t}},\\
		&\sum_{i=1}^2 \norm{h_{i,xx}(t,\cdot)}_{L^1},\sum_{i=1}^2 \norm{h_{i,x}(t,\cdot)}_{L^\infty}=O(1)\delta_0^3%\frac{2500\kappa_1^6\kappa^3\kappa_A\kappa_B^7\kappa_P^{13}\delta_0}{\sqrt{t}.
		,\\
		&\sum_{i=1}^2\norm{h_{i,xxx}(t,\cdot)}_{L^1}, \sum_{i=1}^2\norm{h_{i,xx}(t,\cdot)}_{L^\infty}=O(1)\delta_0^4.%\frac{2500\kappa_1^6\kappa^3\kappa_A\kappa_B^7\kappa_P^{13}\delta_0}{\sqrt{t}.
		%\frac{C_3}{\sqrt{t}t}.
		\end{aligned}
		\end{equation}
\label{lemme4.2}
\end{lemma}
We calculate
\begin{align*}
	&(Bh)_x+(h\bullet B u_x-u_x\bullet B  h)	-Ah\\
	&=(\al_1h_1\tilde{r}_1+(s_1h_1+\al_2 h_2)r_2)_x-h_1A\tilde{r}_1-\la_2h_2r_2+(h_1v_2-h_2v_1) (\tilde{r}_1\cdot DB r_2-r_2\cdot DB\tilde{r}_1)\\
	&=\al_1h_{1,x}\tilde{r}_1+h_1\left[-A\tilde{r}_1+\al_1^\p v_1\tilde{r}_1+\al_1v_1\tilde{r}_1\bullet \tilde{r}_1+v_1\tilde{r}_1\bullet s_1r_2\right]+\al_1v_2h_1r_2\bullet \tilde{r}_1\\
	&+\al_1v_{1,x}h_1\tilde{r}_{1,v}+\al_1h_1\si_{1,x}\tilde{r}_{1,\si}+h_{1,x}s_1r_2+v_2h_1r_2\bullet s_1r_2\\
	&-\theta_1^\p \left(\frac{w_1}{v_1}\right)_x h_1 s_{1,\si}r_2+v_{1,x}h_1s_{1,v}r_2+\tilde{r}_1\bullet \al_2 v_1h_2r_2+r_2\bullet \al_2 v_2h_2r_2\\
	&+(\al_2 h_{2,x}-\la_2h_2)r_2+(h_1v_2-h_2v_1) (\tilde{r}_1\cdot DB r_2-r_2\cdot DB\tilde{r}_1).
\end{align*}
By using \eqref{rel:A-v} we have
%	\begin{align*}
%		&(Bh)_x%+(h\cdot DBu_x-u_x\cdot DB h)
		%-Ah\\
		%&=(\al_1h_{1,x}-(\la_1-v_1 \al_1^\p-\si_1)h_1)\tilde{r}_1+(\al_1v_{1,x}-(\la_1-v_1\al_1^\p-\si_1)v_1)h_1\tilde{r}_{1,v}\\
		%&+\frac{1}{\al_1}(\al_1 h_{1,x}-(\la_1-v_1\al_1^\p-\si_1)h_1)s_1r_2+\al_1v_2h_1r_2\bullet \tilde{r}_1\\
		%&+\frac{1}{\al_1}(\al_1 v_{1,x}-(\la_1-v_1\al_1^\p-\si_1)v_1)h_1s_{1,v}r_2-h_1\si_1\tilde{r}_1\\
		%&-\al_1h_1\theta'_1\left(\frac{w_1}{v_1}\right)_x\tilde{r}_{1,\si}+v_2h_1r_2\bullet s_1r_2-\theta_1^\p h_1\left(\frac{w_1}{v_1}\right)_x s_{1,\si}r_2+\tilde{r}_1\bullet \al_2 v_1h_2r_2\\
		%&+r_2\bullet \al_2 v_2h_2r_2+(\al_2 h_{2,x}-\la_2h_2)r_2+(h_1v_2-h_2v_1) (\tilde{r}_1\cdot DB r_2-r_2\cdot DB\tilde{r}_1).
	%\end{align*}
%After a simplification
\begin{align*}
	&(Bh)_x%+(h\cdot DBu_x-u_x\cdot DB h)
	-Ah\\
	&=(\al_1h_{1,x}-(\la_1-v_1 \al_1^\p)h_1)\tilde{r}_1+(\al_1v_{1,x}-(\la_1-v_1\al_1^\p-\si_1)v_1)h_1\tilde{r}_{1,v}\\
	&+\frac{1}{\al_1}(\al_1 h_{1,x}-(\la_1-v_1\al_1^\p-\si_1)h_1)s_1r_2+\al_1v_2h_1r_2\bullet \tilde{r}_1\\
	&+\frac{1}{\al_1}(\al_1 v_{1,x}-(\la_1-v_1\al_1^\p-\si_1)v_1)h_1s_{1,v}r_2\\
	&-\al_1h_1\theta_1^\p \left(\frac{w_1}{v_1}\right)_x\tilde{r}_{1,\si}+v_2h_1r_2\bullet s_1r_2-  h_1\theta_1^\p \left(\frac{w_1}{v_1}\right)_x s_{1,\si}r_2\\
	&+\tilde{r}_1\bullet \al_2 v_1h_2r_2+(\al_2 h_{2,x}-(\la_2-r_2\bullet \al_2 v_2)h_2)r_2\\
	&+(h_1v_2-h_2v_1) (\tilde{r}_1\cdot DB r_2-r_2\cdot DB\tilde{r}_1)
\end{align*}
Taking derivative with respect to $x$ we get
\begin{align*}
	&[(Bh)_x+(h\bullet B u_x-u_x\bullet B  h)
	-Ah]_x\\
	&=(\al_1h_{1,x}-(\la_1-v_1 \al_1^\p)h_1)_x\tilde{r}_1+(\al_1h_{1,x}-(\la_1-v_1 \al_1^\p)h_1)v_1\tilde{r}_1\bullet \tilde{r}_1\\
	&+(\al_1h_{1,x}-(\la_1-v_1 \al_1^\p)h_1)v_2r_2\bullet \tilde{r}_1+(\al_1h_{1,x}-(\la_1-v_1 \al_1^\p)h_1)v_{1,x}\tilde{r}_{1,v}\\
	&-\theta_1^{\p} \left(\frac{w_1}{v_1}\right)_x(\al_1h_{1,x}-(\la_1-v_1 \al_1^\p)h_1)\tilde{r}_{1,\si}+(\al_1v_{1,x}-(\la_1-v_1\al_1^\p-\si_1)v_1)_xh_1\tilde{r}_{1,v}\\
	&+(\al_1v_{1,x}-(\la_1-v_1\al_1^\p-\si_1)v_1)h_{1,x}\tilde{r}_{1,v}+(\al_1v_{1,x}-(\la_1-v_1\al_1^\p-\si_1)v_1)h_1v_1\tilde{r}_1\bullet \tilde{r}_{1,v}\\
	&+(\al_1v_{1,x}-(\la_1-v_1\al_1^\p-\si_1)v_1)h_1v_2r_2\bullet \tilde{r}_{1,v}+(\al_1v_{1,x}-(\la_1-v_1\al_1^\p-\si_1)v_1)h_1v_{1,x}\tilde{r}_{1,vv}\\
	&-(\al_1v_{1,x}-(\la_1-v_1\al_1^\p-\si_1)v_1)h_1\theta_1^\p \left(\frac{w_1}{v_1}\right)_x\tilde{r}_{1,v\si}\\
	&-\frac{\al_1^\p v_1}{\al_1^2}(\al_1 h_{1,x}-(\la_1-v_1\al_1^\p-\si_1)h_1)s_1r_2\\%+\al_1v_2h_1r_2\bullet \tilde{r}_1\\
	&+\frac{1}{\al_1}((\al_1 h_{1,x}-(\la_1-v_1\al_1^\p-\si_1)h_1)s_1)_xr_2+\al^\p_1v_1v_2h_1r_2\bullet \tilde{r}_1+\al_1(v_{2,x}h_1+v_2h_{1,x})r_2\bullet \tilde{r}_1\\
	&+\al_1v_2h_1(r_2\otimes u_x):D^2_u \tilde{r}_1+\al_1v_2h_1v_{1,x}r_2\bullet \tilde{r}_{1,v}-\al_1 \theta_1^\p \left(\frac{w_1}{v_1}\right)_xv_2h_1r_2\bullet \tilde{r}_{1,\si}\\
	&-\frac{\al_1^\p v_1}{\al^2_1}(\al_1 v_{1,x}-(\la_1-v_1\al_1^\p-\si_1)v_1)h_1s_{1,v}r_2+\frac{1}{\al_1}((\al_1 v_{1,x}-(\la_1-v_1\al_1^\p-\si_1)v_1)h_1s_{1,v})_xr_2\\
	&-v_1\al^\p_1 h_1\theta_1^\p \left(\frac{w_1}{v_1}\right)_x \tilde{r}_{1,\si}-\al_1h_1\left[-\theta_1^{\p\p} \left(\frac{w_1}{v_1}\right)_x^2+\theta_1^\p \left(\frac{w_1}{v_1}\right)_{xx}\right]\tilde{r}_{1,\si}\\
	&-\al_1 h_{1,x}\theta_1^\p\left(\frac{w_1}{v_1}\right)_x\tilde{r}_{1,\si}-\al_1  h_1\theta_1^\p \left(\frac{w_1}{v_1}\right)_xv_1\tilde{r}_1\bullet \tilde{r}_{1,\si}-\al_1 h_1\theta_1^\p \left(\frac{w_1}{v_1}\right)_xv_2r_2\bullet \tilde{r}_{1,\si}\\
	&-\al_1 h_1\theta_1^\p \left(\frac{w_1}{v_1}\right)_xv_{1,x}\tilde{r}_{1,v\si}+\al_1 h_1(\theta_1^\p)^2 \left(\frac{w_1}{v_1}\right)^2_x\tilde{r}_{1,\si\si}\\
	&+(v_{2,x}h_1+v_2h_{1,x})r_2\bullet s_1r_2+v_2h_1(u_x\otimes r_2):D^2 s_1r_2+v_2h_1v_{1,x}r_2\bullet s_{1,v}r_2\\
	&-v_2h_1 \theta_1^\p \left(\frac{w_1}{v_1}\right)_xr_2\bullet s_{1,\si}r_2+ h_1\left[\theta_1^{\p\p} \left(\frac{w_1}{v_1}\right)^2_x-\theta_1^\p \left(\frac{w_1}{v_1}\right)_{xx}\right]s_{1,\si}r_2- h_{1,x}\theta_1^\p \left(\frac{w_1}{v_1}\right)_x s_{1,\si}r_2\\
	&- h_1\theta_1^\p \left(\frac{w_1}{v_1}\right)_x v_1\tilde{r}_1\bullet s_{1,\si}r_2-  h_1\theta_1^\p \left(\frac{w_1}{v_1}\right)_x v_2r_2\bullet s_{1,\si}r_2\\
	&+  h_1(\theta_1^\p)^2 \left(\frac{w_1}{v_1}\right)^2_x s_{1,\si\si}r_2- h_1\theta_1^\p \left(\frac{w_1}{v_1}\right)_xv_{1,x} s_{1,v\si}r_2\\
	&+(\tilde{r}_1\otimes u_x):D^2\al_2 v_1h_2r_2+(\tilde{r}_1\bullet \tilde{r}_1)\bullet \al_2 v^2_1h_2r_2+(r_2\bullet \tilde{r}_1)\bullet \al_2 v_1v_2h_2r_2\\
	&+\tilde{r}_{1,v}\bullet \al_2 v_1v_{1,x}h_2r_2-\tilde{r}_{1,\si}\bullet \al_2\theta_1^\p \left(\frac{w_1}{v_1}\right)_xv_1h_2r_2\\
	&+\tilde{r}_1\bullet \al_2 (v_{1,x}h_2+v_1h_{2,x})r_2+(\al_2 h_{2,x}-(\la_2-r_2\bullet \al_2 v_2)h_2)_xr_2\\
	&+(h_1v_2-h_2v_1)_x (\tilde{r}_1\cdot DB r_2-r_2\cdot DB\tilde{r}_1)+(h_1v_2-h_2v_1)v_1 ((\tilde{r}_1\bullet \tilde{r}_1)\cdot DB r_2-r_2\cdot DB(\tilde{r}_1\bullet \tilde{r}_1))\\
	&+(h_1v_2-h_2v_1)v_2 ((r_2\bullet \tilde{r}_1)\cdot DB r_2-r_2\cdot DB(r_2\bullet \tilde{r}_1))+(h_1v_2-h_2v_1)v_{1,x} (\tilde{r}_{1,v}\cdot DB r_2-r_2\cdot DB\tilde{r}_{1,v})\\
	&-(h_1v_2-h_2v_1)\theta_1^\p \left(\frac{w_1}{v_1}\right)_x (\tilde{r}_{1,\si}\cdot DB r_2-r_2\cdot DB\tilde{r}_{1,\si})\\
	&+(h_1v_2-h_2v_1) ((\tilde{r}_1\otimes u_x)\cdot D^2B r_2-(u_x\otimes r_2)\cdot D^2B\tilde{r}_1).
\end{align*}
Similarly, we have using \eqref{ut1}
\begin{align*}
	h_{t}&=h_{1,t}\tilde{r}_1+v_{1,t}h_1\tilde{r}_{1,v}+h_{2,t}r_2+h_1u_t\bullet \tilde{r}_1+h_1\si_{1,t}\tilde{r}_{1,\si}\\
	&=h_{1,t}\tilde{r}_1+v_{1,t}h_1\tilde{r}_{1,v}+h_{2,t}r_2-h_1\theta_1^\p\left(\frac{w_1}{v_1}\right)_t\tilde{r}_{1,\si}\\
	&+w_1 h_1[\tilde{r}_1\bullet \tilde{r}_1+v_1\tilde{r}_{1,v}\bullet \tilde{r}_1]\\
	&+(\al_2v_{2,x}-(\la_2-r_2\bullet \al_2 v_2)v_2)h_1r_2\bullet \tilde{r}_1\\
	&+\frac{1}{\al_1}(w_1+\si_1v_1)(s_1+v_1s_{1,v})h_1r_2\bullet \tilde{r}_1\\
	&+h_1v^2_{1}\si_1\tilde{r}_{1,v}\bullet \tilde{r}_1+v_1h_1\si_{1,x}s_{1,\si}r_2\bullet \tilde{r}_1+\al_1\si_{1,x}v_1h_1\tilde{r}_{1,\si}\bullet \tilde{r}_1\\
	&+h_1v_1v_2r_2\bullet s_1r_2\bullet \tilde{r}_1+\al_1v_{1}h_1v_2r_2\bullet \tilde{r}_1\bullet \tilde{r}_1\\
	&+h_1v_1v_2(\tilde{r}_1\bullet \al_2)(r_2\bullet \tilde{r}_1).
\end{align*}
We obtain
$$
\begin{aligned}
	&(h_{1,t}+(\tilde{\la}_1h_1)_x-(\al_1h_{1,x})_x)\tilde{r}_1+(h_{2,t}+(\tilde{\la}_2h_2)_x-(\al_2h_{2,x})_x)r_2\\
	&+(v_{1,t}+(\tilde{\la}_1v_1)_x-(\al_1v_{1,x})_x)h_1\tilde{r}_{1,v}\\
	&=\frac{1}{\al_1}((\al_1 h_{1,x}-(\la_1-v_1\al_1^\p-\si_1)h_1)s_1)_xr_2\\
	&+\frac{1}{\al_1}((\al_1 v_{1,x}-(\la_1-v_1\al_1^\p-\si_1)v_1)h_1s_{1,v})_xr_2
	\end{aligned}
	$$
	\begin{equation}
	\begin{aligned}
	&+h_1\theta_1^\p\left(\frac{w_1}{v_1}\right)_t\tilde{r}_{1,\si}-\al_1 h_1\left[\theta_1^\p\left(\frac{w_1}{v_1}\right)_{xx}-\theta_1^{\p\p}\left(\frac{w_1}{v_1}\right)_x^2\right]\tilde{r}_{1,\si}+\widetilde{\mathcal{R}},
\end{aligned}
\label{eqprincip}
\end{equation}
where $\widetilde{\mathcal{R}}$ is defined as follows
\begin{align*}
	\widetilde{\mathcal{R}}=&(\al_1h_{1,x}-(\la_1-v_1 \al_1^\p)h_1)v_1\tilde{r}_1\bullet \tilde{r}_1-(\al_1v_{1,x}-(\la_1 - v_1\al_1^\p)v_1)h_1\tilde{r}_1\bullet \tilde{r}_1\\
	&+(\al_1h_{1,x}-(\la_1-v_1 \al_1^\p)h_1)v_2r_2\bullet \tilde{r}_1+[(\al_1h_{1,x}-(\la_1-v_1 \al_1^\p)h_1)v_{1,x}+(\si_1v_1)_x h_1]\tilde{r}_{1,v}\\
	&-\theta_1^\p\left(\frac{w_1}{v_1}\right)_x(\al_1h_{1,x}-(\la_1-v_1 \al_1^\p)h_1)\tilde{r}_{1,\si}\\
	&+(w_1+\si_1v_1)h_{1,x}\tilde{r}_{1,v}+(w_1+\si_1v_1)h_1v_1\tilde{r}_1\bullet \tilde{r}_{1,v}\\
	&+(w_1+\si_1v_1)h_1v_2r_2\bullet \tilde{r}_{1,v}+(w_1+\si_1v_1)h_1v_{1,x}\tilde{r}_{1,vv}-(w_1+\si_1v_1)h_1\theta_1^\p\left(\frac{w_1}{v_1}\right)_x\tilde{r}_{1,v\si}\\
	&{\color{blue}-}\frac{\al_1^\p v_1}{\al_1^2}(\al_1 h_{1,x}-(\la_1-v_1\al_1^\p-\si_1)h_1)s_1r_2\\
	&+\al^\p_1v_1v_2h_1r_2\bullet \tilde{r}_1+\al_1(v_{2,x}h_1+v_2h_{1,x})r_2\bullet \tilde{r}_1\\
	&+\al_1v_2h_1(r_2\otimes u_x):D^2_u \tilde{r}_1+\al_1v_2h_1v_{1,x}r_2\bullet \tilde{r}_{1,v}-\al_1\theta_1^\p\left(\frac{w_1}{v_1}\right)_x v_2h_1r_2\bullet \tilde{r}_{1,\si}\\
	&-\frac{\al_1^\p v_1}{\al^2_1}(w_1+\si_1v_1)h_1s_{1,v}r_2-v_1\al^\p_1 h_1\theta_1^\p\left(\frac{w_1}{v_1}\right)_x\tilde{r}_{1,\si}-\al_1 h_{1,x}\theta_1^\p\left(\frac{w_1}{v_1}\right)_x \tilde{r}_{1,\si}\\
	&-\al_1  h_1\theta_1^\p\left(\frac{w_1}{v_1}\right)_x v_1\tilde{r}_1\bullet \tilde{r}_{1,\si}-\al_1  h_1\theta_1^\p\left(\frac{w_1}{v_1}\right)_xv_2r_2\bullet \tilde{r}_{1,\si}-\al_1  h_1\theta_1^\p\left(\frac{w_1}{v_1}\right)_xv_{1,x}\tilde{r}_{1,v\si}\\
	&+\al_1 h_1(\theta_1^\p)^2 \left(\frac{w_1}{v_1}\right)^2_x\tilde{r}_{1,\si\si}+(v_{2,x}h_1+v_2h_{1,x})r_2\bullet s_1r_2+v_2h_1(u_x\otimes r_2):D^2 s_1r_2\\
	&+v_2h_1v_{1,x}r_2\bullet s_{1,v}r_2-v_2h_1\theta_1^\p\left(\frac{w_1}{v_1}\right)_x r_2\bullet s_{1,\si}r_2+ h_1\left[\theta_1^{\p\p} \left(\frac{w_1}{v_1}\right)^2_x-\theta_1^\p \left(\frac{w_1}{v_1}\right)_{xx}\right]s_{1,\si}r_2\\
	&-  h_{1,x}\theta_1^\p\left(\frac{w_1}{v_1}\right)_x s_{1,\si}r_2- h_1\theta_1^\p\left(\frac{w_1}{v_1}\right)_x v_1\tilde{r}_1\bullet s_{1,\si}r_2- h_1\theta_1^\p\left(\frac{w_1}{v_1}\right)_x v_2r_2\bullet s_{1,\si}r_2\\
	&- h_1 (\theta_1^\p)^2 \left(\frac{w_1}{v_1}\right)_x^2s_{1,\si\si}r_2-  h_1\theta_1^\p\left(\frac{w_1}{v_1}\right)_x v_{1,x} s_{1,v\si}r_2\\
	&+(\tilde{r}_1\otimes u_x):D^2\al_2 v_1h_2r_2+(\tilde{r}_1\bullet \tilde{r}_1)\bullet \al_2 v^2_1h_2r_2+(r_2\bullet \tilde{r}_1)\bullet \al_2 v_1v_2h_2r_2\\
	&+\tilde{r}_{1,v}\bullet \al_2 v_1v_{1,x}h_2r_2-\tilde{r}_{1,\si}\bullet \al_2\theta_1^\p\left(\frac{w_1}{v_1}\right)_x v_1h_2r_2+\tilde{r}_1\bullet \al_2 (v_{1,x}h_2+v_1h_{2,x})r_2\\
	&+(h_1v_2-h_2v_1)_x (\tilde{r}_1\cdot DB r_2-r_2\cdot DB\tilde{r}_1)\\
	&+(h_1v_2-h_2v_1)v_1 ((\tilde{r}_1\bullet \tilde{r}_1)\cdot DB r_2-r_2\cdot DB(\tilde{r}_1\bullet \tilde{r}_1))\\
	&+(h_1v_2-h_2v_1) {\color{blue}v_2}((r_2\bullet \tilde{r}_1)\cdot DB r_2-r_2\cdot DB(r_2\bullet \tilde{r}_1))\\
	&+(h_1v_2-h_2v_1)v_{1,x} (\tilde{r}_{1,v}\cdot DB r_2-r_2\cdot DB\tilde{r}_{1,v})\\
	&-(h_1v_2-h_2v_1)\theta_1^\p\left(\frac{w_1}{v_1}\right)_x (\tilde{r}_{1,\si}\cdot DB r_2-r_2\cdot DB\tilde{r}_{1,\si})\\
	&+(h_1v_2-h_2v_1) ((\tilde{r}_1\otimes u_x)\cdot D^2B r_2-(u_x\otimes r_2)\cdot D^2B\tilde{r}_1)\\
	&-(w_1+\si_1v_1)h_1v_1\tilde{r}_{1,v}\bullet \tilde{r}_1-(\al_2v_{2,x}-(\la_2-r_2\bullet \al_2 v_2)v_2)h_1r_2\bullet \tilde{r}_1\\
	&-\frac{1}{\al_1}(w_1+\si_1v_1)(s_1+v_1s_{1,v})h_1r_2\bullet \tilde{r}_1\\
	&-v_1 {\color{blue}h_1}\theta_1^\p\left(\frac{w_1}{v_1}\right)_xs_{1,\si}r_2\bullet \tilde{r}_1-\al_1\theta_1^\p\left(\frac{w_1}{v_1}\right)_xv_1h_1\tilde{r}_{1,\si}\bullet \tilde{r}_1\\
	&-h_1v_1v_2r_2\bullet s_1r_2\bullet \tilde{r}_1-\al_1v_{1}h_1v_2r_2\bullet \tilde{r}_1\bullet \tilde{r}_1-h_1v_1v_2(\tilde{r}_1\bullet \al_2)(r_2\bullet \tilde{r}_1).
\end{align*}
As in the previous section taking the scalar product of \eqref{eqprincip} with $l_1$
 we get
\begin{equation}
	h_{1,t}+(\tilde{\la}_1h_1)_x-(\al_1h_{1,x})_x=0.
	\label{7.7}
\end{equation}
Now taking the scalar product with $l_2$ we obtain
\begin{equation}
	h_{2,t}+(\tilde{\la}_2h_2)_x-(\al_2h_{2,x})_x=\tilde{\varphi}_2,
	\label{7.8}
\end{equation}
where $\tilde{\varphi}_2$ is defined as follows
\begin{equation}
\begin{aligned}
	\tilde{\varphi}_2&=\frac{1}{\al_1}((\al_1 h_{1,x}-(\la_1-v_1\al_1^\p-\si_1)h_1)s_1)_x\\
	&+\frac{1}{\al_1}\big((w_1+\si_1v_1)h_1s_{1,v}\big)_x+\langle l_2, \widetilde{\mathcal{R}}\rangle\\
	&+\langle l_2,\left(h_1\theta_1^\p\left(\frac{w_1}{v_1}\right)_t\tilde{r}_{1,\si}-\al_1 h_1\left[\theta_1^\p\left(\frac{w_1}{v_1}\right)_{xx}-\theta_1^{\p\p}\left(\frac{w_1}{v_1}\right)_x^2\right]\tilde{r}_{1,\si}\right)\rangle.
\end{aligned}
\label{formva2}
\end{equation}
We are now going to estimate each terms on the right hand side of \eqref{formva2}. We start by defining the unknown $\hat{h}_1=\al_1 h_{1,x}-\tilde{\la}_1h_1$ which is the equivalent of $w_1$ for $v_1$. By using the same techniques as in the previous section, we will show energy estimate which will allow to control the quantity $h^2_{1,x}\chi_{\{|\frac{\hat{h}_1}{h_1}|\geq\delta_1\}}$ in $L^1([\hat{t},T^*]\times\R)$. Let us start with the term $((\al_1 h_{1,x}-(\la_1-v_1\al_1^\p-\si_1)h_1)s_1)_x$, we note that %using the fact that
%$\alpha_1 v_{1,xx}=w_{1,x}-\alpha'_1 v_{1}v_{1,x}+\lambda'_1v_1^2+\lambda_1v_{1,x}-2v_1v_{1,x}\alpha'_1-\alpha_1''v_1^3$
\begin{align*}
	&((\al_1 h_{1,x}-(\la_1-v_1\al_1^\p-\si_1)h_1)s_1)_x\\
	&=[\al_1h_{1,xx}-\la_1 h_{1,x}-\la_1^\p v_1h_1+ \al_1^\p v_{1,x}h_1+\al_1^{\p\p}v_1^2h_1+2\al_1^\p v_1h_{1,x}+\si_{1,x}h_1+\si_1h_{1,x}]s_1\\
	&+(\al_1 h_{1,x}v_{1,x}-\la_1v_{1,x}h_1+v_{1,x}v_1\al_1^\p h_1+\si_1v_{1,x}h_1)s_{1,v}\\
	& +(\al_1 h_{1,x}-\tilde{\la}_1 h_1+\si_1 h_1)(v_1\tilde{r_1}\bullet s_1+v_2r_2\bullet s_1)-(\al_1 h_{1,x}-\tilde{\la}_1 h_1+\si_1 h_1)\theta'_1\left(\frac{w_1}{v_1}\right)_x s_{1,\sigma} \\
	&=2\al_1^\p (v_1h_{1,x}-v_{1,x}h_1)s_1+[h_{1,xx}v_1-h_1v_{1,xx}]\frac{\al_1 s_1}{v_1}\\
	&+\frac{ s_1}{v_1}[w_{1,x}h_1-\al_1^\p v_1v_{1,x}h_1+\la_1^\p v_1^2h_1+\la_1v_{1,x}h_1-\al_1^{\p\p}v_1^3h_1+\al_1^\p v_1v_{1,x}h_1]\\
	&+[-\la_1 h_{1,x}-\la_1^\p v_1h_1 +\al_1^{\p\p}v_1^2h_1+\si_1h_{1,x}]s_1-\theta_1^\p\left(\frac{w_1}{v_1}\right)_xh_1s_{1}\\
	&+( h_{1,x}w_1+\la_1 h_{1,x}v_1-\al_1^\p v_1^2 h_{1,x}-\la_1v_{1,x}h_1+v_{1,x}v_1\al_1^\p h_1+\si_1v_{1,x}h_1)s_{1,v}\\
	& +\big(\alpha_1(h_{1,x}v_1-v_{1,x}h_1)+h_1(w_1+\sigma_1 v_1)\big)\tilde{r}_1\bullet s_1+ (\al_1 h_{1,x}-\tilde{\la}_1 h_1+\si_1 h_1)v_2r_2\bullet s_1\\
	& -(\al_1 h_{1,x}-\tilde{\la}_1 h_1+\si_1 h_1)\theta'_1\left(\frac{w_1}{v_1}\right)_x s_{1,\sigma}.
\end{align*}
We can simplify as follows
\begin{align*}
&((\al_1 h_{1,x}-(\la_1-v_1\al_1^\p-\si_1)h_1)s_1)_x  \nonumber\\
	&=2\al_1^\p (v_1h_{1,x}-v_{1,x}h_1)s_1 +[h_{1,xx}v_1-h_1v_{1,xx}]\frac{\al_1 s_1}{v_1}  \nonumber\\
	&+\frac{s_1}{v_1}[(w_{1,x}h_1-w_1h_{1,x})+\lambda_1(v_{1,x}h_1-h_{1,x}v_1)+(w_1+\sigma_1 v_1)h_{1,x}]-\theta_1^\p\left(\frac{w_1}{v_1}\right)_xh_1s_{1}\\
	&+( h_{1,x}w_1+\la_1 h_{1,x}v_1-\al_1^\p v_1^2 h_{1,x}-\la_1v_{1,x}h_1+v_{1,x}v_1\al_1^\p h_1+\si_1v_{1,x}h_1)s_{1,v}\\
	& +\big(\alpha_1(h_{1,x}v_1-v_{1,x}h_1)+h_1(w_1+\sigma_1 v_1)\big)\tilde{r}_1\bullet s_1+ (\al_1 h_{1,x}-\tilde{\la}_1 h_1+\si_1 h_1)v_2r_2\bullet s_1\\
	& -(\al_1 h_{1,x}-\tilde{\la}_1 h_1+\si_1 h_1)\theta'_1\left(\frac{w_1}{v_1}\right)_x s_{1,\sigma} \\[2mm]
	&=2\al_1^\p (v_1h_{1,x}-v_{1,x}h_1)s_1 +[h_{1,xx}v_1-h_1v_{1,xx}]\frac{\al_1 s_1}{v_1}  \nonumber\\
	&+\frac{s_1}{v_1}[(w_{1,x}h_1-w_1h_{1,x})+\lambda_1(v_{1,x}h_1-h_{1,x}v_1)+(w_1+\sigma_1 v_1)h_{1,x}]\\
	&+\frac{s_{1}}{v_1}\big(-\theta'_1\frac{w_1}{v_1}(h_{1,x}v_1-v_{1,x}h_1)+\theta'_1(h_{1,x}w_1-w_{1,x}h_1)\big)\\
	&+\big( h_{1,x}(w_1+\sigma_1v_1)+(\tilde{\lambda}_1-\sigma_1)(h_{1,x}v_1-v_{1,x}h_1)\big)s_{1,v}\\
	& +\big(\alpha_1(h_{1,x}v_1-v_{1,x}h_1)+h_1(w_1+\sigma_1 v_1)\big)\tilde{r}_1\bullet s_1+ (\al_1 h_{1,x}-\tilde{\la}_1 h_1+\si_1 h_1)v_2r_2\bullet s_1\\
	& -(\al_1 h_{1,x}-\tilde{\la}_1 h_1+\si_1 h_1)\theta'_1\left(\frac{w_1}{v_1}\right)_x s_{1,\sigma}.
	%&((\al_1 h_{1,x}-(\la_1-v_1\al_1^\p-\si_1)h_1)s_1)_x  \nonumber\\
	%&=2\al_1^\p (v_1h_{1,x}-v_{1,x}h_1)s_1 +[h_{1,xx}v_1-h_1v_{1,xx}]\frac{\al_1 s_1}{v_1}  \nonumber\\
	%&+\frac{ s_1}{v_1}[w_{1,x}h_1-w_1h_{1,x}+(w_1+\si_1v_1)h_{1,x} +\la_1(h_1v_{1,x}-h_{1,x}v_1)]  \nonumber\\
	%&-\theta_1^\p(h_1w_{1,x}-h_{1,x}w_1)\frac{s_1}{v_1}+\theta_1^\p(h_1v_{1,x}-h_{1,x}v_1)\frac{s_1}{v_1}+\si_1(v_{1,x}h_1-v_1h_{1,x})s_{1,v}  \nonumber\\
	%&+[(w_1+\si_1v_1)h_{1,x}+\la_1(h_{1,x}v_1-h_1v_{1,x})+\al_1^\p v_1(h_1v_{1,x}-h_{1,x}v_1).\label{estimate-h-1-1}
\end{align*}
Since $\frac{s_1}{v_1}=O(1)$,  $\frac{s_{1,\sigma}}{v_1}=O(1)$ and using the fact that 
\begin{equation}
\sigma_{1,x}h_1v_1=-\theta'_1\frac{w_1}{v_1}(h_{1,x}v_1-v_{1,x}h_1)+\theta'_1(h_{1,x}w_1-w_{1,x}h_1),
\label{astucesig}
\end{equation}
 it implies that
\begin{equation}
\begin{aligned}
&((\al_1 h_{1,x}-(\la_1-v_1\al_1^\p-\si_1)h_1)s_1)_x  =O(1)|v_1h_{1,x}-v_{1,x}h_1|+O(1)|h_{1,xx}v_1-h_1v_{1,xx}|\\
&+O(1)|w_{1,x}h_1-w_1h_{1,x}|+O(1)|(w_1+\sigma_1 v_1)h_{1,x}|+O(1)|h_1v_2|+O(1)|h_{1,x}v_2|\\
&+O(1)h_{1,x}^2\chi_{\{|\frac{\hat{h}_1}{h_1}|\geq\delta_1\}}+O(1)v_1^2(\theta'_1)^2\left(\frac{w_1}{v_1}\right)_x^2\chi_{\{|\frac{\hat{h}_1}{h_1}|\geq\delta_1\}}
\end{aligned}
\label{superimp1}
\end{equation}
\begin{remark}
We can observe that the term $|h_{1,xx}v_1-h_1v_{1,xx}|$ is new compared to the interaction terms in \cite{BB-triangular}.
\end{remark}
Using the fact that
$$(w_{1,x}+\si_1v_{1,x})h_1=(h_1w_{1,x}-h_{1,x}w_1)+h_{1,x}(w_1+\sigma_1 v_1)+\sigma_1(v_{1,x}h_1-h_{1,x}v_1),$$
we estimate the term $\big((w_1+\si_1v_1)h_1s_{1,v}\big)_x$ as follows 
\begin{align}
	&((w_1+\si_1v_1)h_1s_{1,v})_x \nonumber\\
	 &=(w_{1,x}+\si_1v_{1,x})h_1s_{1,v} 
	-\theta_1^\p\left(\frac{w_1}{v_1}\right)_xv_1h_1s_{1,v}+(w_1+\si_1v_1)v_1h_1\tilde{r}_1\bullet s_{1,v}\chi_{\left\{\de_1<\abs{\frac{w_1}{v_1}}\right\}}  \nonumber\\
	&+(w_1+\si_1v_1)v_2h_1r_2\bullet s_{1,v}\chi_{\left\{\de_1<\abs{\frac{w_1}{v_1}}\right\}}+(w_1+\si_1v_1) (h_{1,x}s_{1,v}+h_1  v_{1,x}s_{1,vv})\chi_{\left\{\de_1<\abs{\frac{w_1}{v_1}}\right\}}  \nonumber\\
	&-(w_1+\si_1v_1)\theta_1\left(\frac{w_1}{v_1}\right)_xh_1  s_{1,v\si}\chi_{\left\{\de_1<\abs{\frac{w_1}{v_1}}\right\}}  \nonumber\\
	&=\mathcal{O}(1)\Big(\abs{w_{1,x}h_1-w_1h_{1,x}}+\abs{v_{1,x}h_1-v_1h_{1,x}}+\abs{h_1v_2}\Big)  \nonumber\\
	&+\mathcal{O}(1)\abs{w_1+\si_1v_1}(\abs{h_{1,x}}+\abs{v_{1,x}}).  \label{estimate-h-1-2}
\end{align}
Similarly, we can show using \eqref{astucesig} and the fact that $\tilde{r}_{1,\sigma}=0(1)v_1$ that
\begin{align}
	&h_1\theta_1^\p\left(\frac{w_1}{v_1}\right)_t\tilde{r}_{1,\si}-\al_1 h_1\left[\theta_1^\p\left(\frac{w_1}{v_1}\right)_{xx}-\theta_1^{\p\p}\left(\frac{w_1}{v_1}\right)_x^2\right]\tilde{r}_{1,\si}    \nonumber\\
	&=\theta_1^\p h_1\left(w_{1,t}-\al_1 w_{1,xx}-\frac{w_1}{v_1}(v_{1,t}-\al_1 v_{1,xx})+2\alpha_1 v_{1,x}\left(\frac{w_1}{v_1}\right)_x\right)[\tilde{r}_{1,\si}/v_1]\nonumber\\
	&+\alpha_1 h_1\theta_1^{\p\p}\left(\frac{w_1}{v_1}\right)_x^2\tilde{r}_{1,\si}   \nonumber\\
	&=\theta_1^\p \left(\al_1^\p(w_{1,x}v_1-w_1v_{1,x})h_1-\tilde{\la}_1 h_1 v_1 \left(\frac{w_1}{v_1}\right)_x \right)[\tilde{r}_{1,\si}/v_1]   \nonumber\\
	&-2\alpha_1\frac{\tilde{r}_{1,\sigma}}{v_1}\frac{v_{1,x}}{v_1} v_1h_1\sigma_{1,x}\chi_{\{|\frac{w_1}{v_1}|\leq 3\delta_1\}}
	+\alpha_1\theta_1^{\p\p}\left(\frac{w_1}{v_1}\right)_x^2h_1\tilde{r}_{1,\si}   \nonumber\\
	&=\mathcal{O}(1)\left(\abs{v_1w_{1,x}-v_{1,x}w_1}+\abs{h_1w_{1,x}-h_{1,x}w_1}+\abs{h_1v_{1,x}-h_{1,x}v_1}\right)\nonumber\\
	&+\mathcal{O}(1)+\abs{\theta_1^{\p\p}}\abs{h_1v_1}\left(\frac{w_1}{v_1}\right)_x^2. \label{estimate-h-1-3}
\end{align}
\begin{lemma}\label{lemma:estimate-R}
	The remainder terms in $\widetilde{\mathcal{R}}$ satisfies
	\begin{align*}
		\widetilde{\mathcal{R}}&=\mathcal{O}(1)\Big(\abs{h_{1,x}v_1-h_1v_{1,x}}+\abs{h_{1,x}w_1-h_1w_{1,x}}+\abs{w_{1,x}v_1-w_{1}v_{1,x}}\Big)\\
		&+\mathcal{O}(1)\Big(\abs{h_1v_2}+\abs{h_1v_{2,x}}+\abs{h_2v_{1}}+\abs{h_2v_{1,x}}+\abs{h_{2,x}v_1}+\abs{h_{1,x}v_2}\Big)\\
		&+\mathcal{O}(1)(\abs{h_{1,x}}+\abs{h_1})\abs{v_1}\left(\frac{w_1}{v_1}\right)_x^2\chi_{\{|\frac{w_1}{v_1}|\leq 3\delta_1\}}+\mathcal{O}(1)\abs{h_{1,xx}v_1-v_{1,xx}h_1}\\
		&+\mathcal{O}(1)\abs{w_1+\si_1v_1}(\abs{h_{1,x}}+\abs{v_1}+|v_{1,x}|)+\mathcal{O}(1)\abs{h_{1,xx}w_1-w_{1,xx}h_1}.
	\end{align*}
\end{lemma}
\begin{proof}[Proof of Lemma \ref{lemma:estimate-R}:]
	We observe that
	\begin{align*}
		&(\al_1h_{1,x}-(\la_1-v_1 \al_1^\p)h_1)v_1\tilde{r}_1\bullet \tilde{r}_1-(\al_1v_{1,x}-(\la_1 - v_1\al_1^\p)v_1)h_1\tilde{r}_1\bullet \tilde{r}_1\\
		&=\alpha_1(h_{1,x}v_1-h_1v_{1,x})\tilde{r}_1\bullet \tilde{r}_1=O(1)|h_{1,x}v_1-h_1v_{1,x}|.
	\end{align*}
Then we note that using \eqref{astucesig}
\begin{align*}
	&(\al_1h_{1,x}-(\la_1-v_1 \al_1^\p)h_1)v_2r_2\bullet \tilde{r}_1+[(\al_1h_{1,x}-(\la_1-v_1 \al_1^\p)h_1)v_{1,x}+(\si_1v_1)_xh_1]\tilde{r}_{1,v}\\
	&=(\al_1h_{1,x}-(\la_1-v_1 \al_1^\p)h_1)v_2r_2\bullet \tilde{r}_1+[h_{1,x}(w_1+\sigma_1 v_1)+\sigma_1(v_{1,x}h_1-h_{1,x}v_1)+(\sigma_1)_x v_1h_1]\tilde{r}_{1,v}\\
	&=\mathcal{O}(1)\big(\abs{w_1+\si_1v_1}\abs{h_{1,x}}+\abs{h_{1,x}v_1-h_1v_{1,x}}+\abs{h_{1,x}w_1-h_1w_{1,x}}+\abs{v_2}(\abs{h_{1,x}}+\abs{h_1})\big).
\end{align*}
Since $\abs{\theta_1^\p v_{1,x}}=\mathcal{O}(1)\abs{v_1}$,  $\abs{\theta_1^\p \frac{w_1}{v_1}}=\mathcal{O}(1)$, $\tilde{r}_{1,\sigma}=O(1)v_1$ and $b\leq 1+b^2$ we get
\begin{align*}
	&-\theta_1^\p\left(\frac{w_1}{v_1}\right)_x(\al_1h_{1,x}-(\la_1-v_1 \al_1^\p)h_1)\tilde{r}_{1,\si}\\
	&=-\theta_1^\p v_1\left(\frac{w_1}{v_1}\right)_x(\al_1h_{1,x}-(\la_1-v_1 \al_1^\p)h_1)[\tilde{r}_{1,\si}/v_1]\\
	&=\mathcal{O}(1)\abs{\theta_1^\p}\left[\abs{h_{1,x}v_1}\abs{\left(\frac{w_1}{v_1}\right)_x}+\abs{h_1\left(w_{1,x}-\frac{w_1}{v_1}v_{1,x}\right)}\right]\\
	&=\mathcal{O}(1)\abs{\theta_1^\p}\left[\abs{h_{1,x}v_1-h_{1}v_{1,x}}\abs{\left(\frac{w_1}{v_1}\right)_x}+\abs{-\frac{w_1}{v_1}(h_1v_{1,x}-v_1h_{1,x})+(h_1w_{1,x}-w_1h_{1,x})}\right]\\
	&=\mathcal{O}(1)\abs{\theta_1^\p}\left[\abs{h_{1,x}v_1-h_{1}v_{1,x}}+(\abs{h_{1,x}}+\abs{h_1})\abs{v_1}\left(\frac{w_1}{v_1}\right)_x^2+\abs{h_{1,x}w_1-h_1w_{1,x}}\right].
\end{align*}
We can check
\begin{align*}
	&(w_1+\si_1v_1)h_{1,x}\tilde{r}_{1,v}+(w_1+\si_1v_1)h_1v_1\tilde{r}_1\bullet \tilde{r}_{1,v}\\
	&=\mathcal{O}(1)\abs{w_1+\si_1v_1}(\abs{h_{1,x}}+\abs{v_1}),\\
	&(w_1+\si_1v_1)h_1v_2r_2\bullet \tilde{r}_{1,v}+(w_1+\si_1v_1)h_1v_{1,x}\tilde{r}_{1,vv}\\
	&=\mathcal{O}(1)\big(\abs{w_1+\si_1v_1}\abs{v_{1,x}}+\abs{v_{1,x}v_2}+\abs{v_1v_2}\big).
	\end{align*}
	Using \eqref{astucesig} and the fact that $\theta'_1 w_1=0(1)v_1$ we have
	\begin{align*}
	&-(w_1+\si_1v_1)h_1\theta_1^\p\left(\frac{w_1}{v_1}\right)_x\tilde{r}_{1,v\si}\\
	&=\mathcal{O}(1)(\abs{h_{1,x}w_1-h_1w_{1,x}}+\abs{h_{1,x}v_1-h_1v_{1,x}}).
	\end{align*}
	We note now that
	\begin{align*}
	&\frac{\al_1^\p v_1}{\al_1^2}(\al_1 h_{1,x}-(\la_1-v_1\al_1^\p-\si_1)h_1)s_1r_2\\
	&=\mathcal{O}(1)\big(\abs{w_1+\si_1v_1}\abs{v_1}+\abs{h_{1,x}v_1-h_1v_{1,x}}\big),\\
	&\al^\p_1v_1v_2h_1r_2\bullet \tilde{r}_1+\al_1(v_{2,x}h_1+v_2h_{1,x})r_2\bullet \tilde{r}_1+\al_1v_2h_1(r_2\otimes u_x):D^2_u \tilde{r}_1\\
	&+\al_1v_2h_1v_{1,x}r_2\bullet \tilde{r}_{1,v}-\al_1\theta_1^\p\left(\frac{w_1}{v_1}\right)_x v_2h_1r_2\bullet \tilde{r}_{1,\si}\\
	&=\mathcal{O}(1)(\abs{h_{1}v_2}+\abs{h_{1,x}v_2}+\abs{h_1v_{2,x}})\\
	&-\frac{\al_1^\p v_1}{\al^2_1}(w_1+\si_1v_1)h_1s_{1,v}r_2-v_1\al^\p_1 h_1\theta_1^\p\left(\frac{w_1}{v_1}\right)_x\tilde{r}_{1,\si}\\
	&=\mathcal{O}(1)\big(\abs{w_1+\si_1v_1} \abs{v_1}+\abs{h_{1,x}v_1-h_1v_{1,x}}+\abs{h_{1,x}w_1-h_1w_{1,x}}\big).
	\end{align*}
	Using \eqref{astucesig} and the fact that $\theta'_1 v_{1,x}=0(1)v_1$ we have
	\begin{align*}
	&-\al_1 h_{1,x}\theta_1^\p\left(\frac{w_1}{v_1}\right)_x \tilde{r}_{1,\si}=-\al_1 (h_{1,x}v_1-v_{1,x}h_1)\theta_1^\p\left(\frac{w_1}{v_1}\right)_x \frac{\tilde{r}_{1,\si}}{v_1}+\alpha_1 v_{1,x}h_1\sigma_{1,x}\\
	&=\mathcal{O}(1)\abs{\theta_1^\p}\left[\abs{h_{1,x}v_1-h_{1}v_{1,x}}+(\abs{h_{1,x}}+\abs{h_1})\abs{v_1}\left(\frac{w_1}{v_1}\right)_x^2+\abs{h_{1,x}w_1-h_1w_{1,x}}\right].
	\end{align*}
	Similarly we have
	\begin{align*}
	&-\al_1  h_1\theta_1^\p\left(\frac{w_1}{v_1}\right)_x v_1\tilde{r}_1\bullet \tilde{r}_{1,\si}-\al_1  h_1\theta_1^\p\left(\frac{w_1}{v_1}\right)_xv_2r_2\bullet \tilde{r}_{1,\si}-\al_1  h_1\theta_1^\p\left(\frac{w_1}{v_1}\right)_xv_{1,x}\tilde{r}_{1,v\si}\\
	&=\mathcal{O}(1)\big(\abs{w_{1,x}v_1-w_1v_{1,x}}+\abs{v_2 h_1}+\abs{h_{1,x}v_1-h_1v_{1,x}}+\abs{h_{1,x}w_1-h_1w_{1,x}}\big),\\
	&\al_1 h_1(\theta_1^\p)^2 \left(\frac{w_1}{v_1}\right)^2_x\tilde{r}_{1,\si\si}+(v_{2,x}h_1+v_2h_{1,x})r_2\bullet s_1r_2+v_2h_1(u_x\otimes r_2):D^2 s_1r_2\\
	&=\mathcal{O}(1)\left(\abs{h_1v_1}\left(\frac{w_1}{v_1}\right)_x^2+\abs{h_{1,x}v_2}+\abs{h_1v_{2,x}}+\abs{h_1v_2}\right).
	\end{align*}
Now using \eqref{formule2}, \eqref{astucesig} and the fact that $s_{1,\sigma}=O(1)v1$, 
$\theta'_1 v_{1,x}=O(1)v_1$, we have
%\frac{w_{1,xx}v_1-v_{1,xx}w_1}{v_1^2}-\frac{2v_{1,x}}{v_1}\left(\frac{w_1}{v_1}\right)_x.	
\begin{align*}
	&v_2h_1v_{1,x}r_2\bullet s_{1,v}r_2-v_2h_1\theta_1^\p\left(\frac{w_1}{v_1}\right)_x r_2\bullet s_{1,\si}r_2 + h_1\left[\theta_1^{\p\p} \left(\frac{w_1}{v_1}\right)^2_x-\theta_1^\p \left(\frac{w_1}{v_1}\right)_{xx}\right]s_{1,\si}r_2\\
	&=\mathcal{O}(1)\left(\theta_1^{\p\p}\abs{h_1v_1}\left(\frac{w_1}{v_1}\right)_x^2 +\abs{h_1v_2}\right)+\mathcal{O}(1)\theta'_1 |h_1(w_{1,xx}-\frac{v_{1,xx}w_1}{v_1})| +\mathcal{O}(1) |\frac{v_{1,x}}{v_1} h_1v_1\sigma_{1,x} |,\\
	&=\mathcal{O}(1)\left(\theta_1^{\p\p}\abs{h_1v_1}\left(\frac{w_1}{v_1}\right)_x^2 +\abs{h_1v_2}\right)+\mathcal{O}(1)\theta'_1 |(h_1w_{1,xx}-w_1h_{1,xx})-\frac{w_1}{v_1}(h_1 v_{1,xx}-v_1h_{1,xx})| \\
	&+\mathcal{O}(1) |\frac{v_{1,x}}{v_1} h_1v_1\sigma_{1,x} |,\\
	&=\mathcal{O}(1)\big(\theta_1^{\p\p}\abs{h_1v_1}\left(\frac{w_1}{v_1}\right)_x^2 +\abs{h_1v_2}+|h_1w_{1,xx}-w_1h_{1,xx}|+|h_1 v_{1,xx}-v_1h_{1,xx}|+|h_{1,x}v_1-v_1h_{1,x}|\\
	&+|h_{1,x}w_1-w_{1,x}h_1|\big).
	\end{align*}
	Next we have similarly
	\begin{align*}
	&-h_{1,x}\theta_1^\p\left(\frac{w_1}{v_1}\right)_x s_{1,\si}r_2- h_1\theta_1^\p\left(\frac{w_1}{v_1}\right)_x v_1\tilde{r}_1\bullet s_{1,\si}r_2\\
	&=\mathcal{O}(1)\abs{\theta_1^\p}\left[\abs{h_{1,x}v_1-h_{1}v_{1,x}}+(\abs{h_{1,x}}+\abs{h_{1}})\abs{v_1}\left(\frac{w_1}{v_1}\right)_x^2+\abs{h_{1,x}w_1-h_1w_{1,x}}\right],\\
	&- h_1\theta_1^\p\left(\frac{w_1}{v_1}\right)_x v_2r_2\bullet s_{1,\si}r_2=\mathcal{O}(1)\abs{h_1v_2},\\
	&- h_1 (\theta_1^\p)^2 \left(\frac{w_1}{v_1}\right)_x^2s_{1,\si\si}r_2-  h_1\theta_1^\p\left(\frac{w_1}{v_1}\right)_x v_{1,x} s_{1,v\si}r_2\\
	&=\mathcal{O}(1)\abs{\theta_1^\p}\left[\abs{h_{1,x}v_1-h_{1}v_{1,x}}+\abs{h_{1}}\abs{v_1}\left(\frac{w_1}{v_1}\right)_x^2+\abs{h_{1,x}w_1-h_1w_{1,x}}\right],\\
	&(\tilde{r}_1\otimes u_x):D^2\al_2 v_1h_2r_2+(\tilde{r}_1\bullet \tilde{r}_1)\bullet \al_2 v^2_1h_2r_2+(r_2\bullet \tilde{r}_1)\bullet \al_2 v_1v_2h_2r_2=\mathcal{O}(1)\abs{v_1h_2},\\
	&\tilde{r}_{1,v}\bullet \al_2 v_1v_{1,x}h_2r_2-\tilde{r}_{1,\si}\bullet \al_2\theta_1^\p\left(\frac{w_1}{v_1}\right)_x v_1h_2r_2+\tilde{r}_1\bullet \al_2 (v_{1,x}h_2+v_1h_{2,x})r_2\\
	&=\mathcal{O}(1)\big(\abs{v_1h_2}+\abs{v_{1,x}h_2}+\abs{v_1h_{2,x}}\big),\\
	&(h_1v_2-h_2v_1)_x (\tilde{r}_1\cdot DB r_2-r_2\cdot DB\tilde{r}_1)+(h_1v_2-h_2v_1)v_1 ((\tilde{r}_1\bullet \tilde{r}_1)\cdot DB r_2-r_2\cdot DB(\tilde{r}_1\bullet \tilde{r}_1))\\
	&+(h_1v_2-h_2v_1) v_2((r_2\bullet \tilde{r}_1)\cdot DB r_2-r_2\cdot DB(r_2\bullet \tilde{r}_1))+(h_1v_2-h_2v_1)v_{1,x} (\tilde{r}_{1,v}\cdot DB r_2-r_2\cdot DB\tilde{r}_{1,v})\\
	&-(h_1v_2-h_2v_1)\theta_1^\p\left(\frac{w_1}{v_1}\right)_x (\tilde{r}_{1,\si}\cdot DB r_2-r_2\cdot DB\tilde{r}_{1,\si})\\
	&+(h_1v_2-h_2v_1) ((\tilde{r}_1\otimes u_x)\cdot D^2B r_2-(u_x\otimes r_2)\cdot D^2B\tilde{r}_1)\\
	&=\mathcal{O}(1)\big(\abs{h_1v_2}+\abs{h_2v_1}+\abs{h_{1,x}v_2}+\abs{h_{2,x}v_1}+\abs{h_1v_{2,x}}+\abs{h_2v_{1,x}}\big).
	\end{align*}
	Again we have since $s_1=O(1)v_1$
	\begin{align*}
	&-(w_1+\si_1v_1)h_1v_1\tilde{r}_{1,v}\bullet \tilde{r}_1-(\al_2v_{2,x}-(\la_2-r_2\bullet \al_2 v_2)v_2)h_1r_2\bullet \tilde{r}_1\\
	&-\frac{1}{\al_1}(w_1+\si_1v_1)(s_1+v_1s_{1,v})h_1r_2\bullet \tilde{r}_1\\
	&=\mathcal{O}(1)\big(\abs{w_1+\si_1v_1}\abs{v_1}+\abs{h_1v_{2,x}}+\abs{h_1v_2}\big),\\
	&-h_1 v_1\theta_1^\p\left(\frac{w_1}{v_1}\right)_xs_{1,\si}r_2\bullet \tilde{r}_1-\al_1\theta_1^\p\left(\frac{w_1}{v_1}\right)_xv_1h_1\tilde{r}_{1,\si}\bullet \tilde{r}_1\\
	&=\mathcal{O}(1)\big(\abs{h_{1,x}w_1-h_1w_{1,x}}+\abs{h_{1,x}v_1-h_1v_{1,x}}\big),\\
	&-h_1v_1v_2r_2\bullet s_1r_2\bullet \tilde{r}_1-\al_1v_{1}h_1v_2r_2\bullet \tilde{r}_1\bullet \tilde{r}_1-h_1v_1v_2(\tilde{r}_1\bullet \al_2)(r_2\bullet \tilde{r}_1)\\
	&=\mathcal{O}(1)\abs{h_1v_2}.
\end{align*}
This completes  the proof of Lemma \ref{lemma:estimate-R}.
\end{proof}
From \eqref{formva2}, \eqref{superimp1}, \eqref{estimate-h-1-2}, \eqref{estimate-h-1-3} and Lemma \ref{lemma:estimate-R} we obtain that
\begin{align}
	\tilde{\varphi}_2&=\mathcal{O}(1)\big(\abs{h_{1,xx}v_1-h_1v_{1,xx}} +\abs{h_{1,xx}w_1-h_1w_{1,xx}}+\abs{w_1+\si_1v_1}(\abs{h_1}+\abs{h_{1,x}}+\abs{v_1}+\abs{v_{1,x}}) \nonumber\\
	&+\abs{w_{1,x}v_1-w_{1}v_{1,x}}\big)+\mathcal{O}(1)\Big(\chi_{\{|\frac{w_1}{v_1}|\leq 3\delta_1\}}(\abs{h_{1,x}}+\abs{h_1})\abs{v_1}\left(\frac{w_1}{v_1}\right)_x^2+\abs{h_{1,x}v_1-h_1v_{1,x}}\nonumber\\
	&+\abs{h_{1,x}w_1-h_1w_{1,x}}\Big)+\mathcal{O}(1)\Big(\abs{h_1v_2}+\abs{h_1v_{2,x}}+\abs{h_2v_{1}}+\abs{h_2v_{1,x}}+\abs{h_{2,x}v_1}+\abs{h_{1,x}v_2}\Big)\nonumber\\
	&+\mathcal{O}(1) |h_{1,x}|^2\chi_{\{|\frac{\hat{h}_1}{h_1}|\geq \delta_1\}}. \label{formevarphi}
\end{align}
We have seen that taking the initial data $h(0,\cdot)$ sufficiently small in $L^1$ we  can assume that
\begin{equation}
\|(h)_1(\hat{t},\cdot)\|_{L^1}+\|(h)_2(\hat{t},\cdot)\|_{L^1}\leq\frac{\delta_0}{5\kappa_1}.
\label{estimthat}
\end{equation}
In addition the solution $h$ of \eqref{equationh} satisfies the estimates \eqref{estimate:parabolic-2-h} for $t\in]0,\hat{t}]$ and we know that the solution can be extended as long as the $L^1$ norm of $h$  remains small. Let us denote the time $T$ satisfying:
\begin{equation}
T=\sup \{t,\int^{t}_{\hat{t}}|\tilde{\varphi}_2(s,x)| ds dx\leq\frac{\delta_0}{5}\}.
\label{10.3-h}
\end{equation}
We want to prove that $T=+\infty$. We proceed again by contradiction and assume now that $T<+\infty$.
 Since $h_1$ and $h_2$ satisfy for $t\geq\hat{t}$ respectively the equations \eqref{7.7} and \eqref{7.8}, applying maximum principle we deduce that for any $t\in[\hat{t},T]$  we have
\begin{equation}
\begin{aligned}
\|h_1(t,\cdot)\|_{L^1}&\leq \|(h)_{1}(\hat{t},\cdot)\|_{L^1}\\\
\|h_2(t,\cdot)\|_{L^1}&\leq \|h_{2}(\hat{t},\cdot)\|_{L^1}+\int^t_{\hat{t}}\int_{\R}|\tilde{\varphi}_2(s,x)|dxds\\
&\leq (\kappa_1 \|h_{1}(\hat{t},\cdot)\|_{L^1}+\|(h)_2(\hat{t},\cdot)\|_{L^1})+\frac{\delta_0}{5}.
%&\leq \frac{\delta_0}{2\kappa_1}.
\end{aligned}
\end{equation}
It implies now that for $t\in[\hat{t},T]$, we get using \eqref{estimthat}
\begin{equation}
 \|(h)_{1}(t,\cdot)\|_{L^1}+\|(h)_{2}(t,\cdot)\|_{L^1}\leq (1+ 2\kappa_1) \|(h)_{1}(\hat{t},\cdot)\|_{L^1}+\|(h)_{2}(\hat{t},\cdot)\|_{L^1}+\frac{\delta_0}{5}\leq\delta_0.
 \end{equation}
 It means that we can apply the Corollary \ref{coro2.3} and the Lemma \ref{lemme4.2} for any $t\in[\hat{t},T]$ and in particular we control the $L^\infty$ norm of $h_i$, $h_{i,x}$ and $\hat{h}_1,\hat{h}_{1,x}$.\\
 We wish now to estimate $\int^t_{\hat{t}}\int_{\R}|\tilde{\varphi}_2(s,x)|ds dx$ and to prove that
 \begin{equation}
 \int^t_{\hat{t}}\int_{\R}|\tilde{\varphi}_2(s,x)|ds dx=O(1)\delta_0^2,
 \label{10.5-h}
 \end{equation}
 which would contradict the fact that $T$ is a supremum. To do this, we must estimate each terms of \eqref{formevarphi}. Using  \eqref{estim6cru} and \eqref{estim2crubis}, we have
\begin{equation}
\begin{aligned}
\int_{\hat{t}}^T\int_{\R} \big(\abs{w_1+\si_1v_1}(\abs{h_1}+\abs{h_{1,x}})+\abs{w_{1,x}v_1-w_{1}v_{1,x}}\big)dx ds&=\mathcal{O}(1)\de_0^2,
\end{aligned}
\label{top0}
 \end{equation}
 Using now the Lemma \ref{lemme4.2}, Lemma \ref{lemma-sc-1} and \eqref{utile} we deduce that
 \begin{equation}
 \begin{aligned}
\int_{\hat{t}}^T\int_{\R}\Big((\abs{h_{1,x}}+\abs{h_1})\abs{v_1}\left(\frac{w_1}{v_1}\right)_x^2+\abs{h_{1,x}v_1-h_1v_{1,x}}+\abs{h_{1,x}w_1-h_1w_{1,x}}\Big)dx ds&=\mathcal{O}(1)\de_0^2.
\end{aligned}
\label{top0a}
\end{equation}
Since $\tilde{\varphi}_2$ satisfies \eqref{10.5-h}, applying Propositions \ref{prop5.2} and \ref{prop5.3} we obtain that
\begin{equation*}
\int_{\hat{t}}^T\int_{\R}\big(	\abs{h_1v_2}+\abs{h_1v_{2,x}}+\abs{h_2v_{1}}+\abs{h_2v_{1,x}}+\abs{h_{2,x}v_1}+\abs{h_{1,x}v_2} dx ds=\mathcal{O}(1)\de_0^2.
\end{equation*}
Now we are going to estimate the new terms compared to \cite{BB-triangular} $\abs{h_{1,xx}v_1-h_1v_{1,xx}}$ and $\abs{h_{1,xx}w_1-h_1w_{1,xx}}$. We recall that $\hat{h}_1=\al_1 h_{1,x}-\tilde{\la}_1h_1$. Then $h_{1,t}=\hat{h}_{1,x}$. This implies
\begin{align*}
	\hat{h}_{1,t}&=\al_1 h_{1,tx}+u_{1,t}\al_1^\p h_{1,x}-(\la_1-v_1\al_1^\p)h_{1,t}-(\la_1^\p-v_1\al_1^{\p\p})u_{1,t}h_{1}+v_{1,t}\al_1^\p h_1\\
	&=(\al_1\hat{h}_{1,x}-(\la_1-v_1\al^\p_1)\hat{h}_1)_x-\al_1^\p v_{1}\hat{h}_{1,x}+u_{1,t}\al_1^\p h_{1,x}\\
	&+(\la_1^\p-v_1\al_1^{\p\p})v_{1}\hat{h}_{1}-(\la_1^\p-v_1\al_1^{\p\p})u_{1,t}h_{1}+v_{1,t}\al_1^\p h_1-v_{1,x}\al_1^\p \hat{h}_1.
\end{align*}
Since $u_{1,t}=w_1$ and $v_{1,t}=w_{1,x}$ we have
\begin{align*}
	\hat{h}_{1,t}
	&=(\al_1\hat{h}_{1,x}-\tilde{\lambda}_1\hat{h}_1)_x-\al_1^\p v_1\hat{h}_{1,x}+w_1\al_1^\p h_{1,x}\\
	&+(\la_1^\p-v_1\al_1^{\p\p})v_1\hat{h}_{1}-(\la_1^\p-v_1\al_1^{\p\p})w_1h_{1}+w_{1,x}\al_1^\p h_1-v_{1,x}\al_1^\p \hat{h}_1\\
	&=(\al_1\hat{h}_{1,x}-\tilde{\lambda}_1\hat{h}_1)_x-\al_1^\p (v_1\hat{h}_{1,x}-v_{1,x}\hat{h}_1)+2\al_1^\p(w_1h_{1,x}- v_{1,x}\hat{h}_{1})\\
	&+(\la_1^\p-v_1\al_1^{\p\p})(v_1(\al_1h_{1,x}-\tilde{\la}_1h_1)-(\al_1v_{1,x}-\tilde{\la}_1v_1)h_{1})\\
	&+\al_1^\p (w_{1,x} h_1-w_1h_{1,x})\\
	&=(\al_1\hat{h}_{1,x}-\tilde{\lambda}_1\hat{h}_1)_x-\al_1^\p (v_1\hat{h}_{1,x}-v_{1,x}\hat{h}_1)+\al_1 (\la_1^\p-v_1\al_1^{\p\p})(v_1h_{1,x}-v_{1,x}h_1)\\
	&+\al_1^\p (w_{1,x} h_1-w_1h_{1,x})+2\al_1^\p((\al_1 v_{1,x}-\tilde{\la}_1v_1)h_{1,x}-v_{1,x}(\al_1 h_{1,x}-\tilde{\la}_1h_1))\\
	&=(\al_1\hat{h}_{1,x}-\tilde{\lambda}_1\hat{h}_1)_x-\al_1^\p (v_1\hat{h}_{1,x}-v_{1,x}\hat{h}_1)+\al_1 (\la_1^\p-v_1\al_1^{\p\p})(v_1h_{1,x}-v_{1,x}h_1)\\
	&+\al_1^\p (w_{1,x} h_1-w_1h_{1,x})+2\tilde{\la}_1\al_1^\p(h_1v_{1,x}-h_{1,x}v_1).
\end{align*}
It implies that $\hat{h}_1$ satisfies the following equation
\begin{equation}
	\hat{h}_{1,t}+(\tilde{\lambda}_1\hat{h}_1)_x-(\alpha_1\hat{h}_{1,x})_x=\hat{\varphi}_1,
	\label{equationhath}
\end{equation}
with
\begin{equation}
\begin{aligned}
\hat{\varphi}_1=&-\al_1^\p (v_1\hat{h}_{1,x}-v_{1,x}\hat{h}_1)+\al_1 (\la_1^\p-v_1\al_1^{\p\p})(v_1h_{1,x}-v_{1,x}h_1)+\al_1^\p (w_{1,x} h_1-w_1h_{1,x})\\
&+2\tilde{\la}_1\al_1^\p(h_1v_{1,x}-h_{1,x}v_1).
\end{aligned}
\label{vah1}
\end{equation}
Now we can calculate
\begin{equation*}
	\begin{aligned}
		&h_{1,xx}v_1-h_1v_{1,xx}=\frac{1}{\al_1}\big(\al_1h_{1,xx}v_1-h_1\al_1v_{1,xx}\big)\\
		&=\frac{1}{\al_1}\big(\hat{h}_{1,x}v_1+\tilde{\la}_{1,x}h_1v_1+\tilde{\la}_1h_{1,x}v_1-h_1w_{1,x}-\tilde{\la}_{1,x}v_1h_1-\tilde{\la}_1h_1v_{1,x}- \alpha'_1 v_1(v_1 h_{1,x}-h_1v_{1,x})\big)\\
		&=\frac{1}{\al_1}\left(\hat{h}_{1,x}v_1-\hat{h}_1v_{1,x}+h_{1,x}w_1-h_{1}w_{1,x}\right)\\
		&+\frac{1}{\al_1}\left(\hat{h}_{1}v_{1,x}-h_{1,x}w_1+\tilde{\la}_1h_{1,x}v_1-\tilde{\la}_1h_1v_{1,x}- \alpha'_1 v_1(v_1 h_{1,x}-h_1v_{1,x})\right).
	\end{aligned}
\end{equation*}
Then we can simplify
\begin{equation}
\begin{aligned}
	&h_{1,xx}v_1-h_1v_{1,xx}=\frac{1}{\al_1}\big(\al_1h_{1,xx}v_1-h_1\al_1v_{1,xx}\big)\\
	&=\frac{1}{\al_1}\left(\hat{h}_{1,x}v_1-\hat{h}_1v_{1,x}+h_{1,x}w_1-h_{1}w_{1,x}\right)\\
	&+\frac{1}{\al_1}\left(\al_1h_{1,x}v_{1,x}-\tilde{\la}_1h_1v_{1,x}-\al_1h_{1,x}v_{1,x}+\tilde{\la}_1h_{1,x}v_1\right)\\
	&+\frac{1}{\al_1}\left(\tilde{\la}_1h_{1,x}v_1-\tilde{\la}_1h_1v_{1,x}-\alpha'_1 v_1(v_1 h_{1,x}-h_1v_{1,x})\right)\\
	&=\frac{1}{\al_1}\left((\hat{h}_{1,x}v_1-\hat{h}_1v_{1,x})+(h_{1,x}w_1-h_{1}w_{1,x})\right)+(\frac{2\tilde{\la}_1}{\al_1}-\frac{\alpha'_1 v_1}{\alpha_1})  \left(h_{1,x}v_1-h_1v_{1,x}\right).
\end{aligned}
\label{supercru4}
\end{equation}
Therefore, we have applying  Lemma \ref{lemma-sc-1} to the couples $(\hat{h}_1,v_1)$, $(h_1,w_1)$ and $(v_1,h_1)$ and using the Lemma \ref{lemme4.2} we get
\begin{equation}
	\int_{\hat{t}}^T\int_{\R}|	h_{1,xx}v_1-h_1v_{1,xx}| dx ds=\mathcal{O}(1)\de_0^2.
	\label{top1}
\end{equation}
Similarly we have
\begin{align*}
	h_{1,xx}w_1-h_1w_{1,xx}	&=\frac{1}{\al_1}\left((\hat{h}_{1,x}w_1-\hat{h}_1w_{1,x})+(h_{1,x}z_1-h_{1}z_{1,x})\right)\\
	&+(\frac{2\tilde{\la}_1}{\al_1}-\frac{\alpha'_1 v_1}{\alpha_1})  \left(h_{1,x}w_1-h_1w_{1,x}\right).
\end{align*}
Again using  Lemma \ref{lemma-sc-1} to the couples $(\hat{h}_1,w_1)$, $(h_1,z_1)$ and $(h_1,w_1)$ and using the Lemma \ref{lemme4.2} we get
\begin{equation}
	\int_{\hat{t}}^T\int_{\R}|	h_{1,xx}w_1-h_1w_{1,xx}| dx ds=\mathcal{O}(1)\de_0^2.
	\label{top2}
\end{equation}
It remains now to deal with the terms  $|h_{1,x}|^2\chi_{\{|\frac{\hat{h}_1}{h_1}|\geq \delta_1\}}$,$ |(w_1-\sigma_1 v_1)(|h_1|+|h_{1,x}|)$ which as in the previous section required energy estimate. To do this we multiply again the equation \eqref{7.7} by $h_1\hat{\theta}(\frac{\hat{h}_1}{h_1})$ with $\hat{\theta}$ defined as in \eqref{6.11} and we integrate. We obtain then
$$
\begin{aligned}
&\int_{\R}\biggl((\frac{h_1^2\hat{\theta}}{2})_t-\frac{h_1^2}{2}(\hat{\theta}_t+2\tilde{\lambda}_1\hat{\theta}_x-(\alpha_1\hat{\theta}_x)_x)+2\alpha_1 h_{1,x}h_1\hat{\theta}_x+\hat{\theta} h_{1,x}(\alpha_1 h_{1,x}-\tilde{\lambda}_1 h_1)\biggl) dx=0.
\end{aligned}
$$
It implies that
\begin{equation}
\begin{aligned}
\int_\R \hat{\theta} h_{1,x}(\alpha_1 h_{1,x}-\tilde{\lambda}_1 h_1)&=- \int_{\R}(\frac{h_1^2\hat{\theta}}{2})_t dx+\int_\R \frac{h_1^2}{2}(\hat{\theta}_t+\tilde{\lambda}_1\hat{\theta}_x-(\alpha_1\hat{\theta}_x)_x)dx\\
&-\int_{\R} 2\alpha_1 h_{1,x}h_1\hat{\theta}_x+\int_\R \frac{h_1^2}{2}\tilde{\lambda}_1\hat{\theta}_x.%-\frac{1}{2}\int_\R v_1^2\tilde{\lambda}_{1,x}\theta dx
\end{aligned}
\label{formenergiea}
\end{equation}
A direct computation gives using the equation \eqref{equationhath}
\begin{equation}
\begin{aligned}
&\hat{\theta}_t+\tilde{\lambda}_1\hat{\theta}_x-(\alpha_1\hat{\theta}_x)_x=\hat{\theta}'(\frac{\hat{h}_{1,t}}{h_1}-\frac{h_{1,t}\hat{h}_1}{h_1^2})+\tilde{\lambda}_1(\frac{\hat{h}_{1,x}}{h_1}-\frac{h_{1,x}\hat{h}_1}{h_1^2})\\
&-\alpha'_1 v_1\hat{\theta}'\left(\frac{\hat{h}_1}{h_1}\right)_x-\alpha_1\hat{\theta}''\left(\frac{\hat{h}_1}{h_1}\right)_x^2
+2\alpha_1\hat{\theta}'\frac{h_{1,x}}{h_1}\left(\frac{\hat{h}_1}{h_1}\right)_x-\alpha_1\hat{\theta}'(\frac{\hat{h}_{1,xx}h_1-h_{1,xx}\hat{h}_1}{v_1^2})\\
&=\frac{\hat{\theta}'}{h_1}(\hat{h}_{1,t}+\tilde{\lambda}_1 \hat{h}_{1,x}-(\alpha_1 \hat{h}_{1,x})_x)-\frac{\hat{\theta}' \hat{h}_1}{h_1^2}(
h_{1,t}+\tilde{\lambda}_1 h_{1,x}-(\alpha_1 h_{1,x})_x)-\alpha_1\hat{\theta}''\left(\frac{\hat{h}_1}{h_1}\right)_x^2
\\
&+2\alpha_1\hat{\theta}'\frac{h_{1,x}}{h_1}\left(\frac{\hat{h}_1}{h_1}\right)_x\\
&=\frac{\hat{\theta}'}{h_1}\hat{\varphi}_1-\alpha_1\hat{\theta}''\left(\frac{\hat{h}_1}{h_1}\right)_x^2
+2\alpha_1\hat{\theta}_x\frac{h_{1,x}}{h_1}.
\end{aligned}
\end{equation}
%\begin{align*}
%\hat{\varphi}_1=&-\al_1^\p (v_1\hat{h}_{1,x}-v_{1,x}\hat{h}_1)+\al_1 (\la_1^\p-v_1\al_1^{\p\p})(v_1h_{1,x}-v_{1,x}h_1)+\al_1^\p (w_{1,x} h_1-w_1h_{1,x})\\
%&+2\tilde{\la}_1\al_1^\p(h_1v_{1,x}-h_{1,x}v_1).
%\end{align*}
Plugging this last expression in \eqref{formenergiea}, it implies that
\begin{equation}
\begin{aligned}
\int_\R \hat{\theta} h_{1,x}(\alpha_1 h_{1,x}-\tilde{\lambda}_1 h_1)&=- \int_{\R}(\frac{h_1^2\hat{\theta}}{2})_t dx-\frac{1}{2}\int_\R \alpha_1\hat{\theta}'' h_1^2\left(\frac{\hat{h}_1}{h_1}\right)_x^2dx-\int_{\R} \alpha_1 \frac{h_{1,x}}{h_1}h_1^2\hat{\theta}_x\\
&+\int_\R \frac{h_1^2}{2}\tilde{\lambda}_1\hat{\theta}_x+\int_{\R}\hat{\theta}' h_1\hat{\va}_1.%-\frac{1}{2}\int_\R v_1^2\tilde{\lambda}_{1,x}\theta dx
\end{aligned}
\label{formenergie1a}
\end{equation}
We observe that $\hat{\theta}(\frac{\hat{h}_1}{h_1})\ne 0$ if
$|\hat{h}_1|=|\alpha_1 h_{1,x}-\tilde{\lambda}_1 h_1|\geq \frac{3\delta_1}{5}|h_1|$ which gives using \eqref{5.7}
\begin{equation}
|\alpha_1 h_{1,x}|\geq \frac{3\delta_1}{5} |h_1|-\|\tilde{\lambda}_1\|_{L^\infty}|h_1|\geq 2\|\tilde{\lambda}_1\|_{L^\infty}|h_1|.
\label{impotia}
\end{equation}
From \eqref{formenergie1a}, \eqref{impotia}, \eqref{condialpha}, \eqref{vah1}, \eqref{supercru4} we get for $t\in[\hat{t},T]$
%\begin{equation}
%\begin{aligned}
%&\frac{c_1}{2}\int_\R \alpha_1\theta v_{1,x}^2\leq- \int_{\R}(\frac{v_1^2\theta}{2})_t dx+\frac{M}{2}\int_\R v_1^2 | \theta'' |\big(\frac{w_1}{v_1}\big)_x^2dx+M\int_{\R} | \theta'  v_1 v_{1,x}\big(\frac{w_1}{v_1}\big)_x |dx\\
%&+2M\int_{\R}| v_{1,x}v_1\theta_x|-\frac{1}{2}\int_\R v_1^2\tilde{\lambda}_{1,x}\theta  dx
%\end{aligned}
%\label{formenergie2}
%\end{equation}
%Now after integration by parts we get
%$$
%\begin{aligned}
%&-\frac{1}{2}\int_\R v_1^2\tilde{\lambda}_{1,x}\theta  dx=\int_\R \tilde{\lambda}_{1}\theta v_1v_{1,x} dx+\frac{1}{2}\int_{\R}v_1^2\tilde{\lambda}_{1}\theta'\big(\frac{w_1}{v_1}\big)_x dx\\
%&\leq \frac{1}{2} \|\tilde{\lambda}_{1}\|_{L^\infty} \int_{\R}|\theta'| |w_{1,x}v_1-v_{1,x}w_1| dx+\frac{1}{2}\|\tilde{\lambda}_1\|_{L^\infty}\int_{\R}\theta |v_{1,x}|^2 dx
%\end{aligned}
%$$
%Plugging the previous estimate in \eqref{formenergie} and using the fact that $|v_{1,x}|=O(1)|v_1|$  when $\theta'\ne 0$  and the fact
%that $\|\tilde{\lambda}_1(t,\cdot)\|_{L^\infty}\leq c_1\leq\alpha_1(u)$ for $\delta_0$ sufficiently small when $t\in[\hat{t},T]$, we deduce that for $t\in[\hat{t},T]$ 
\begin{equation}
\begin{aligned}
\frac{c_1}{2 }\int_\R \theta h_{1,x}^2\,dx&\leq- \int_{\R}\left(\frac{h_1^2\theta}{2}\right)_t dx+O(1)\int_{\{ |\frac{\hat{h}_1}{h_1}|\leq \delta_1\} } h_1^2 | \theta'' |\left(\frac{\hat{h}_1}{h_1}\right)_x^2dx\\
&+O(1)\int_\R(1+|\frac{h_{1,x}}{h_1}|  \chi_{\{|\frac{\hat{h}_1}{h_1}|\leq\frac{4\delta_1}{5}\}} )|\hat{h}_{1,x}h_1-h_{1,x}\hat{h}_1|  dx\\
&+O(1)\int_{\R}\big(|v_1h_{1,xx}-v_{1,x}h_{1,xx}|+|v_1h_{1,x}-v_{1,x}h_1|\big)dx\\
&+O(1)\int_{\R}|w_{1,x} h_1-w_1h_{1,x}|\,dx.
\end{aligned}
\label{formenergie3a}
\end{equation}
It remains now to estimate the term $\int_{\{ |\frac{\hat{h}_1}{h_1}|\leq \delta_1\} } h_1^2 | \theta'' |\left(\frac{\hat{h}_1}{h_1}\right)_x^2dx$, as previously we set:
$${\cal L}_1(t)=\int_{\R}\sqrt{h_1^2+\hat{h}_1^2} (t,x)dx.$$
As previously we can prove that
$$
\begin{aligned}
\frac{d}{dt}{\cal L}_1(t)\leq-\int \frac{ \alpha_1 |h_1||\left(\frac{\hat{h}_1}{h_1}\right)_x|^2}{(1+\left(\frac{\hat{h}_1}{h_1}\right)^2)^{\frac{3}{2}}}dx+\|\hat{\varphi}_1(t,\cdot)\|_{L^1}.
\end{aligned}
$$
We deduce then that from \eqref{vah1}, \eqref{top0a}, \eqref{top2} and Lemma \ref{lemme4.2}, we have
\begin{equation}
\int_{\hat{t}}^T\int_{\{ |\frac{\hat{h}_1}{h_1}|\leq \delta_1\} } h_1^2 | \theta'' |\left(\frac{\hat{h}_1}{h_1}\right)_x^2dx=O(1)\delta_0^2.
\label{techj}
\end{equation}
Combining now \eqref{formenergie3a}, \eqref{techj}, \eqref{top0a} and the Lemma \ref{lemma-sc-1}
applied to the couple $(h_1,\hat{h}_1)$ we deduce that
\begin{equation}
\begin{aligned}
&\int_{\hat{t}}^T\int_{\{|\frac{\hat{h}_1}{h_1}|\geq\frac{4\delta_1}{5}\}} h_{1,x}^2 dx ds=O(1)\delta_0^2.
\end{aligned}
\label{energiefinale1}
\end{equation}
We can now estimate the term $(|h_{1,x}|+|h_1|)(w_1+\si_1v_1)$. First we observe that this term is different from zero if $|\frac{w_1}{v_1}|\geq\delta_1$ then we have:
\begin{align*}
	&\int\limits_{\hat{t}}^{T}\int\limits_{\R}\abs{w_1+\si_1v_1}(|h_1|+\abs{h_{1,x}})\,dxds\\
	&\leq \int\limits_{\hat{t}}^{T}\int_{\{|\frac{\hat{h}_1}{h_1}|\leq\frac{4\delta_1}{5}\}\cup \{|\frac{w_1}{w_1}|\geq\delta_1\} }\abs{w_1+\si_1v_1}(|h_1|+\abs{h_{1,x}}) dx ds\\
	&+\int\limits_{\hat{t}}^{T}\int_{\{|\frac{\hat{h}_1}{h_1}|\geq\frac{4\delta_1}{5}\}\cup \{|\frac{w_1}{w_1}|\geq\delta_1\} }\abs{w_1+\si_1v_1}(|h_1|+\abs{h_{1,x}})dx ds.
\end{align*}
First we have
\begin{align}
&\int\limits_{\hat{t}}^{T}\int_{\{|\frac{\hat{h}_1}{h_1}|\geq\frac{4\delta_1}{5}\}\cup \{|\frac{w_1}{w_1}|\geq\delta_1\} }\abs{w_1+\si_1v_1}(|h_1|+\abs{h_{1,x}})dx ds\\
&=O(1)\int\limits_{\hat{t}}^{T}\int_{\{|\frac{\hat{h}_1}{h_1}|\geq\frac{4\delta_1}{5}\} }(h_{1,x})^2 dx ds+O(1)\int\limits_{\hat{t}}^{T}\int_{ \{|\frac{w_1}{w_1}|\geq\delta_1\} }|w_1+\si_1v_1|^2dx ds\\
&= O(1)\int\limits_{\hat{t}}^{T}\int_{\{|\frac{\hat{h}_1}{h_1}|\geq\frac{4\delta_1}{5}\} }(h_{1,x})^2 dx ds+O(1)\int\limits_{\hat{t}}^{T}\int_{ \{|\frac{w_1}{w_1}|\geq\delta_1\} }(v_{1,x})^2 dx ds.\label{estima1}
\end{align}
We observe now that $\{|\frac{\hat{h}_1}{h_1}|\leq\frac{4\delta_1}{5}\}\cup \{|\frac{w_1}{v_1}|\geq\delta_1\}$ is included in $\{|\frac{\hat{h}_1}{h_1}|\leq\frac{4}{5}|\frac{w_1}{v_1}|\}\cup\{|\alpha_1h_{1,x}|\leq \delta_1 |h_1|\}$ provided that $\delta_0$ is sufficiently small in terms of $\delta_1$ then we have
\begin{align*}
	\abs{h_{1,x}v_1-h_1v_{1,x}}=\frac{1}{\al_1}\abs{\hat{h}_1v_1-w_1h_1}=\frac{\abs{h_1v_1}}{\al_1}\abs{\frac{w_1}{v_1}-\frac{\hat{h}_1}{h_1}}\geq \frac{\abs{h_1w_1}}{5\al_1}\geq  \frac{\abs{h_{1,x} w_1}}{5\delta_1}\geq \frac{\abs{h_{1,x} v_1}}{5}
\end{align*}
Therefore we get
\begin{align}
\int\limits_{\hat{t}}^{T}\int_{\{|\frac{\hat{h}_1}{h_1}|\leq\frac{4\delta_1}{5}\}\cup \{|\frac{w_1}{w_1}|\geq\delta_1\} }\abs{w_1+\si_1v_1}(|h_1|+\abs{h_{1,x}}) dx ds&=O(1)\int\limits_{\hat{t}}^{T}\int \abs{h_{1,x}v_1-h_1v_{1,x}} dx ds\nonumber\\
&=O(1)\delta_0^2.
\label{estima2}
\end{align}
Combining \eqref{energiefinale1}, \eqref{estima1}, \eqref{estima2} and  \eqref{estim6cru}, we deduce that
\begin{equation}
	\int\limits_{\hat{t}}^{T}\int\limits_{\R}\abs{w_1+\si_1v_1}(|h_1|+\abs{h_{1,x}})\,dxds=O(1)\delta_0^2.
	\label{superfinal}
	\end{equation}
%Now we define $T_*=\sup\left\{T>0,\,\norm{\tilde{\varphi}_1}_{L^1([0,T]\times\R)}\leq\frac{\de_0}{2}\right\}$. 
From \eqref{top0}, \eqref{top0a}, \eqref{top1}, \eqref{top2} and \eqref{superfinal} we deduce that $T_*=+\f$. %If $T_*<+\f$ then by previous argument we can show that $\norm{\tilde{\varphi}_1}_{L^1([0,T_*]\times\R)}=\mathcal{O}(1)\de_0^2<\de_0/2$ for sufficiently small $\de_0>0$. 
This completes the proof of  \eqref{thm1-Lipschitz}.

We wish to conclude this section by proving the estimate \eqref{L1-cont} which will finish the proof of the Theorem \ref{theo1}. From the Proposition \ref{prop:parabolic} and the Corollary \ref{coro2.2}, we have seen that there exists $C>0$ large enough such that
\begin{equation}
\begin{cases}
\|u_{xx}(t,\cdot)\|_{L^1}&\leq \frac{C\delta_0}{\sqrt{t}}\;\;\mbox{if}\;0<t\leq\hat{t},\\
&\leq \frac{C\delta_0}{\sqrt{\hat{t}}}\;\;\mbox{if}\;t\geq\hat{t}.
\label{7.52}
\end{cases}
\end{equation}
We deduce in particular that there exists $L_3>0$ such that for any $t>0$
\begin{equation}
\|u_t(t,\cdot)\|_{L^1}\leq L_3(1+\frac{1}{2\sqrt{t}}).
\label{ut}
\end{equation}
For any $0\leq s<t$ we have then using \eqref{ut}
$$\|u(t)-u(s)\|_{L^1}\leq\int^t_s\|u_t(\theta)\|_{L^1}d\theta\leq L_3(|t-s|+|\sqrt{t}-\sqrt{s}|).$$
\section{Vanishing viscosity limit}\label{sec:vv-lim}
As claimed in Corollary \ref{theo2} we want to prove vanishing viscosity limit as $\e\rr0$ for the following Cauchy problem
\begin{equation}
	u^\e_t+A(u^\e)u^\e_x=\e(B(u^\e)_x)_x\mbox{ and }u^\e(0,x)=\bar{u}(x).
\end{equation}
Note that if we define $u$ as follows $u^\e(t,x)=u(\frac{t}{\e},\frac{x}{\e})$ then $u$ solves the following problem with fix viscosity but scaled initial data, 
\begin{equation}
	u_t+A(u)u_x=(B(u)_x)_x\mbox{ and }u(0,x)=\bar{u}(\e x).
\end{equation}
Observe that
\begin{align*}
	TV(\bar{u}(\e \cdot ))&= TV(\bar{u}(\cdot)),\\
	\norm{\bar{u}(\e\cdot )}_{L^1}&=\frac{1}{\e}	\norm{\bar{u}(\cdot )}_{L^1}.
\end{align*}
%From section \ref{sec:BV}, \ref{sec:stability} and \ref{sec:propagation} we get
%\begin{align*}
%	TV(u(t))&\leq L_1 TV(\bar{u}),\\
%	\norm{u(t)-v(t)}_{L^1}&\leq \frac{L_2}{\e}\norm{\bar{u}-\bar{v}}_{L^1},\\
%	\norm{u(t)-u(s)}_{L^1}&\leq L_3\left(\abs{\frac{t}{\e}-\frac{s}{\e}}+\abs{\frac{\sqrt{t}}{\sqrt{\e}}-\frac{\sqrt{s}}{\sqrt{\e}}}\right),
%\end{align*}
%if $\bar{u}(x)=\bar{v}(x)$ for $x\in[a,b]$ then we have
%\begin{equation*}
%	\abs{u(t,x)-v(t,x)}\leq \al_1\norm{\bar{u}-\bar{v}}_{L^\f}\left(e^{c_1(\B_1 t-(x-a/\e))}+e^{c_1(\B_1 t+(x-b/\e))}\right).
%\end{equation*}
Therefore using the Theorem \ref{theo1}, we obtain for $0\leq s\leq t$
\begin{align}
	TV(u^\e(t))&\leq L_1 TV(\bar{u}),\label{est-1}\\
	\norm{u^\e(t)-v^\e(t)}_{L^1}&=\e\norm{u(t)-v(t)}_{L^1}\leq L_2\norm{\bar{u}-\bar{v}}_{L^1},\label{est-2}\\
	\norm{u^\e(t)-u^\e(s)}_{L^1}&\leq L_3\left(\abs{t-s}+\sqrt{\e}\abs{\sqrt{t}-\sqrt{s}}\right),\label{est-3}.
	%\abs{u^\e(t,x)-v^\e(t,x)}&\leq \al_1\norm{\bar{u}-\bar{v}}_{L^\f}\left(e^{\frac{c_1}{\e}(\B_1 t-(x-a))}+e^{\frac{c_1}{\e}(\B_1 t+(x-b))}\right).\label{est-4}
\end{align}
The convergence of $u^\e$ as $\e\rr0$ follows from a standard argument with an application of Helly's theorem and the $L^1$ continuity \eqref{est-3}. Indeed, due to the uniform TV estimate \eqref{est-1} by using Helly's theorem we can pass to a the limit (up to a subsequence) for a countable dense set $\{t_n\}$ and then applying $L^1$ continuity we can define the limit function at all time $t>0$. We set
\begin{equation}
	L^1_{loc}-\lim\limits_{k\rr\f}u^{\e_k}(t,\cdot)=u(t,\cdot).
\end{equation}
We wish no to prove that $u$ is a global weak solution for the system \eqref{hyperbo}. In particular up to a subsequence $(u^{\e})_{\e>0}$ converges almost everywhere to $u$, it implies in particular that for any function $\varphi$ in $C^\infty_0(\R^+\times\R)$ we have applying dominated convergence
\begin{equation}
\begin{aligned}
&\int_{\R^+}\int_\R\varphi_x(s,x)f_1(u_1^\e(s,x) ds dx\rightarrow_{\e\rightarrow0} \int_{\R^+}\int_\R\varphi_x(s,x)f_1(u(s,x) ds dx\\
&\int_{\R^+}\int_\R\varphi_x(s,x)g(u^\e(s,x) ds dx\rightarrow_{\e\rightarrow0} \int_{\R^+}\int_\R\varphi_x(s,x)g(u(s,x) ds dx.
\end{aligned}
\end{equation}
Similarly we have for any $T>0$
 \begin{equation}
\begin{aligned}
&\e\int^T_0 \int_{\R^+}\int_\R |B(u^\e)u^\e_x| dx ds\leq \sup_{\{|v-u^*|\leq 2\delta_0\}}|B(v)|\delta_0\e T,
\end{aligned}
\end{equation}
then $\e B(u^\e)u^\e_x$ converges to zero in $L^1_{loc, t,x}$. It implies that $u$ is a global weak solution of \eqref{hyperbo}. In addition as in \cite{BB-vv-lim-ann-math} we can show that the solution $u$ satisfies  the Liu condition for any shock. Applying the result of \cite{BDL}, it implies that the solution $u$ of \eqref{hyperbo} is unique since it is a global weak solution small in $BV$ and satisfying Liu shocks conditions. In particular it means that the sequence $(u^\e)_{\e>0}$ has a unique accumulation point and then the convergence of $(u^\e)_{\e>0}$ to $u$ is strong in $L^1_{loc}$. It completes the proof of the Corollary \ref{theo2}.%
%Then we have
%\begin{align}
%	&((Bh)_x+(h\cdot DBu_x-u_x\cdot DB h)-Ah)_x\\
%	&=(\al_1h_1\tilde{r}_1+(s_1h_1+\al_2 h_2)r_2)_x-h_1A\tilde{r}_1-\la_2h_2r_2\\
%	&+(h_1v_2-h_2v_1) (\tilde{r}_1\cdot DB r_2-r_2\cdot DB\tilde{r}_1).
%\end{align}

%\begin{align*}
%	 &+\frac{1}{\al_1}v_{1}(\la_1-v_1\al_1^\p-\si_1)s_1r_2\\
%	&\quad\quad+\frac{1}{\al_1}v^2_{1}(\la_1-v_1  \al_1^\p-\si_1)s_{1,v}r_2+v_1\si_1\tilde{r}_1.
%\end{align*}

\appendix
\section{Viscous travelling wave}\label{App-1}
The viscous  travelling wave $u(t,x)=U(x-\sigma t)$ of the equation \eqref{eqn-main} satisfies
\begin{equation*}
	(A(U)-\si)U^\p=(B(U)U^\p)^\p.
\end{equation*}
We can write
\begin{equation*}
	U^{\p\p}=B^{-1}(U)(A(U)-\si)U^\p-B^{-1}(U)(U^\p\cdot DB(U))U^\p.
\end{equation*}
We consider the following system of ordinary differential equations
\begin{equation}\label{ODE-1}
	\begin{cases}
	\begin{aligned}
		\dot{u}&=v,\\
		\dot{v}&=B^{-1}(u)(A(u)-\si)v-B^{-1}(u)(v\bullet DB(u))v,\\
		\dot{\si}&=0.
	\end{aligned}
	\end{cases}
\end{equation}
We note that $P^*_i:=(u^*,0,\la_{i}(u^*))$ are equilibrium points for $1\leq i\leq 2$. We linearize near the point $P^*_i$ and get
\begin{equation}
	\begin{cases}
	\begin{aligned}
			\dot{u}&=v,\\
		\dot{v}&=B^{-1}(u^*)(A(u^*)-\la_i(u^*))v,\\
		\dot{\si}&=0.
	\end{aligned}
	\end{cases}
\end{equation} 
Let $\{r_i(u)\}_{1\leq i\leq n}$ and $\{l_i(u)\}_{1\leq i\leq n}$ be the sets of right and left eigenvector of $A(u)$.% such that $\abs{r_i}=\abs{l_i}=1$. 
We denote $r_i^*=r_i(u^*),l_i^*=l_i(u^*)$. We define $V_i,1\leq i\leq 2$ as follows
\begin{equation}
	v=\sum\limits_{j}V_jr_j^*,\quad V_j:=l_j^*\cdot v.
\end{equation}
Then we have
\begin{equation}
	Z=DF(u^*,0,\la_i(u^*))=\begin{pmatrix}
		O_n&I_n& 0\\
		O_n&B^{-1}(u^*)(A(u^*)-\la_i(u^*)I_n)&0\\
		0&0&0
	\end{pmatrix}.
\end{equation}
Subsequently, we get
\begin{equation}
	Z^2=(DF(u^*,0,\la_i(u^*)))^2=\begin{pmatrix}
		O_n&A(u^*)-\la_i(u^*)I_n& 0\\
		O_n&(A(u^*)-\la_i(u^*)I_n)^2&0\\
		0&0&0
	\end{pmatrix}.
\end{equation}
Therefore the center subspace will look like
\begin{equation}
	\mathcal{N}_i:=\left\{(u,v,\si)\in\R^2\times\R^2\times\R;\,V_j=0,j\neq i\right\}.
\end{equation}
Note that $\mbox{dim}(\mathcal{N}_i)=4$, by Center Manifold Theorem \cite{V}, there exists a smooth manifold $\mathcal{M}_i\subset \R^{5}$ which is tangent to $\mathcal{N}$ at $P^*_i$. Furthermore, $\mathcal{M}_i$ has dimension $4$ and is locally invariant under the flow of \eqref{ODE-1}. We can write
\begin{equation}
	V_j=\varphi_j(u,V_i,\si),\quad j\neq i.
\end{equation}
We further consider 
\begin{equation}
	\mathcal{D}_i:=\left\{\abs{u-u^*}<\e,\,\abs{V_i}<\e,\,\abs{\si-\la_i(u^*)}<\e\right\}.
\end{equation} 
Since $\mathcal{M}_i$ is tangent to $\mathcal{N}_i$ we have
\begin{equation}\label{quadratic-varphi-1}
	\varphi_j(u,V_i,\si)=\mathcal{O}(1)\left(\abs{u-u^*}^2+\abs{V_i}^2+\abs{\si-\la_i(u^*)}^2\right).
\end{equation}
Note that equilibrium points $(u,0,\si)$ with $\abs{u-u^*}<\e,\abs{\si-\la_i(u^*)}<\e$ lie in $\mathcal{M}_i$ then we have
\begin{equation}\label{equilibrium-1}
	\varphi_j(u,0,\si)=0\mbox{ for all }j\neq i.
\end{equation}
Hence, we may write
\begin{equation}
	\varphi_j(u,V_i,\si)=\psi_j(u,V_i,\si) V_i,
\end{equation}
for some $\psi_j$. Now, we would like to make a change of coordinates $V_i\mapsto \tilde{V}_i$ as follows
\begin{equation}
	\tilde{V}_i=l_i(u)\cdot v=V_i l_i(u)\cdot (r_i^*+\sum_{j\ne i}\psi_j(u,V_i,\sigma)r_j^*)
\end{equation}
It implies that $\tilde{V}_i=\zeta_{ii}(u,V_i,\si)V_i$ with $\zeta_{ii}(u,V_i,\si)=l_i(u)\cdot (r_i^*+\sum_{j\ne i}\psi_j(u,V_i,\sigma)r_j^*)$
a $C^2$ function such that $\zeta_{ii}(u^*,0,\si)=1$. If we consider the function $f_{(u,\sigma)}(V_i)=\zeta_{ii}(u,V_i,\si)V_i$, we observe that $f_{(u,\sigma)}'(V_i)=\zeta_{ii}(u,V_i,\si)+V_i\zeta_{ii,V_i}(u,V_i,\si)\ne 0$ in ${\cal D}_i$. It implies that $f_{(u,\sigma)}$ is locally invertible and we can make the change of coordinates  $V_i\mapsto \tilde{V}_i$.
%For $k\neq i$ we have  $\tilde{V}_k=\tilde{\phi}_k(u,\tilde{V}_i,\si)$ for some smooth function $\tilde{\phi}$. From \eqref{equilibrium-1} we deduce that $\tilde{V}_k=\tilde{\psi}_k(u,\tilde{V}_i,\si)\tilde{V}_i$. 
Therefore, for any point $(u,v,\si)\in\mathcal{M}_i$ we can write
\begin{equation}
\begin{aligned}
v&=\sum\limits_{k}l_k(u)\cdot v\, r_k(u)=\sum\limits_{k}V_i l_k(u)\cdot (r_i^*+\sum_{j\ne i}\psi_{j}(u,V_i,\sigma)r_{j}^*)r_{k}(u)\\
&=\sum\limits_{k}l_k(u)\cdot v\, r_k(u)=\tilde{V}_ir_i(u)+\sum\limits_{k\ne i}\tilde{V}_i \frac{1}{\xi_{ii}(u,V_i,\sigma)}l_k(u)\cdot (r_i^*+\sum_{j\ne i}\psi_{j}(u,V_i,\sigma)r_{j}^*)\,r_{k}(u)\\
&=\tilde{V}_i\left(r_i(u)+\sum\limits_{j\neq i}\tilde{\psi}_j(u,\tilde{V}_i,\si)r_j(u)\right)=\tilde{V}_i \tilde{r_i} (u,\tilde{V}_i,\si),
\end{aligned}
\end{equation}
with $\tilde{\psi}_j(u,\tilde{V}_i,\si)=\frac{1}{\xi_{ii}(u,V_i,\sigma)}l_j(u)\cdot (r_i^*+\sum_{j\ne i}\psi_{j}(u,V_i,\sigma)r_{j}^*)$.
From \eqref{quadratic-varphi-1} it follows that
\begin{equation}
	\tilde{r}_i(u,\tilde{V}_i,\si)\rr r_i^*\mbox{ as }(u,\tilde{V}_i,\si)\rr(u^*,0,\la_i(u^*)).
\end{equation}
Therefore, we may write $	v=\tilde{V}_i \tilde{r}_i$ and we have
\begin{equation}
	\mathcal{M}_i=\left\{(u,v,\si_i);\,\,v=\tilde{V}_i \tilde{r}_i(u,\tilde{V}_i,\si_i)\right\}\mbox{ for }1\leq i\leq 2,
\end{equation}
provided that $(u,\tilde{V}_i,\si_i)\in\R^2\times\R^2\times\R$ are in a sufficiently small neighborhood of $(u^*,0,\la_i(u^*))$. %We make again a last change of variable by setting $v_i=\tilde{V}_i\|r_i'(u,\tilde{V}_i,\si)\|$. It implies that
%\begin{equation}
%	\mathcal{M}_i=\left\{(u,v,\si_i);\,\,v=v_i\tilde{r}_i(u,\tilde{V}_i,\si_i)\right\}\mbox{ for }1\leq i\leq 2,
%	\end{equation}
Next we derive few more estimates on $\tilde{\psi}_k,\,k\neq i$. 
\begin{claim}\label{claim-3.1}
	In a neighbourhood of $(u^*,0,\la_i^*)$ we have
	\begin{equation}
		\tilde{\psi}_{k}(u,\tilde{V}_i,\si)=\mathcal{O}(1)\abs{\tilde{V}_i}.
	\end{equation}
	Furthermore, we have
	\begin{equation}
		\tilde{\psi}_{k,\si}(u,\tilde{V}_i,\si),\tilde{\psi}_{k,u\si}(u,\tilde{V}_i,\si),\tilde{\psi}_{k,\si\si}(u,\tilde{V}_i,\si)=\mathcal{O}(1)\abs{\tilde{V}_i}.
		\label{tech1}
	\end{equation}
\end{claim}
\begin{proof}[Proof of Claim \ref{claim-3.1}:] Let us consider an $i$ viscous travelling wave, $u(t,x)=U(x-\sigma t)$ and assume that this solution is contained in the center manifold ${\cal M}_i$ then we have $U'(x)=\tilde{V}_i(x)\tilde{r}_i(U(x),\tilde{V}_i(x),\sigma)$ and it yields
	\begin{equation}
		U^{\p\p}=\tilde{V}_i^{\p}[\tilde{r}_i+\tilde{V}_i\tilde{r}_{i,v}]+\tilde{V}^2_i\tilde{r}_{i,u}\tilde{r}_i.
	\end{equation}
	Then we have
	\begin{equation}
		\tilde{V}_i^{\p}B(u)[\tilde{r}_i+\tilde{V}_i\tilde{r}_{i,v}]+\tilde{V}^2_iB\tilde{r}_{i,u}\tilde{r}_i+\tilde{V}_i^2\tilde{r}_i\bullet DB(u)\tilde{r}_i=\tilde{V}_i(A(u)-\si)\tilde{r}_i.
	\end{equation}
	We have
	\begin{align}
		\tilde{V}_i^{\p}=\frac{1}{\langle B(u)[\tilde{r}_i+\tilde{V}_i\tilde{r}_{i,v}],\tilde{r}_i \rangle }(\tilde{\la}_i-\si\|\tilde{r}_i\|^2)\tilde{V}_i+\mathcal{O}\left(\abs{\tilde{V}_i}^2\right),
	\end{align}
	with $\tilde{\lambda}_i=\langle A(u)\tilde{r}_i\,\tilde{r}_i\rangle$. It gives then
	\begin{equation}
		\frac{1}{\langle B(u)[\tilde{r}_i+\tilde{V}_i\tilde{r}_{i,v}],\tilde{r}_i \rangle }(\tilde{\la}_i-\si\|\tilde{r}_i\|^2)\tilde{V}_iB(u)[\tilde{r}_i+\tilde{V}_i\tilde{r}_{i,v}]+\mathcal{O}\left(\abs{\tilde{V}_i}^2\right)=\tilde{V}_i(A(u)-\si)\tilde{r}_i.
	\end{equation}
	Dividing by $\tilde{V}_i$ and passing $\tilde{V}_i\rr0$, we obtain
	\begin{equation}\label{eqn:vi=0}
		\frac{1}{\langle B(u)\tilde{r}_i(u,0,\si),\tilde{r}_i(u,0,\si) \rangle }(\tilde{\la}_i(u,0,\si)-\si\|\tilde{r_i}(u,0,\sigma)\|^2)B(u)\tilde{r}_i(u,0,\si)= (A(u)-\si)\tilde{r}_i(u,0,\si).
	\end{equation}
	Set $b(u,0,\si)=\langle B(u)\tilde{r}_i(u,0,\si),\tilde{r}_i(u,0,\si) \rangle>0$. We get
	\begin{align}
		B(u)\tilde{r}_i(u,0,\si)&=\left(Br_i(u)+\sum\limits_{j\neq i}\tilde{\psi}_j(u,0,\si)Br_j(u)\right)\\
		&=\left(\alpha_i(u)r_i(u)+\sum\limits_{j\neq i}\tilde{\psi}_j(u,0,\si)\alpha_j(u)r_j(u)\right),\\
		A(u)\tilde{r}_i(u,0,\si)&=\left(Ar_i(u)+\sum\limits_{j\neq i}\tilde{\psi}_j(u,0,\si)Ar_j(u)\right)\\
		&=\left(\la_i(u)r_i(u)+\sum\limits_{j\neq i}\tilde{\psi}_j(u,0,\si)\la_j(u)r_j(u)\right).
	\end{align}
	Applying above identities on \eqref{eqn:vi=0} we obtain
	\begin{align}
		&\frac{(\tilde{\la}_i(u,0,\si)-\si\|\tilde{r_i}(u,0,\sigma)\|^2)}{b(u,0,\si)}\alpha_i(u)r_i(u)+	\frac{(\tilde{\la}_i(u,0,\si)-\si\|\tilde{r_i}(u,0,\sigma)\|^2)}{b(u,0,\si)}\\
		&\times\sum\limits_{j\neq i}\tilde{\psi}_j(u,0,\si)\alpha_j(u)r_j(u)=(\la_i(u)-\si)r_i(u)+\sum\limits_{j\neq i}\tilde{\psi}_j(u,0,\si)(\la_j(u)-\si)r_j(u).
	\end{align}
	Therefore, we have 
	\begin{align}
		\frac{(\tilde{\la}_i(u,0,\si)-\si\|\tilde{r_i}(u,0,\sigma)\|^2)}{b(u,0,\si)}\alpha_i(u)&=\la_i(u)-\si,\\
		\frac{(\tilde{\la}_i(u,0,\si)-\si\|\tilde{r_i}(u,0,\sigma)\|^2)}{b(u,0,\si)}\alpha_j(u)&=\la_j(u)-\si\,\,\mbox{ if }\,\,\tilde{\psi}_j(u,0,\si)\neq0.
	\end{align}
	This implies  since $\tilde{\lambda}_i(u,0,\si)=\langle A(u)\tilde{r}_i(u,0,\si)\,\tilde{r}_i(u,0,\si)\rangle$ and using the fact that $\tilde{r}_i(u,0,\si)$ and $\sigma$ are respectively close from $r_i^*$ and $\lambda_i^*$ we can find $\e_0$ small enough such  that  we have
	\begin{equation}
		c_0\leq \abs{\la_i(u)-\la_j(u)}=\frac{\abs{\tilde{\la}_i(u,0,\si)-\si \|\tilde{r_i}(u,0,\sigma)\|^2}}{b(u,0,\si)}\abs{\alpha_i(u)-\alpha_j(u)}\leq \e_0.
	\end{equation}
	This is a contradiction. Hence, $\tilde{\psi}_j(u,0,\si)=0$ for all $j\neq i$.
	
\end{proof}
%\begin{claim}\label{claim-3.2}
%	In a neighbourhood of $(u^*,0,\la_i^*)$ we have
%	\begin{equation}
%		\tilde{\psi}_{k}(u,\tilde{V}_i,\si)=\mathcal{O}(1)\abs{\tilde{V}_i}.
%	\end{equation}
%	Furthermore, we have
%	\begin{equation}
%		\tilde{\psi}_{k,\si}(u,\tilde{V}_i,\si),\tilde{\psi}_{k,u\si}(u,\tilde{V}_i,\si),\tilde{\psi}_{k,\si\si}(u,\tilde{V}_i,\si)=\mathcal{O}(1)\abs{\tilde{V}_i}.
%	\end{equation}
%\end{claim}

\noi\textbf{Acknowledgements:} AJ would like to thank the project PRIN 2022YXWSLR “BOUNDARY ANALYSIS FOR DISPERSIVE AND VISCOUS FLUIDS” - DIT.PN012.008 for supporting his postdoctoral position at IMATI-CNR, Pavia, Italy.


\begin{thebibliography}{99}
	
	\bibitem{Ancona-Bianchini}
	\newblock F. Ancona and S. Bianchini,
	\newblock Vanishing viscosity solutions for general hyperbolic systems with boundary (2003) (preprint IAC-CNR 28)
	
	\bibitem{BB-temple}
	\newblock S. Bianchini and A. Bressan,
	\newblock BV solutions for a class of viscous hyperbolic systems.
	\newblock {\em Indiana Univ. Math. J.} 49 (2000), no. 4, 1673--1713.
	
	\bibitem{BB-triangular}
	\newblock S. Bianchini and A. Bressan, 
	\newblock A center manifold technique for tracing viscous waves.
	\newblock {\em Commun. Pure Appl. Anal.} 1 (2002), no. 2, 161--190.
	
	\bibitem{BB-vv-lim-ann-math}
	\newblock S. Bianchini and A. Bressan, 
	\newblock Vanishing viscosity solutions of nonlinear hyperbolic systems.
	\newblock{\em Ann. of Math.} (2)161(2005), no.1, 223--342.
	
	\bibitem{BGJ}
	\newblock C. Bourdarias, M. Gisclon and S. Junca,
	\newblock Hyperbolic models in gas-solid chromatography.
	\newblock{\em Bol. Soc. Esp. Mat. Apl. SeMA} No. 43 (2008), 29--57.
	
	\bibitem{Bressan-book}
	\newblock A. Bressan, 
	\newblock Hyperbolic systems of conservation laws. The one-dimensional Cauchy problem
	\newblock Oxford Lecture Ser. Math. Appl., 20 Oxford University Press, Oxford, 2000. xii+250 pp.
	
	\bibitem{BDL}
	\newblock A. Bressan and C. De Lellis,
	\newblock A remark on the uniqueness of solutions to hyperbolic conservation laws.
	\newblock {\em Arch. Ration. Mech. Anal.} 247 (2023), no. 6, Paper No. 106, 12 pp.
	
	\bibitem{Chen-Kang-Vas}
	\newblock G. Chen, M.-J. Kang and A. Vasseur,
	\newblock From Navier-Stokes to BV solutions of the barotropic Euler equations. 
	\newblock Preprint 2024. arXiv:2401.09305.
	
	\bibitem{Chen-Per}
	\newblock G.-Q. Chen and M. Perepelitsa,
	\newblock Vanishing viscosity limit of the Navier-Stokes equations to the Euler equations for compressible fluid flow.
	\newblock{\em Comm. Pure Appl. Math.} 63 (2010), no. 11, 14691504.
	
	\bibitem{Clarke}
	\newblock F. H. Clarke, 
	\newblock On the inverse function theorem.
	\newblock{\em Pacific J. Math.} 64 (1976), no.1, 97-102.
	
	\bibitem{Friedman}
	\newblock A. Friedman,
	\newblock Partial differential equations of parabolic type.
	\newblock Prentice-Hall, Inc., Englewood Cliffs, NJ, 1964. xiv+347 pp.
	
	\bibitem{Glimm}
	\newblock J. Glimm, 
	\newblock Solutions in the large for nonlinear hyperbolic systems of equations. 
	\newblock{\em  Comm. Pure Appl. Math.} 18 (1965), 697--715.
	
	\bibitem{GX}
	\newblock J. Goodman and Z. Xin,
	\newblock Viscous limits for piecewise smooth solutions to systems of conservation laws.
	\newblock {\em Arch. Rational Mech. Anal.} 121 (1992), no.3, 235--265.
	
	\bibitem{HJ}
	\newblock B. Haspot and A. Jana, 
	\newblock Vanishing viscosity limit for hyperbolic system of Temple class in 1-d with nonlinear viscosity.
	\newblock arXiv preprint arXiv:2407.12766
	
	\bibitem{KMR}
	\newblock K. H. Karlsen, S. Mishra and N. H. Risebro,
	\newblock Convergence of finite volume schemes for triangular systems of conservation laws.
	\newblock {\em Numer. Math.} 111 (2009), no. 4, 559--589.
	
	\bibitem{KMR-1}
	\newblock K. H. Karlsen, S. Mishra and N. H. Risebro,
	\newblock Semi-Godunov schemes for general triangular systems of conservation laws.
	\newblock {\em J. Engrg. Math.} 60 (2008), no. 3-4, 337--349.
	
	\bibitem{Serre-1}
	\newblock D. Serre,
	\newblock Solutions \`{a} variations born\'{e}es pour certains syst\`{e}mes hyperboliques de lois de conservation.(French.)
	\newblock {\em J. Differential Equations} 68 (1987), no.2, 137--168.
	
	\bibitem{Spinolo}
	\newblock L. V. Spinolo,
	\newblock Vanishing viscosity solutions of a $2\times 2$ triangular hyperbolic system with Dirichlet conditions on two boundaries.
	\newblock {\em Indiana Univ. Math. J.} 56 (2007), no. 1, 279--364.
	
	\bibitem{V}
	\newblock A. Vanderbauwhede,
	\newblock Centre manifolds, normal forms and elementary bifurcations.
	\newblock {\em Dynamics reported}, Vol. 2, 89--169. Dynam. Report. Ser. Dynam. Systems Appl., 2 John Wiley \& Sons, Ltd., Chichester, 1989.
\end{thebibliography}
\end{document}